\documentclass[12pt,a4paper]{amsart}
\usepackage{graphicx,latexsym,amsfonts,amsmath,amssymb,rotating,txfonts,mathrsfs,enumerate, mathtools}
\usepackage{epic}
\usepackage{curves}
\usepackage{pdfsync}
\usepackage{ytableau}
\usepackage{tikz}
\usepackage{calrsfs}
\usepackage[utf8]{inputenc}

\input xy
\xyoption{all}

\newcommand{\Spec}{{\rm Spec}}

\newcommand{\Hom}{ \,{\rm Hom} \,}
\newcommand{\Sym}{ \,{\rm Sym} \,}

\newcommand{\im}{ \,{\rm Im} \,}

{\bf}{\rm}
\newtheorem{theorem}{Theorem}[section]
\newtheorem*{theorem*}{Theorem}
\newtheorem{proposition}[theorem]{Proposition}
\newtheorem{corollary}[theorem]{Corollary}
\newtheorem{lemma}[theorem]{Lemma}
\newtheorem{definition}[theorem]{Definition}
\newtheorem{remark}[theorem]{Remark}
\newtheorem{conjecture}[theorem]{Conjecture}
\newtheorem{example}[theorem]{Example}

{\bf}{\it}

\newcommand{\CC}{{\mathbb C }}

\newcommand{\PP}{ {\mathbb P }}
\newcommand{\QQ}{{\mathbb Q }}

\newcommand{\ff}{{\mathbf f }}

\newcommand{\calc}{\mathcal{C}}
\newcommand{\cald}{\mathcal{D}}

\newcommand{\calo}{\mathcal{O}}

\newcommand{\cale}{\mathcal{E}}

\newcommand{\cals}{\mathcal{S}}

\newcommand{\tg}{\tilde{g}}

\newcommand{\reg}{\mathrm{reg}}
\newcommand{\codim}{\mathrm{codim}}

\newcommand{\Euler}{\mathrm{Euler}}

\newcommand{\mapd}[2]{J_k(#1,#2)}

\newcommand{\Tp}{\mathrm{Tp}}
\newcommand{\bbb}{\mathbf{b}}

\newcommand{\epd}[1]{\mathrm{eP}[#1]}

\newcommand{\mdeg}[1]{\mathrm{mdeg}[#1]}
\newcommand{\emu}{\mathrm{emult}}

\newcommand{\dist}{{\mathrm{dst}}}
\newcommand{\lead}{\mathrm{lead}}
\newcommand{\Span}{\mathrm{Span}}
\newcommand{\supp}{\mathrm{supp}}
\newcommand{\Curv}{\mathrm{Curv}}

\newcommand{\res}{\operatornamewithlimits{Res}}

\newcommand{\ires}{\res_{z_1=\infty}\res_{z_{2}=\infty}\dots\res_{z_k=\infty}}
\newcommand{\iresd}{\res_{z_1=\infty}\res_{z_{2}=\infty}\dots\res_{z_d=\infty}}
\newcommand{\sires}{\res_{\mathbf{z}=\infty}}

\newcommand{\dbz}{\,d\mathbf{z}}

\newcommand{\HC}{\mathrm{HC}}

\newcommand{\Thom}{\mathrm{Thom}}

\newcommand{\symdot}{\mathrm{Sym}^{\le k}\CC^n}

\newcommand{\grass}{\mathrm{Grass}}

\newcommand{\flag}{\mathrm{Flag}}
\newcommand{\diff}{\mathrm{Diff}}
\newcommand{\TT}{\mathrm{T}}

\newcommand{\Lt}{\tilde{\Lambda}}

\newcommand{\bz}{\mathbf{z}}

\newcommand{\bx}{\mathbf{x}}
\newcommand{\tx}{\tilde{x}}
\newcommand{\ta}{\tilde{a}}

\newcommand{\baa}{\mathbf{a}}

\newcommand{\bff}{\mathbf{f}}

\newcommand{\bU}{\mathbf{U}}
\newcommand{\bW}{\mathbf{W}}

\newcommand{\kt}{{K}}

\newcommand{\OO}{\mathcal{O}}

\newcommand{\jetreg}[2]{J_{k}^{\mathrm{reg}}({#1},{#2})}

\newcommand{\jetnondeg}[2]{J_{k}^{\mathrm{nondeg}}({#1},{#2})}

\newcommand{\tc}{\hat T}
\newcommand{\hN}{\hat{N}}

\newcommand{\GL}{\mathrm{GL}}
\newcommand{\sym}{\mathrm{Sym}}

\newcommand{\Hilb}{\mathrm{Hilb}}
\newcommand{\CHilb}{\mathrm{CHilb}}
\newcommand{\GHilb}{\mathrm{GHilb}}
\newcommand{\RHilb}{\mathrm{RHilb}}
\newcommand{\BHilb}{\mathrm{BHilb}}

\newcommand{\N}{\mathrm{N}}
\newcommand{\BN}{\mathrm{BN}}
\newcommand{\CN}{\mathrm{CN}}
\newcommand{\hch}{\mathrm{HC}}
\newcommand{\hcpt}{H^*_{\mathrm{cpt}}}
\newcommand{\C}{\mathbb{C}}

\def\a{\alpha}
\def\b{\beta}
\def\g{\gamma}

\def\l{\lambda}

\def\s{\sigma}

\def\L{\Lambda}

\title{Tautological integrals on Hilbert scheme of points I} 

\author{Gergely B\'erczi}
\address{Department of Xathematics, Aarhus University}
\email{gergely.berczi@math.au.dk}

\date{}
\begin{document}

\begin{abstract}
We develop a new method to study intersection theory of the main component of the Hilbert scheme of points on complex manifolds. The main result is an iterated residue formula for tautological integrals. We formulate a Chern-Segre-type positivity conjecture for these integrals, and we introduce higher dimensional Segre and Chern integrals.
\end{abstract}

\maketitle
{\scriptsize
\tableofcontents}

\section{Introduction}\label{sec:intro}
The Hilbert scheme of $k$ points on a smooth complex variety $X$, denoted by $\Hilb^k(X)$, is formed by all length-k subschemes of $X$.  
On a surface $S$ the Hilbert scheme $\Hilb^k(S)$ is a nonsingular irreducible variety of dimension $2k$. It has been extensively studied in the last 50 years and it occupies a central position in many important branches of mathematics, see e.g \cite{nakajima}.  While Hilbert schemes on surfaces are well-understood and broadly studied, the Hilbert scheme of points on manifolds of dimension three or higher are extremely wild objects with several unknown irreducible components and bad singularities \cite{vakil}, and our understanding of their geometry is very limited \cite{joachimthesis,joachim}.
The main (also called geometric) component 
\[\GHilb^k(X)=\overline{\{\xi \in \Hilb^k(X): \xi=p_1 \sqcup \ldots \sqcup p_k: p_i \neq p_j\}}\]
is the closure of the locus of non-reduced subschemes. This is a singular component of dimension $k \cdot \dim(X)$, and can be considered as a canonical compactification of the configurations space of $k$ distinct points on $X$. It is the most  interesting and relevant part of $\Hilb^k(X)$ from an enumerative geometry viewpoint, because several classical enumerative geometry problems can be reformulated using its topology and intersection theory, see \cite{nakajima,berczitau3}. 

On surfaces $\Hilb^k(S)=\GHilb^k(S)$, and the cohomological intersection theory of $\Hilb^k(S)$ can be approached from several different directions: 1) via the inductive recursions set up in \cite{egl, esa,esa2}; 2) using Nakajima calculus \cite{nakajima,groj,lehn,gottsche2} or more recently 3) virtual localisation on Quot schemes \cite{GP, mop2}. Lehn's conjecture \cite{lehn} on top Segre numbers of tautological line bundles, and its recent extension, the Segre-Verlinde duality conjectures \cite{mop1,johnson,gottschekool,mellitgottsche}, incapsulate the complexity of this theory. However, these techniques fail at a fundamental level in higher dimensions. 

The present paper accompanied with \cite{berczitau3} develops a new approach to calculate tautological integrals over the main component of the Hilbert scheme in higher dimensions. Let $X$ be a smooth complex variety of dimension $n$ and let $V$ be a rank $r$ algebraic vector bundle on $X$. Let $V^{[k]}$ be the corresponding tautological rank $rk$ bundle on $\Hilb^k(X)$, whose fibre at $\xi \in \Hilb^k(X)$ is $H^0(\xi,V|_\xi)$. Tautological integrals refer to Chern polynomials of $V^{[k]}$. The result of this paper can be interpreted as a generating function of tautological integrals in the dimension of $X$, not in the number of points $k$; this is a previously unseen feature of integrals.

The key idea is to reduce integration over $\GHilb^k(X)$ to a small punctual component, the so-called curvilinear component $\CHilb^k(X)$, where \cite{bercziG&T} provides integral formula. The punctual curvilinear locus at $p\in X$ is the set of curvilinear subschemes supported at $p$: 
\[\mathrm{Curv}^k_p(X)=\{\xi \in \Hilb^k_p(X): \xi \subset \mathcal{C}_p \text{ for some smooth curve } \mathcal{C} \subset X\}
=\{\xi: \calo_\xi \simeq \CC[z]/z^{k}\}\]

We call the closure $\CHilb^k_p(X)=\overline{\mathrm{Curv}^{k}_p(X)}$ the punctual curvilinear component supported at $p$, and $\CHilb^k(X)=\cup_{p\in X}\CHilb_p^k(X)$. We note that $\CHilb_p^k(X)$ is an irreducible component of the punctual Hilbert scheme $\Hilb^k_p(X)$ of dimension $(k-1)(n-1)$, and conjecturally, this is the smallest component \cite{joachimthesis}. On a smooth surface $S$, $\CHilb^k_p(S)=\Hilb^k_p(S)$ according to \cite{briancon}.

 Our strategy to overcome the incapsulated complexity of deformation theory is the following.
\begin{enumerate}
\item We introduce a homotopy trick involving stable holomorphic maps $f: \CC^n \to \CC^n$, which reduces the support of our tautological integrand form. The reduced support is locally irreducible at the the punctual points and its punctual part sits in the curvilinear component $\CHilb^k(X)$.
\item We define the fully nested Hilbert scheme $\N^k(X)$ and use a sieve method from \cite{rennemo} to reduce integration over the various diagonals of $\N^k(X)$. The deepest term is an integral over the punctual part $\N^k_0(X)$, and this let us use $X=\CC^n$ and equivariant integration over $\N^k_0(\CC^n)$. The problem lies in the complexity of $\N^k_0(\CC^n)$, which sits over $\GHilb^k_0(\CC^n)$, and has several wild, big components.
\item We define a partial blow-up of $\N^k(\CC^n)$ which fibers over the flag manifold $\flag_k(\CC^n)$, and using the symmetries of the flag manifold, we transform equivariant localisation into an iterated residue. A key technical part of this paper is a residue vanishing theorem on the fully nested Hilbert scheme: this shows that the fixed point contribution of those components of $\N^k(\CC^n)$ which do not map dominantly to $\CHilb^k(\CC^n)$ is zero.  As a result, we can reduce integration to a neighborhood of $\CHilb^k(\CC^n)$.  
\item We show that the support of our reduced tautological form at the points of $\CHilb^k(\CC^n)$ is isomorphic to the total space of the Haiman bundle $B$, and Thom isomorphism reduces integration to $\CHilb^k(\CC^n)$.
 \end{enumerate}

To explain our main formula, let $\Pi(k)$ denote the set of partitions of $\{1,\ldots, k\}$ into nonempty subsets; an element $\a=\{\a_1,\ldots, \a_s\} \in \Pi(k)$ consists of subsets $\a_i \subset \{1,\ldots, k\}$ such that $\a_i\cap \a_j=\emptyset$ for $1\le i<j\le s$ and $\{1,\ldots, k\}=\cup_{i=1}^s \a_i$. We introduce a set of variables 
\[\bz^{\a_i}=\{z^{\a_i}_1,\ldots, z^{\a_i}_{|\a_i|-1}\}\]
for each element of the partition, where $|\a|$ stands for the number of elements in $\a$.

\begin{theorem}[\textbf{Integrals over $\GHilb^{k}(X)$}]\label{mainthm}
Let $V$ be a rank $r$ vector bundle over the smooth projective variety $X$ of dimension $n$, with Chern roots $\theta_1, \ldots, \theta_r$. Let $\Phi$ be a Chern polynomial of degree $nk$. Then  
\[\int_{\GHilb^{k}(X)}\Phi(V^{[k]}) =\sum_{(\a_1,\ldots, \a_s) \in \Pi(k)} \int_{X^s} \res_{\bz^{\a_1}=\infty}\ldots \res_{\bz^{\a_s}=\infty} \mathcal{R}^\a(\theta_i, \bz) d\bz^{\a_1}\ldots d\bz^{\a_s}\]
where $\mathcal{R}^\a(\theta_i, \bz)$ stands for the rational form 
\[\Phi(V(\bz^{\a_1}) \oplus \ldots \oplus V(\bz^{\a_s})) 
\prod_{l=1}^s \left(\frac{(-1)^{|\a_l|-1} \prod_{1\le i<j \le |\a_l|-1}(z_i^{\a_l}-z_j^{\a_l})Q_{|\a_l|-1}d\bz^{\a_l}}
{\prod_{i+j\le q\le 
|\a_l|-1}(z_i^{\a_l}+z_j^{\a_l}-z_q^{\a_l})(z_1^{\a_l}\ldots z_{|\a_l|-1}^{\a_l})^{n+1}}\prod_{i=1}^{|\a_l|-1} 
s_X\left(\frac{1}{z_i^{\a_l}}\right)\right).\]
Here we use the following notations:
\begin{itemize}
	\item $\res_{\bz^{\a_l}=\infty}=\res_{z^{\a_l}_1=\infty} \ldots \res_{z^{\a_l}_{|\a_l|-1}=\infty}$ is the iterated residue at infinity.
	\item $s_X\left(\frac{1}{z}\right)$ is the Segre series of $X$ and $Q_k$ are the universal Borel-multidegrees explained in \cite{bsz}.
	\item $V(z)$ stands for the bundle $V$ tensored by the line $\mathbb{C}_z$, which is the representation of a torus $T$ with weight $z$. Hence its Chern roots are $z+\theta_1,\ldots ,z+\theta_r$. Moreover 
	\[V(\bz^{\a_l})=V \oplus V(z^{\a_l}_1) \oplus \ldots \oplus V(z^{\a_l}_{|\a_l|-1})\]
	has rank $r|\a_l|$, and 
	\begin{equation}\label{residueeuler} \Phi(V(\bz^{\a_1}) \oplus \ldots \oplus V(\bz^{\a_s})) =  \Phi(\theta_i^1,z^{\a_1}_1+\theta_i,\ldots z^{\a_1}_{|\a_1|-1}+\theta_i, \ldots, \theta_i^s,z^{\a_r}_1+\theta_i,\ldots z^{\a_r}_{|\a_r|-1}+\theta_i)
	\end{equation}
	Note that we have $s$ copies of the roots $\theta_1,\ldots , \theta_r$ of $V$, and we think of the $i$th copy $\theta_1^i,\ldots , \theta_r^i$ as the Chern roots of $V=V^i$ sitting over the $i$th copy of $X$ in $X^s$.
	\item The iterated residue is a homogeneous symmetric polynomial of degree $ms$ in the Chern roots $\theta^i_j$, that is, the Chern roots of $V^1 \oplus \ldots \oplus V^s$ over $X^s$, and integration over $X^s$ evaluates the homogeneous $(m,m,\ldots,m)$ part with at fundamental class of $X^s$.   
\end{itemize}

\end{theorem}

Let us add more explanation to the residue formula. First, if we set the degree of every variable $z_i^{\a_l}$ and $\theta_i^j$ to be $1$, then the total degree of the rational expression $\mathcal{R}^\a(\theta_i, \bz)$ is $nk-(n+1)(k-s)=(n+1)s-k$.
The iterated residue has $k-s$ variables, and hence 
\[\res_{\bz^{\a_1}=\infty}\ldots \res_{\bz^{\a_s}=\infty} \mathcal{R}^\a(\theta_i, \bz) d\bz^{\a_s}\ldots d\bz^{\a_1}\]
is a symmetric polynomial in the Chern roots $\theta_1^i,\ldots, \theta_r^i$ of $V^i$ for $1\le i \le s$, of degree 
\[(n+1)s-k+(k-s)=ns=\dim(X^s)\]
Here $V^i$ is the pull-back of $V$ from the $i$th factor in $X^s$, and integration over $X^s$ gives the result.  
This shows that the dependence on Chern classes of $X$ in fact can be expressed via the Segre classes of $X$. 
For fixed $k$ the rational expression in $\mathcal{R}^\a$ in the formula is independent of the dimension $n$ of $X$, and the iterated residue depends on $n$ only through the total Segre class $s_X$ of $X$. The iterated residue is then some linear combination of the coefficients of (the expansion of) $\mathcal{R}^\a$'s multiplied by Segre classes of $X$. By increasing the dimension, the iterated residue involves new terms of the expansion of $\mathcal{R}^\a$'s, and we can think of the our residue formula as a generating function of tautological integrals for fixed $k$ but varying $n$. 

A very special case of the integral formula appeared in \cite{bsz2}, where we studied equivariant push-forwards along maps $f: X \to Y$ of complex manifolds, and we only had to deal with the class $\Phi=c_d((f^*TY)^{[k]})$ where $k,n,d,r$ satisfy some very restrictive numerical conditions. To lift these restrictions, in this paper we develop a new approach, and in particular, reduce the support of the integrand to the curvilinear locus, and prove a residue vanishing theorem on the fully nested Hilbert scheme.

An old conjecture of Rim\'anyi \cite{rimanyi} says that Thom polynomials of Morin singularities have positive Chern-expansion. We show that this conjecture can be translated into the following surprising and mysterious conjecture.

\begin{conjecture}[Chern-Segre positivity conjecture]
The tautological integral in Theorem \ref{mainthm} has positive coefficients when written in Chern classes of $F$ and Segre classes of $X$. 
\end{conjecture}

Finally, in the last section, we introduce higher dimensional Segre and Chern integrals, and prove a product formula for these involving iterated residues. 

In the second part \cite{berczitau3} of this work we will extend the theory for so-called geometric subsets of $\GHilb^k(X)$. These arise when we fix the topological type (local algebra) of our subscheme $\xi=\xi_1 \sqcup \ldots \sqcup \xi_s \in \Hilb^k(X)$ at every point of the support, and we take the closure in the Hilbert scheme. Tautological integrals over geometric subsets have important applications in enumerative geometry. We devote the third part \cite{berczitau4} to study some of these: we develop formulas for counts of projective hypersurfaces with given set of singularities in sufficiently ample linear systems, and develop higher dimensional Nakajima calculus. 

\textbf{Acknowledgments}. The author gratefully acknowledges useful discussions 
with Andr\'as Szenes. This research was supported by Aarhus University Research Foundation grant AUFF 29289.

\section{Overview of the strategy}\label{sec:strategy} 
This section is an overview of our strategy, where we explain the key ideas and steps of the argument but leave the details for subsequent sections. 

Let $X$ be a smooth complex manifold of dimension $\dim(X)=n$. For $n>2$ and $k\ge 7$ the Hilbert scheme $\Hilb^k(X)$ is singular and not irreducible. We will work with the ordered Hilbert scheme $\Hilb^{[k]}(X)$ which is a branched cover of the ordinary Hilbert scheme (see \S \ref{sec:nested}) with an associated ordered Hilbert-Chow morphism $\HC: \Hilb^{[k]}(X) \to X^{\times k}$. The geometric (main) component of $\Hilb^{[k]}(X)$ is
\[\GHilb^k(X)=\overline{\{x_1 \sqcup \ldots \sqcup x_k: x_i \neq x_j\}} \subset \Hilb^{[k]}(X)\]
is the closure of the locus of reduced subschemes (a reduced point is a $k$-tuple of different points on $X$). The main components is singular of dimension $nk$. For $p\in X$ we define the punctual geometric part supported at $p$ as 
\[\GHilb^k_p(X)=\Hilb_p^{[k]}(X) \cap \GHilb^k(X)\]
On surfaces this is irreducible and equal to the full punctual Hilbert scheme.  However, for $\dim(X)\ge 3$ and $k>11$ $\GHilb^k_p(X)$ is not irreducible; its components are called the \textit{smoothable components} of $\Hilb_p^{k}(X)$. There is a distinguished smoothable component, called the curvilinear component $\CHilb^k_p(X)$, whose dimension is $(n-1)(k-1)$, but there are other, higher dimensional smoothable components. Description of the smoothable components is out of reach, the component structure, deformation theory and singularities of the Hilbert scheme is widely open problem. Surprisingly, the following is true: 
\[\text{Xain idea: Integration on } \GHilb^k(X) \text{ can be reduced to integration over } \CHilb^k(X).\]
This combined with the residue formula of \cite{bercziG&T} for integrals over the curvilinear component will give the desired formula.  

\noindent \textbf{Step 1: The stability trick} Let $V$ be a bundle over the complex manifold $X$ and $\Phi(V^{[k]})$ a Chern polynomial of the tautological bundle. After Step 3 below, we can in fact assume that $X=\CC^n$. Let $f: X\to X$ be a stable Thom-Boardman map, see \S\ref{sec:stablemaps} for definitions. Stable maps are dense in the space of holomorphic maps, and we will pick one in the homotopy class of the identity if that exists, otherwise we approximate the identity map with stable maps. In \S \ref{sec:prooftauintegrals} we will deform the integrand $\Phi(V^{[k]})$ to $\Phi(V^{[k]})_f$ such that 
\[\int_{\GHilb^k(X)}\Phi(V^{[k]})=\int_{\GHilb^k(X)} \Phi(V^{[k]})_f
\]
but $\Phi(V^{[k]})_f$ is represented by a form whose support has better geometric properties: it is locally irreducible and curvilinear, as follows. We define the $f$-Hilbert scheme 
\[\GHilb^k(f)=\overline{\{\xi=\xi_1 \sqcup \ldots \sqcup \xi_k \in \GHilb^k(X): f(\xi_1)=\ldots =f(\xi_s) \in X\}}\]
as the set of subschemes supported on the fibers of $f$ 
In \cite{bsz2} we prove that 
\[\mathrm{supp}(\Phi(V^{[k]})_f) \subset \GHilb^k(f).\]
In Proposition \ref{prop:hilbfloc} we prove that for stable Thom-Boardman map $f$ 
\[\GHilb^k_0(f)=\GHilb^k(f) \cap \GHilb^k_0(X) \subset \CHilb^k(X),\]
that is, the punctual $k$-fold locus of a stable map sits in the curvilinear component, hence 
\begin{equation}\label{support1}
\mathrm{supp}(\Phi_f)\cap \GHilb^k_0(X) \subseteq \CHilb^k(X)
\end{equation}
Moreover, we show that a small neighborhood $\GHilb^k_\nabla(f)$ of $\GHilb^k_0(f)$ in $\GHilb^k(f)$ is given by a tautological (Haiman) bundle $B$:
\begin{equation}\label{localmodel}
\xymatrix{\GHilb^k_\nabla(f) \ar[d]  \ar[r] &  B  \ar[d]  \\
  \GHilb^k_0(f) \ar@{^{(}->}[r]^{\iota} &  \CHilb^k(X)}
 \end{equation}
with a  topological isomorphism $\GHilb^k_\nabla(f) \simeq \iota^*(B)$.  

\noindent \textbf{Step 2: The Sieve } We introduce a master blow-up $\pi_\Lambda: \N^k(X) \to \GHilb^k(X)$ which we call the fully-nested Hilbert scheme. It admits blow-up morphisms $\pi_\a: \N^k(X) \to \GHilb^\a(X)$ to approximating Hilbert spaces $\GHilb^\a(X)$ defined for partitions $\a=\a_1 \sqcup \ldots \sqcup \a_r =\{1,\ldots, k\}$ of the $k$ points into smaller groups. For a bundle $V$ over $X$ this defines approximating tautological bundles $V^\a$ over $\GHilb^\a(X)$, and the pull-backs $\pi_\a^*V^\a$ of tautological bundles from the different approximating sets provide necessary abundance of bundles over $N^k(X)$ to define a sieve formula, following ideas of Li and Rennemo. This sieve gives $K$-theoretic decomposition of $V^{[k]}$ as a sum 
\[V^{[k]}=\sum_{\a \in \Pi(k)}W^\a\]
of bundles indexed by partitions of $\{1, \ldots k\}$. For any partition $\{1,\ldots, k\}=\a_1 \sqcup \ldots \sqcup \a_s$ the bundle $W^\a$ is supported on the approximating punctual subset  
\[\supp(W^\a)=N^{\a}_0(X)=\HC^{-1}(\Delta_{\a})\]
and $W^\a$ is a linear combination of the classes $V^{[\b]}$ for those partitions $\b \in \Pi(k)$ which are refinements of $\a$. In particular, for $\a=\Lambda=\{1,\ldots, k\}$ the trivial partition
\begin{equation}\label{deepestterm}
W^\Lambda=\pi_\Lambda^* V^{[k]}+\sum_{\a \in \Pi(k)\setminus \Lambda} \gamma_\a \pi_\a^*V^\a   
\end{equation}
is a linear combination of tautological classes pulled-back from the tautological approximating classes. Let $\Phi$ be a Chern polynomial in the Chern roots of $V^{[k]}$. Then the sieve reduces the integration over $\GHilb^k(X)$ to $N^k_\nabla(X)$:
\begin{equation}\label{step1} 
\xymatrixcolsep{5pc} \xymatrix{\int_{\GHilb^k(X)}\Phi(V^{[k]}) \ar@^{~>}[r]^-{\mathrm{Sieve}} & \int_{N^\Lambda_\nabla(X)}\Phi(W^\Lambda)+ \sum_{\a \in \Pi_k \setminus \Lambda} \int_{N^{\alpha}_{\nabla}(X)}\Phi(W^{\alpha})}
\end{equation}
where the terms in the sum are determined inductively, and $N^\a_\nabla(X)$ stands for a small, tubular neighborhood of the $\a$-punctual part $N^\a_0(X)$. This leaves us with calculating the leading first term.

\noindent \textbf{Step 3: From $X$ to $X=\CC^n$.} Since $\Phi(W^\Lambda)$ is compactly supported in a neighborhood $N^k_\nabla(X)$ of the punctual part $N^k_0(X)$, we can work locally, and pull-back our local neighborhood $N^k_\nabla(X)$ from a universal bundle. The Chern-Weil map then reduces push-forward (integration) to equivariant push-forward (integration) over $X=\CC^n$. The punctual nested Hilbert scheme $N^k_0(\CC^n)$ is not irreducible, but $\Phi(W^\Lambda)$ is supported in a tubular neighborhood which is the union of tubular neighborhoods of its components. 
\begin{equation}\label{step2} 
\xymatrixcolsep{5pc} \xymatrix{\int_{N^\Lambda_\nabla(X)}\Phi(W^{\Lambda}) \ar@^{~>}[r]^-{\mathrm{Chern-Weil}} & \int^T_{N^{\Lambda}_{\nabla}(\CC^n)}\Phi(W^{\Lambda})},
\end{equation}
and we will apply equivariant integration $\int^T$ using the induced action of the maximal torus $T\subset \GL(m,\CC)$.

\noindent \textbf{Step 4: Residue Vanishing Theorem on the nested Hilbert scheme} The next crucial step in the strategy is the construction of a blow-up $\hat{N}^k_\nabla(\CC^n) \to N^k_\nabla(\CC^n)$ and $\widehat{\GHilb}^k_\nabla(\CC^n) \to \GHilb^k_\nabla(\CC^n)$ such that $\hat{N}^k_\nabla(\CC^n)$ fibers over the flag $\flag_{k-1}(T\CC^n)$ of $(1,2,\ldots, k-1)$ dimensional subspaces of the tangent bundle, which itself fibers over $\CC^n$:
\[\hat{N}^k_\nabla(\CC^n) \to \widehat{\GHilb}^k_\nabla \to \flag_{k-1}(T\CC^n)\to \CC^n.\]
 The fiber over the origin is the balanced Hilbert scheme $\BHilb^{k}_0(\CC^n)$ consisiting of subschemes with baricenter at the origin. This fibration introduces symmetry which allows us to transform the Atiyah-Bott equivariant localisation over $\BHilb^k_0(\CC^n)$ into an iterated residue. The symmetry arranges fixed points into cycles and leads to residue vanishing properties. Our first residue vanishing theorem tells that the contribution of the terms $\Phi(V^\a)$ in \eqref{deepestterm} to the push-forward (integral) formula is zero for $\a \neq \Lambda$. But the leading term is a pull-back from the Hilbert scheme, hence 
\begin{equation}\label{step3} 
\xymatrixcolsep{5pc} \xymatrix{ \int^T_{N^{\Lambda}_{\nabla}(\CC^n)}\Phi(W^{\Lambda}) \ar@^{~>}[r]^-{\mathrm{\text{Sieve Res. Vanishing }}} & \int^T_{\GHilb^{k}_{\nabla}(\CC^n)}\Phi(V^{\Lambda})}
\end{equation}
where $\Hilb^{k}_{\nabla}(\CC^n)$ is a neighborhood of the punctual Hilbert scheme on $\CC^n$. In short, the sieve and the first residue vanishing theorem allows us to go back from the master blow-up space to the Hilbert scheme and use equivariant techniques over $X=\CC^n$.
By Step 1, we can assume that the support of the integrand is locally irreducible and curvilinearly supported, hence the neighborhood of the punctual Hilbert scheme can be replaced by the neighborhood of the curvilinear component:
\begin{equation}\label{step3b} 
\int^T_{\GHilb^{k}_{\nabla}(\CC^n)}\Phi(V^{\Lambda})=\int^T_{\CHilb^{k}_{\nabla}(\CC^n)}\Phi(V^{\Lambda})
\end{equation}
We note that in \cite{bsz2} we also prove a Residue Vanishing Theorem, but the argument in \cite{bsz2} only works for very special tautological classes, and the Sieve Residue Vanishing theorem is of slightly different nature. 

By the stability trick we can assume $\Phi(V^{[k]})$ is supported on $\GHilb^k_\nabla(f)$, hence the local model \eqref{localmodel} combined with Thom isomorphism gives  
\[\xymatrixcolsep{5pc} \xymatrix{\int_{\CHilb^k_\nabla(\CC^n)}\Phi(V^{[k]}) \ar@{~>}[r]^-{\text{Local\ model}} &  \int_{\CHilb^k(\CC^n)}\frac{\Phi(V^{[k]})}{\Euler(B)}}\]

\noindent \textbf{Step 5: Residue Vanishing Theorem on the curvilinear Hilbert scheme} The problem is now reduced to integration over the curvilinear component $\CHilb^k(\CC^n)$. We adapt and generalise the formula developed in \cite{bercziG&T,bsz} to perform this integration. The key step is to prove a new residue vanishing theorem, which is stronger than \cite{bsz,bsz2}. This tells us that only one fixed point, the Porteous point, has nonzero contribution in the residue formula  This allows us to ignore the complexity of the singular spaces involved , and reduces it to understand the geometry of a distinguished fixed point 
\[\xymatrixcolsep{5pc} \xymatrix{\int_{\CHilb^k(\CC^n)}\frac{\Phi(V^{[k]})}{\Euler(B)} \ar@{~>}[r]^-{\text{Strong Res Vanishing}} & \text{Integral formula for } n \ge k}\]
This residue formula gives the deepest term in the sieve, and the other terms follow by induction. 

\noindent \textbf{Step 6: Removing the $n\ge k$ condition} We finally apply the trick in \cite{bercziG&T} to derive the integral formula for any $k$ on $\CC^n$. This uses the non-reductive quotient model (i.e jet curve model) of the curvilinear Hilbert scheme in \cite{bsz}.

\section{Hilbert scheme of points on smooth varieties}\label{sec:hilbertschemes}

\subsection{Components and punctual components}
Let $X$ be a smooth projective variety of dimension $n$ and let  
\[\Hilb^k(X)=\{\xi \subset X:\dim(\xi)=0,\mathrm{length}(\xi)=\dim H^0(\xi,\calo_\xi)=k\}\]
denote the Hilbert scheme of $k$ points on $X$ parametrizing all length-$k$ subschemes of $X$. For $p\in X$ let 
\[\Hilb^{k}_p(X)=\{ \xi \in \Hilb^k(X): \mathrm{supp}(\xi)=p\}\]
denote the punctual Hilbert scheme consisting of subschemes supported at $p$. Let 
\[\rho: \Hilb^k(X) \to \Sym^kX,\  \xi \mapsto \Sigma_{p\in X}\mathrm{length}(\calo_{\xi,p})p\]
denote the Hilbert-Chow morphism. Then $\Hilb^k(X)_p=\rho^{-1}(kp)$.

The Hilbert scheme of $k$ points on smooth surfaces is a $2k$ dimensional nonsingular variety \cite{fogarty}. It occupies a central position in many important branches of mathematics and mathematical physics. In particular, it plays a pivotal role in enumerative geometry,  see e.g \cite{nakajima}. The cohomological intersection theory of the Hilbert scheme of points on surfaces can be approached from several different directions: 1) via the inductive recursions set up in \cite{egl, esa,esa2}; 2) using Nakajima calculus \cite{nakajima,groj,lehn} or more recently 3) virtual localisation on Quot schemes \cite{virtual, mop2}. Lehn's conjecture \cite{lehn} on top Segre numbers of tautological line bundles incapsulates the complexity of this theory.

While Hilbert schemes on surfaces are well-understood and broadly studied, the Hilbert scheme of points on manifolds of dimension three or higher are extremely wild objects with several unknown irreducible components and bad singularities \cite{vakil}, and our undertanding of them is very limited. 

In enumerative geometry applications we are mostly interested in the geometric component (also called the main component, or smoothable component) of the Hilbert scheme $\Hilb^k(X)$. This component, which we denote by $\GHilb^k(X)$ is the closure of the open locus formed by the reduced points in $\Hilb^k(X)$, that is, those supported at $k$ distinct points on $X$. This is an irreducible, but highly singular component of dimension $km$. Ekedahl and Skjelnes \cite{ekedahl} constructs the geometric component as a certain blow-up along an ideal of $\Sym^k X$--this is the generalisation of the classical result of Haiman \cite{haiman} on surfaces. 

In this paper we develop iterated residue formulas for certain tautological integrals on the geometric component. The idea is to reduce integration to the punctual curvilinear component, which, in turn, is a projective compactification of a non-reductive quotient: the moduli of $k$-jets in $X$ up to polynomial reparametrisatons. 
A subscheme $\xi \in \Hilb^k_p(X)$ is called curvilinear if $\xi$ is contained in some smooth jet of a curve $\calc_p \subset X$ at $p$.  Equivalently, $\xi$ is curvilinear if $\calo_\xi$ is isomorphic  to the $\CC$-algebra $\CC[z]/z^{k}$.
The punctual curvilinear locus at $p\in X$ is the set of curvilinear subschemes supported at $p$: 
\begin{multline}\nonumber 
\mathrm{Curv}^k_p(X)=\{\xi \in \Hilb^k_p(X): \xi \subset \mathcal{C}_p \text{ for some smooth curve } \mathcal{C} \subset X\}=\\
=\{\xi \in \Hilb^k_p(X):\calo_\xi \simeq \CC[z]/z^{k}\}
\end{multline}

We call the closure $\CHilb^k_p(X)=\overline{\mathrm{Curv}^{k}_p(X)}$ the punctual curvilinear component supported at $p$. $\CHilb_p^k(X)$ is an irreducible component of the punctual Hilbert scheme $\Hilb^k_p(X)$ of dimension $(k-1)(n-1)$. On a smooth surface $S$, $\CHilb^k_p(S)=\Hilb^k_p(S)$ according to Brianchon. Moreover, 
\[\GHilb^k_p(S)=\Hilb_p^{[k]}(S) \cap \GHilb^k(S)\]
is irreducible and equal to the full punctual Hilbert scheme.  However, for $\dim(X)\ge 3$ $\GHilb^k_p(S)$ is typically not irreducible; its components are called the smoothable components of $\Hilb_p^{k}(X)$. The punctual curvilinear component $\CHilb^k_p(X)$ is a smoothable component. On the other hand, the  Iarrobino-type punctual components \cite{iarrobino} are not necessarily smoothable: their dimension can be higher than $km$. Description of the smoothable components is a hard problem and is unknown in general.  It was an open question until recently whether there exist divisorial components of $\GHilb^k_p(S)$, that is, components $C_p \subset R_p^{[k]}$ whose sweep $\cup_{p\in X} C_p$ over $X$ is a divisor of the good component. The dimension of this $C_p$ is necessarily $m(k-1)-1$.  Erman and Velasco in \cite{erman} give affirmative answer to this question when $d\ge 4, k\ge 11$. A beautiful illustration of the components and structure of $\Hilb^k(X)$ can be found in the PhD thesis of Jelisiejew \cite{joachimthesis}, called the 'Bellis Hilbertis'.

To sum up, when $\dim(X)\ge 3$, and the number of points $k$ is large enough, the punctual Hilbert scheme $\Hilb^k_p(X)$ is not necessarily irreducible or reduced, and it has smoothable and non-smoothable components. Among the smoothable components, there is a distinguished one, the curvilinear component, we will call all other smoothable components exotic components.



\section{Stable maps}\label{sec:stablemaps}

Let $J(n)=J(n,1)$ be the space of germs of holomorphic functions $(\CC^n,0) \to (\CC,0)$; in local coordinates $(x_1,\ldots,x_n)$ at the origin 
\[J(n,1)=\{h\in\CC[[x_1\ldots x_n]];\; h(0)=0\}\] 
is the algebra of power series without a constant term. Let $J_k(n,1)$ be the space of $k$-jets of holomorphic functions on $\CC^n$ near the origin, i.e. the quotient of $J(n,1)$ by the ideal of those
power series whose lowest order term is of degree at least $k+1$.

Our basic object is $J_k(n,m)$, the space of $k$-jets of holomorphic
maps $(\CC^n,0)\to(\CC^m,0)$. This is a finite-dimensional complex
vector space, which one can identify as $J_k(n,1) \otimes \CC^m$; hence
$\dim J_k(n,m) =m\binom{n+k}{k}-m$.  We will call the elements of
$J_k(n,m)$ as map-jets of order $k$, or simply map-jets. 

One can compose map-jets via substitution and elimination of terms of degree greater than $k$; this leads to the composition maps
\begin{equation}
  \label{comp}
J_k(n,m) \times J_k(m,p) \to J_k(n,p),\;\;  (\Psi_2,\Psi_1)\mapsto
\Psi_2\circ\Psi_1.
\end{equation}
The  set
\[
\diff_k(n) = \{\Delta\in \mapd nn;\; \mathrm{Lin}(\Delta) \text{ invertible}\}.
\]
is an algebraic group under the composition map \eqref{comp}, and this gives the 
so-called {\em left-right} action (also called $\mathcal{A}$-action) of the group $\diff_k(n) \times \diff_k(n)$ on $\mapd nm$:
\[ [(\Delta_L,\Delta_R),\Psi]
 \mapsto \Delta_L\circ\Psi\circ \Delta_R^{-1}
 \] 

Let $r$ be a nonnegative integer. An {\em unfolding} of a map-jet
$\Psi \in J_k(n,m)$ is a map-jet $\widehat{\Psi}\in \mapd{n+r}{m+r}$ of
the form
\[(x_1,\ldots ,x_n,y_1,\ldots, y_r)\mapsto
(F(x_1,\ldots, x_n,y_1,\ldots, y_r),y_1,\ldots, y_r)
\]
where $F\in \mapd {n+r}m$ satisfies
\[F(x_1\ldots, x_n,0,\ldots,
0)=\Psi(x_1,\ldots, x_n).\] The {\em trivial unfolding} is the
map-jet
\[(x_1,\ldots, x_n, y_1,\ldots, y_r)\to (\Psi(x_1,\ldots
,x_n),y_1, \ldots ,y_r).\]

\begin{definition}
  A map-jet $\Psi \in J_k(n,m)$ is {\em $k$-stable} if all unfoldings of
  $\Psi$ are left-right equivalent to the trivial unfolding. A holomorphic map $f:X \rightarrow Y$ of
complex manifolds is $k$-stable if the $k$-jet $f_p$ of its germ at $p$ is stable for all $p\in X$.
\end{definition}

Informally, the germ $f_p$ is stable if for any
small deformation $\tilde{f}$ of $f$, there is a point in the
vicinity of $p$ at which the germ of $\tilde{f}$ is left-right
equivalent to the germ of $f$ at $p$.

\begin{proposition}{\cite{arnold}} Stable maps $f: X\to X$ form an open dense subset in the space of holomorphic maps.
\end{proposition}

Next we recall the Thom-Boardman classification of maps, which was introduced by Thom \cite{thom} and Boardman \cite{boardman}  (see also \cite{arnold}). 
Let $I=(i_1\ge i_2 \ge \ldots \ge i_k)$ be a finite nonincreasing sequence of nonnegative integers. 
\begin{definition}[Thom, \cite{thom}]\label{thomboardmandef1} For a smooth map $f: X \to Y$, we define $\Sigma^I(f) \subset X$ inductively as follows. Let 
\[\Sigma^i(f)=\{x \in X : \dim(\ker(d_xf))=i\}\]
Suppose that $\Sigma^i(f) \subset X$ is a smooth submanifold and $0\le j\le i$; then we define $\Sigma^{ij}(f)$ to be $\Sigma^j(f|_\Sigma^i(f))$. Inductively, assume that $\Sigma^{i_1,\ldots, i_r}(f) \subset X$ is smooth and define
\[\Sigma^{i_1,\ldots, i_{r+1}}(f) = \Sigma^{i_{r+1}}((f|_{\Sigma^{i_1,\ldots, i_r}(f)})\]
\end{definition}

Boardman in \cite{boardman} proposed a different definition to avoid the condition on the smoothness of the sets $\Sigma^I(f)$. For $I=(i_1,\ldots, i_k)$ he defined a smooth, not necessarily closed submanifold $\Sigma^I \subset J_k(X,Y)$ of the $k$-jet space, independent of the choice $f$. 

\begin{definition} Let $B$ be an ideal in $\mathcal{E}_n=J(n)\oplus \CC$, the ring at the origin of germs of holomorphic functions on $\CC^n$. The $k$th Jacobian extension $\Delta_k(B)$ is the ideal generated by $B$ together with the $k \times k$ determinants $\det(\partial \varphi_i/\partial x_j)$ formed by partial derivatives of functions in $B$ in some fixed local coordinates at the origin on $\CC^n$. It will be convenient to introduce the relabeling $\Delta^k(B):=\Delta_{n-k+1}(B)$. Note that $\Delta^k(B)$ actually does not depend on the choice of these local coordinates.
\end{definition}
We get a chain of ideals
\[B = \Delta^0(B) \subseteq  \Delta^1(B) \subseteq \ldots  \subseteq \Delta^{n+1}(B) = \mathcal{E}_n\]
The largest $\Delta^k$ such that $\Delta^k(B) \neq \mathcal{E}_n$ is called the critical Jacobian extension. Note that the critical extension of an ideal $B$ is $\Delta^{n-r}=\Delta_{r+1}$, where $r=\dim_\CC(\mathfrak{n}^2+B/\mathfrak{n}^2)$.
\begin{definition}[Boardman, \cite{boardman}] Let $I=(i_1,\ldots, i_k)$ be a nonincreasing set of nonnegative integers. The map jet $f=(f_1,\ldots, f_m) \in J_k(\CC^n,\CC^m)$ belongs to $\Sigma^I$ (and said to have Boardman symbol $I$) if the ideal $B=(f_1,\ldots, f_m) \subset \mathcal{E}_n$ has successive critical extensions 
\[\Delta^{i_1}B,\ \  \Delta^{i_2}\Delta^{i_1}B,\ \  \Delta^{i_3}\Delta^{i_2}\Delta^{i_1}B ,\ldots \]
The Boardman class $\Sigma^I$ is a smooth submanifold of $J_k(n,m)$.
\end{definition}
\begin{proposition}[\cite{arnold}] Let $I=(i_1,\ldots, i_k)$ be a Boardman symbol. 
\begin{enumerate}
\item The codimension of the submanifold $\Sigma^I$ in $J_k(n,m)$ is given by the formula
\[\mathrm{codim}(\Sigma^I)=(m-n+i_1)\mu(i_1,\ldots, i_k)-(i_1-i_2)\mu(i_2,\ldots, i_k)-\ldots -(i_{k-1}-i_k)\mu(i_k)\]
where $\mu(i_1,\ldots, i_s)$ is the number of sequences $(j_1,\ldots, j_s)$ satisfying $j_1>0$, $j_1 \ge \ldots j_s$ and $i_s \ge j_s$.
\item The multiplicity of $f\in \Sigma^I(n,m)$ is
\[\dim(A_f)=i_1+\ldots +i_k+1\]
\end{enumerate}
\end{proposition}

\begin{definition}[Boardman, \cite{boardman}] For a map $f: X \to Y$ between manifolds and $x \in X$ let $f=(f_1, \ldots ,f_m)$ be the coordinate functions of $f$ in some local coordinate system. Then $x \in \Sigma^I(f)$ if $B=(f_1,\ldots, f_m) \in \Sigma^I$.
\end{definition} 

\begin{remark} Thom-Boardman singularities of order $k$ are $k$-determined, that is, it is enough to
look at the first $k$ differentials of a map to decide whether it belongs to the given Thom-
Boardman class (this is clear from Boardman's definition). They are also stable in
the sense that if $f:\CC^n \to \CC^m$ belongs to $\Sigma^I$,then so does $f \oplus id_\CC: \CC^{n+1} \to \CC^{m+1}$.
\end{remark}

Let $J_k(X,Y)=J_k(X) \otimes \CC^m \to X$ denote the $k$-jet bundle whose fiber over $x \in X$ is the set $J_k(n,m)$ of $k$-jets of germs at $x$. For a map $f: X \to Y$ its $k$-jet extension defines a section $s_f$ of $J_k(X,Y)$. 

\begin{theorem}[Boardman, \cite{boardman}] \begin{enumerate}
\item Assume that the $k$-jet extension $j^kf$ of the map $f: X  \to Y$ is transverse to the manifolds $\Sigma^I$ (we call these maps Thom-Boardman maps). Then 
\[\Sigma^I(f)=(j^kf)^{-1}(\Sigma^I)\]
for any $I$ of length $k$. 
\item Any smooth map $f:X \to Y$ can be arbitrarily well approximated , together with any number of its derivatives, by a Thom-Boardman map.
\end{enumerate}
\end{theorem}

In fact, we have the following stronger result, which follows from Thom's transversality theorem.
\begin{proposition}[(see \cite{wilson})] The set of Thom-Boardman maps is residual, that is, the intersection of countable number of open dense subsets in $J_k(X,Y)$. In this sense, a generic map is Thom-Boardman. 
\end{proposition}

\begin{remark}\label{remark:notstratification}
Thom-Boardman singularities do not give a stratification of the jet space $J_k(X,Y)$: they provide a partition of $X$ of a generic mapping into locally closed submanifolds $\Sigma^I(f)$, but the closure of a submanifold is not necessarily a union of similar submanifolds. Indeed, it can be shown (see Lander \cite{lander}) that $\Sigma^2 \cap \overline{\Sigma^{1,\ldots, 1}} \neq \emptyset$ for any number of $1$'s. But if this number of $1$'s is sufficiently large, then $\dim(\Sigma^{1,\ldots, 1})<\dim(\Sigma^2)$.
\end{remark}

\begin{remark} Existence of good approximation of smooth maps does not imply that Thom-Boardman maps form a dense open subset of the space $J_k(X,Y)$ of $k$-jets of maps from $X$ to $Y$. Thom's Transversality Theorem (see e.g. \cite{arnold})  implies that this set is residual. But Wilson \cite{wilson}] showed that the Thom-Boardman maps form an open subset of $J(X,Y)$, if and only if either $n\le m$ and $3n-4<2m$ or $n>m$ and $2m<n+4$. These numerical conditions hold precisely if $\codim(\Sigma^2)>n$, that is, there are no corank-$2$ singularities of $f$ (such maps are also called corank $1$ maps or Morin maps). Moreover, If these numerical conditions hold, then a map is Thom-Boardman if and only if its germ is stable at each point. 

\end{remark}

Recall from Remark \ref{remark:notstratification} that the Thom-Boardman classes $\Sigma^I$ do not give a stratification of the jet space $J_k(X,Y)$. But we still have the following structural statement about the Thom-Boardman classification. 

\begin{theorem}{(see \cite{bsz2})} \label{prop:boardman} 
 For a Boardman symbol $I=(i_1,\ldots, i_k)$ and a Thom-Boardman map $f: X \to Y$ we have 
\[\Sigma^{i_1+\ldots +i_k}(f) \subset \overline{\Sigma^{i_1,\ldots i_k}(f)}\]
\end{theorem}

\subsection{The $f$-Hilbert scheme for stable maps}\label{sec:proofmain}
We start with the introduction and analysis of the multipoint Hilbert scheme $\GHilb^k(f)$, which plays a central role in the proof. This is followed by local analysis of the $k$-fold locus in $X$ at its punctual locus. We then give the proof of the multipoint residue formula. 
 
Let $f: X\to Y$ be a map between smooth compact complex manifolds of dimension $\dim(X)=n, \dim(Y)=m$. This induces the $k$th Hilbert extension map
\[\mathrm{hf}^{[k]}: \Hilb^{[k]}(X) \to \Hilb^{[k]}(X \times Y)\]
which sends the subscheme $\xi_I$ to $\xi_{(I,I_{\Gamma(f)})}$ where $I_{\Gamma(f)}$ is the ideal of the graph $\Gamma(f)$ of $f$. Heuristically, the map $f^{[n]}$ pushes $\xi_I$ onto the graph along the $Y$ axis. $\mathrm{hf}^{[k]}$ is a regular embedding (see \cite{bsz2}). We introduce two subsets of $\Hilb^{[k]}(X)$ which play central role in our argument. 

\begin{definition} \begin{enumerate} \item The $f$-Hilbert scheme $\Hilb^k(f)$ is defined by the following pull-back diagram
\[\xymatrix{\Hilb^k(f) \ar@{^{(}->}[r] \ar@{^{(}->}[d]^-{j} & \Hilb^k(X) \times Y \ar@{^{(}->}[d]^-{j} \\ 
\Hilb^k(X) \ar@{^{(}->}[r]^-{\mathrm{hf}^{[k]}} & \Hilb^k(X \times Y)}\]
We use the same notation $f$ for the associated map 
\[f: \Hilb^k(f) \to Y \ \ \ \xi \mapsto f(\xi).\]
\item The geometric component $\GHilb^k(f) \subset \Hilb^{[k]}(f)$ is the closure of the open part consisting of $k$ points on a level set of $f$: 
\[\GHilb^k(f)=\overline{\{\xi=\xi_1 \sqcup \ldots \sqcup \xi_k \in \Hilb^{[k]}(X): f(\xi_1)=\ldots =f(\xi_s) \in Y\}}\subset \Hilb^k(f)\]
\end{enumerate}
\end{definition}
\begin{remark}Intuitively, both $\Hilb^k(f)$ and $\GHilb^k(f)$ contain subschemes of length $k$ in $X$ whose projection to the graph of $f$ is horizontal. However, not every such 'horizontal' subscheme can be approximated by $k$ different points on level sets. More precisely, a punctual subscheme $\xi \in \Hilb^k_p(f)$ whose support is $p\in X$ sits in $\GHilb^k_p(f)$ if and only if 
\[\xi =\lim_{i\to \infty} \xi_i \text{ where } \xi_i \subset f^{-1}(p_i) \text{ for some } p_i \to p\]
\end{remark}
In fact, the test curve model of Morin singularities gives the following description. 
\begin{proposition}\label{prop:hilbf}  Let $\Hilb^k_0(f)=\cup_{p\in X} \Hilb^k_p(f)$ denote the punctual part of $\Hilb^k(f)$, consisting of subschemes supported at {\it some} point on $X$, and let $\pi_X: \Hilb^k_0(f) \to X$ denote the projection. For the Morin algebra $A_{k}=k[t]/t^{k}$ of dimension $k$ let 
\[\Hilb_{A_k}(f)=\{\xi \in \Hilb^k_0(f): \calo_\xi \simeq A_k\}\]
Then 
\[\overline{\pi_X(\Hilb_{A_k}(f))}=\Sigma_{A_k}(f)\]
where we recall that $\Sigma_{A_k}(f)=\overline{\{p\in X: A_{f,p} \simeq A_k\}}$ is the Thom locus whose cohomology cycle gives the Thom polynomial $\Tp_k$. 
\end{proposition}
\begin{proof}
The test curve model for the map $f:X \to Y$ says that for a generic point $p\in \Sigma_{A_k}(f)$ there is a test curve $h: (\CC,0) \to (X,p)$ whose $k$-jet is annihilated by $f$, that is, $j^kf \circ j^kh =0$.
The $k$-jet $j^kh$ in this equation is unique up to $\diff_k(\CC)$. The image of $j^kh$ in $X$ is a $k$-jet of a smooth curve $C_h$ at $p$, and it defines a curvilinear subscheme $\xi_h=(h_1,\ldots, h_n)\in \mathrm{Curv}^k_p(X)$ where $(h_1,\ldots, h_n)$ is the ideal of $h$ in some local coordinates. We claim that $\xi_h \in \Hilb^k(f)$, and this finishes the proof. To check this, note that $j^kf \circ j^kh =j^k(f \circ h)=0$ means that the first $k$ derivatives of the curve $f \circ h$ are zero. Equivalently, the first $k$ derivatives of the corresponding curve  
\[\mathrm{Graph}(C_h)=\{(j^kh(x),j^kf (j^kh(x)):x\in \CC\}  \subset X\times Y\]
on $\mathrm{Graph}(f)$ are horizontal in $X\times Y$.  
\end{proof}
Next we show that for Thom-Boardman maps $\Hilb^k(f)$ can be replaced by $\GHilb^k(f)$ in Proposition \ref{prop:hilbf}.
\begin{proposition}\label{imagepunctual}  Let $f:X \to Y$ be a Thom-Boardman map. Then $\overline{\pi_X(\GHilb_{A_k}(f))}=\Sigma_{A_k}(f)$. 
\end{proposition}
\begin{proof}
Let $p\in X$ is a point of multiplicity $k$, that is, $\dim(A_{f,p})=k$. If $\GHilb^k_p(f)$ is nonempty, then $\GHilb^k_p(f)=\mathrm{Spec}(A_{f,p})$
is the scheme given by the local algebra of $f$ at $p$, and this ideal sits in the curvilinear component $\CHilb^k_p(X)$ by Proposition \ref{prop:boardman}. This gives a 1-1 map $\GHilb^k_0(f) \to \pi_X(\GHilb^k_0)$ close to points $p$ where $f_p$ has multiplicity $k$. Hence $\pi_X(\GHilb^k_0(f)) \subset \Sigma_{A_k}(f)$. However, for Thom-Boardman maps $\Sigma_{A_k}(f)$ is irreducible of codimension $(k-1)(m-n+1)$, and $\GHilb^k_0(f)$ has the same codimension.
\end{proof}

\subsection{Local analysis of multipoint locus}

Recall the $k$-fold locus $X_k(f)=\overline{\pi_1(X_k^\times)}$ of the map $f:X \to Y$ where 
\[X_k^\times=\{(p_1,\ldots, p_k):f(p_1)=\ldots =f(p_k), p_i \neq p_i\} \subset X^{\times k} \]
and $\pi_1$ is projection to the first factor. The punctual $k$-fold locus is
\[X_k^0(f)=\pi_1(\Delta \cap \overline{X_k^\times})\]
where $\Delta$ is the small diagonal of the Cartesian product $X^{\times k}$. Equivalently, 
\[X_k(f)=\pi_1(HC(\GHilb^k(f))) \text{ and } X_k^0(f)=\pi_1(HC(\GHilb^k_0(f)))\]
where $HC: \Hilb^{[k]}(X) \to X^k$ is the ordered Hilbert-Chow morphism. 
By Proposition \ref{imagepunctual}, for Thom-Boardman f, the map 
\[\pi_X=\pi_1 \circ HC: \GHilb^{k+1}_0(f) \to \pi_X(\GHilb^{k+1}_0(f))=\Sigma_{A_k}\]
is isomorphism over the $k$-fold locus, i.e where the multiplicity of $f$ is strictly $k$, not larger.

In this section we study local behaviour of the $k$-fold locus $X_k(f)$ near its punctual part. We start with looking at Porteous points. 

\begin{definition} The point $p\in X$ is called $k$-Porteous point of the holomorphic map $f:X \to Y$ if the local algebra of $f$ at $p\in \Sigma^k$ is isomorphic to $A_{\Sigma^k}=\CC[x_1,\ldots, x_k]/(x_1,\ldots, x_k)^2$. The multiplicity of $f_p$ is $k+1$, and its Thom-Boardman type is $(k,0,0,\dots)$, that is, $p \in \Sigma^k(f)$.
\end{definition}

The local algebra of $f$ at a generic point of the Morin locus $\Sigma_{A_k}(f)$ is the Morin algebra $A_k=\CC[t]/t^k$ with one generator, that is $\dim_\CC(A_k/\mathfrak{m}^2)=1$. We will call all other points, where the local algebra is not Morin, or equivalently $\dim(A_f/\mathfrak{m}^2)\ge 2$, boundary points. We call $V_f=A_f/\mathfrak{m} \subset \CC^n$ the supporting subspace. A generic boundary point has two-dimensional supporting subspace. 
For a $k$-Porteous germ $f$ the supporting subspace $V_f=A_f/\mathfrak{m}^2=\ker(f) \subset \CC^n$ is $k$-dimensional, and for any $n-1$-dimensional hyperplane $W \subset \CC^n$ with $\dim(W\cap V_f)=k-1$, the restriction $f|_{W}:(W,0) \to (\CC^n,0)$ is $k-1$ Porteous.


\begin{proposition}\label{prop:kporteous}
Let $f:(\CC^n,0) \to (\CC^m,0)$ be a stable Thom-Boardman map germ with Porteous local algebra $A_{\Sigma^k}$, and let $\xi_f=\mathrm{Spec}(A_{\Sigma^k}) \in \CHilb^{k+1}_0(\CC^n)$ be the corresponding point in the Hilbert scheme. Then
\begin{enumerate}
\item a small $\Delta$-neighborhood of $\GHilb^{k+1}_0(f)$ in $\GHilb^{k+1}(f)$ fibers as 
\[\xymatrix{\GHilb^{k+1}_\nabla(f) \ar[r]^{\pi^{k+1}_f} & \GHilb^{k+1}_0(f) \ar@{->>}[r]^{\pi_{\CC^n}} & \overline{\Sigma_{A_k}(f)} \subset \CC^n}\]
where $\pi^{k+1}_f$ is a fibration with affine fibers of dimension $k$, and the projection $\pi_{\CC^n}$ is surjective birational morphism, bijection over $\overline{\Sigma}_k(f) \setminus \Sigma_{k+1}$. Hence $\GHilb^{k+1}(f)$ is locally irreducible at the punctual part $\GHilb^{k+1}_0(f)$.
\item Pick a basis $\CC^n=\Span(e_1,\ldots, e_n)$ such that the supporting subspace of $f$ is $V_f=\Span(e_1,\ldots, e_k)$, and let $A_{f,p} \in \GHilb^{k+1}_0(f)$ be a boundary local algebra with supporting subspace $V_{f,p}$ such that $\Span(e_1, e_k) \subset V_{f,p}$. Let $W=\Span(e_1,\ldots, e_{k-1},e_{k+1},\ldots, e_n)$. Then $f|_W: W \to \CC^n$ is $k-1$ Porteous germ with fibration $\pi^{k}_{f|_W}: \GHilb^{k}_\nabla(f|_W) \to \GHilb^{k}_0(f|_W)
$, and the fiber over $A_{f,p}$ is isomorphic to 
\[(\pi^{k+1}_f)^{-1}(A_{f,p}) \simeq  (\pi^{k}_{f|_W})^{-1}(A_{f|_W,p}) \times \Span_\CC(e_k).\]
 \item The fiber over the origin $0 \in \Sigma_{A_k}(f)$ is isomorphic to 
\[\GHilb^{k+1}_{0}(f) \simeq \{((y_1,0,\ldots ,0),(0,y_2,\ldots, 0),\ldots (0,\ldots, 0,y_k)): y_i \in \CC\}\simeq (\CC^1)^{\otimes k}\]
\end{enumerate}
\end{proposition}

\begin{proof}
We follow the procedure described in the previous section to construct a stable Porteous germ in any codimension. The rank 0 genotype of an $k$-Porteous singularity is 
\[g_k: (x_1,\ldots, x_k) \mapsto (x_ix_j :1\le i\le j \le k)\]
whose codimension is $l=k(k+1)/2-k={k \choose 2}$. A stable unfolding of codimension $l$ is the map $\tg_k:(\CC^{k^2(k+1)/2},0) \to (\CC^{k^2(k+1)/2+{k \choose 2}},0)$ 
\[
\tg_k: (x_i, a_{ij}^s: 1\le i,j,s \le k) \mapsto \\
(x_ix_j+\sum_{s=1}^k a_{ij}^s x_s, a_{ij}^s: 1\le i,j,s \le k)
\]
where we add a linear form $L_{ij}=a_{ij}^1x_1+\ldots +a_{ij}^kx_k$ to each coordinate function, with the restriction that $k$ out of the $k^2(k+1)/2$ coefficients $a_{ij}^s$ are equal to zero. The choice of this $k$-tuple is not unique, resulting in different stable unfoldings, and a possible choice is
\begin{equation}\label{whicharezero}
a_{1j}^1=0 \text{ for } 1\le j \le k.
\end{equation}

We would like to describe the level set $\tg_k^{-1}(\Lambda_{ij}, a^s_{ij})$ for $|\Lambda_{ij}|,|a^s_{ij}|<\epsilon$. This is equivalent to solving the system 
\begin{equation}
x_ix_j+\sum_{s=1}^k a_{ij}^s x_s-\Lambda_{ij}=0 \tag{[ij]}
\end{equation}
With the choice \eqref{whicharezero} the equations indexed by ([12]),([13]),...,([1k]) can be written in the following form
\begin{equation}\begin{pmatrix}x_1+a_{12}^2 & a_{12}^3 & a_{12}^4 & \cdots & a_{12}^k \\
a_{13}^2 & x_1+a_{13}^3 & a_{13}^4 & \cdots & a_{13}^k \\
a_{14}^2 & a_{14}^3 & x_1+a_{14}^4 & & a_{14}^k \\
\cdots & & & \ddots & \\
a_{1k}^2 & & & & x_1+a_{1k}^k
\end{pmatrix} 
\begin{pmatrix}
x_2 \\ x_3\\ \vdots \\ \\ x_k
\end{pmatrix}=\begin{pmatrix}
\Lambda_{12} \\ \Lambda_{13} \\ \vdots \\ \\ \Lambda_{1k}
\end{pmatrix} \tag{[12])-([1k]}
\end{equation}   
For generic choice of $x_1,a_{1j}^s$ the $(k-1) \times (k-1)$ matrix $A$ on the left hand side is diagonalizable, and in an eigenbasis given by $J$ the equation can written as $(J^{-1}AJ)J^{-1}\mathbf{x}=J^{-1}\Lambda$. The trace of $A$ is linear, whereas the determinant of $A$ is a degree $k-1$ polynomial in $x_1$, hence the same is true for the diagonal matrix $J^{-1}AJ$. Therefore the diagonal entries of $J^{-1}AJ$ are linear forms in $x_1$, and in the eigenbasis $J$ the equation ([12])-([1k]) has the form 
\[\mathrm{diag}(x_1+\a_2,x_1+\a_3,\ldots, x_1+\a_k)\tilde{\mathbf{x}}=\tilde{\Lambda}.\]
that is, with $\tx_1=x_1$
\begin{equation}
\tx_j(x_1+\a_j)=\tilde{\Lambda}_{1j} \text{ for } 2\le j \le k \tag{[1j]'}
\end{equation}
for $\tilde{\bx}=J^{-1}\bx, \tilde{\Lambda}=J^{-1}\Lambda$, and some $\a_2,\ldots, \a_k \in \CC$. Note that $\a_2,\ldots, \a_k$ depend only on $a_{1j}^s$, but not on $\L_{ij}$. 
Now equation ([11]) can be rewritten in the basis $J$ as
\begin{equation}\label{eq11}
\Lt_{11}=\tx_1^2+\ta_{11}^2\tx_2+\ldots +\ta_{11}^k\tx_k
\end{equation}
After substituting [1j]' into this equation has various different forms as follows. 

First assume that $\prod_{i=2}^k \Lt_{1i} \neq 0$ then we multiply \eqref{eq11} by $\prod_{i\neq 1}(\tx_1+\a_i)$ and obtain  
\begin{equation}
\Lt_{11}= \tx_1^2\prod_{i \neq 1}(\tx_1+\a_i)+\ta_{11}^2\Lt_{12}\prod_{i \neq 1,2}(\tx_1+\a_i)+\ldots + \ta_{11}^k\Lt_{1k}\prod_{i \neq 1,k}(\tx_1+\a_i)  \tag{[11]'}
\end{equation}
The first term has degree $k+1$ in $\tx_1$, hence the $x_1$ coordinates of the $k+1$ points in $\tg_k^{-1}(\Lambda_{ij}, a_{ij})$ are the $k+1$ solutions (with multiplicity) of (11) and the corresponding $x_2,\ldots, x_k$ coordinates are given by [1j]'.
For $s>1$ equation ([ss]) has slightly different form 
\begin{multline} \Lt_{ss}= \Lt_{1s}^2 \prod_{i \neq 1,s}(\tx_1+\a_i)+ \ta_{ss}^1\tx_1(\tx_1+\a_s)^2\prod_{i \neq 1,s}(\tx_1+\a_i)+\\ \tag{[ss]'}
+\ta_{ss}^2\Lt_{12}(\tx_1+\a_s)^2\prod_{i \neq 1,2,s}(\tx_1+\a_i)+\ldots + \ta_{ss}^k\Lt_{1k}(\tx_1+\a_s)^2\prod_{i \neq 1,k,s}(\tx_1+\a_i) \nonumber
\end{multline}
but these are generically also degree $k+1$ polynomials in $\tx_1$. On the other hand, the equations $([ij]')$ for $1<i\neq j$ have degree less than $k+1$ in $x_1$, hence in order to get $k+1$ points in the level set $\tg_k^{-1}(\Lambda_{ij}, a_{ij})$ these equations must vanish. Hence the level set $\tg_k^{-1}(\Lambda_{ij}, a_{ij})$ has (maximal) multiplicity $k+1$ if and only if 
\begin{enumerate}[(i)]
\item The equations $([11]'),([22]'), \ldots, ([kk]')$ are collinear.
\item The equations $([ij]')$ for $1<i \neq j$ are identically $0$. 
\end{enumerate}
Equations given by (i) and (ii) together with $[11]',[12]',\ldots, [1k]'$ cut out the $k+1$-fold locus $\CC^{k^2(k+1)/2}_{k+1}$ of $\tg_k$ from the source space $\CC^{k^2(k+1)/2}$. 
Note that $a_{1j}^s$ determine the eigenvalues $\a_2,\ldots, \a_k$, and the nonzero $\Lt_{1j}$'s together with $\a_2,\ldots, \a_k$ determine the roots $\b_1,\ldots, \b_{k+1}$. Indeed, substituting $\tx_1=-\a_j$ into [11]' we get 
\[\Lt_{11}=\ta_{11}^j \Lt_{1j} \prod_{i \neq j}(\a_i-\a_j)\]
which determine $\ta_{11}^j$ for $2\le j \le k$ and hence this determines the roots. Due to the missing degree $k$ term from [11]', $\sum_{i=1}^{k+1}\b_i=0$. For fixed $\Lt_{12},\ldots, \Lt_{1k}$ any $k+1$-tuple $\bbb=(\b_1,\ldots, \b_{k+1})$ with $\sum_{i=1}^{k+1}\b_i=0$ and $||\bbb||$ sufficiently small appears as root set for some choice of $a_{1j}^s$. Then we choose $a_{ij}^s$ for $i>1$ such that (i) and (ii) are satisfied, and by [1j]'
\[\tg_k^{-1}(\Lt_{ij},a_{ij})=\left\{\left(\b_i,\frac{\Lt_{12}}{\b_i+\a_2},\ldots, \frac{\Lt_{1k}}{\b_i+\a_k}, a_{ij}^s \right):1\le i \le k+1\right\}.\]
Then  
\[\GHilb^{k+1}_\nabla(\tg_k)=\overline{\left\{\sqcup_{i=1}^{k+1}\left(\b_i,\frac{\Lt_{12}}{\b_i+\a_2},\ldots, \frac{\Lt_{1k}}{\b_i+\a_k},\baa \right): |\baa|,|\Lt_{ij}| <\varepsilon \text{ and (i),(ii) holds} \right\}}\]
and $\pi_f^{k+1}$ is the projection to the punctual part $\GHilb^{k+1}_0(\tg_k)$ is where $\b_1=\ldots =\b_{k+1}=0$. Let $\pi_{\CC^{k^2(k+1)/2}}:\Hilb^{k+1}_0(\CC^{k^2(k+1)/2}) \to \CC^{k^2(k+1)/2}$ be the projection. By Proposition \ref{prop:hilbf}, $\pi_{\CC^{k^2(k+1)/2}}(\GHilb^{k+1}_0(\tg_k))=\overline{\Sigma_{A_k}(\tg_k)}$, and the Morin locus $\Sigma_{A_k}(\tg_k)=\Sigma^{1,1,\ldots, 1}(f)$ is irreducible since $\tg_k$ is stable, Thom-Boardman. Hence $\GHilb^{k+1}_\nabla(\tg_k)$ fibers over the irreducible $\overline{\Sigma_{A_k}(\tg_k)}$ with affine fibers of dimension $k$: the roots $\b_i$ can run freely under the restriction that their sum is zero:
\[\GHilb^{k+1}_{\Lt_{12},\ldots, \Lt_{1k}}(\tg_k)=\overline{\left\{\sqcup_{i=1}^{k+1}\left(\b_i,\frac{\Lt_{12}}{\b_i+\a_2},\ldots, \frac{\Lt_{1k}}{\b_i+\a_k},\baa \right): \text{ (i),(ii) holds } \right\}} \simeq \mathbf{A}^{k}_{\varepsilon}\]
Hence $\GHilb^{k+1}_\nabla(\tg_k)$ is irreducible, and we finally note that if the matrix $A$ in (12)-(1k) is 
not diagonalizable, we can still use its Jordan normal form and the 
corresponding level set is given by the limit of level sets corresponding to 
diagonalizable matrices.

To finish the proof of (1) and to prove (2) we study the special fibers of the fibration 
\begin{equation}\label{fibration} \GHilb^{k+1}_\nabla(\tg_k) \to \CC^{k-1},\ (\Lt_{ij},a_{ij}^s) \mapsto (\Lt_{12}, \ldots, \Lt_{1k}).
\end{equation}
We have seen that if $\prod_{i=2}^k \Lt_{1i}\neq 0$ then the fibers are locally affine spaces of dimension $k$. If $\prod_{i=2}^k \Lt_{1i}=0$, wlog assume $\Lt_{1k}=0$. By [12]'  $\tx_k(\tx_1+\a_k)=0$. If $\tx_k=0$, then $\tg_k|_{\tx_k=0}$ is a $k-1$-Porteous singularity which is the stable unfolding of the genotype 
\[g_{k-1}: (x_1, \ldots, x_{k-1}) \mapsto (x_ix_j :1\le i\le j \le k-1).\] 
By induction we can assume that the $k$-fold locus $\GHilb^k(\tg_k|_{\tx_k=0})$ is locally irreducible on the $\tx_k=0$ hyperplane at the origin. The $k+1$th point comes from [12]' when $\tx_1+\a_k=0$. Then $\tx_k$ runs freely but $\tx_j=\frac{\Lt_{1j}}{\a_j-\a_k}$ for $1 \le j \le k-1$ from [1j]', assuming $\a_j \neq \a_k$. Hence the special fiber $\GHilb^{k+1}_{\Lt_{12},\ldots, \Lt_{1k-1}}(\tg_k)$ is locally affine, of dimension $k$, and hence it must be the limit of the generic fibers in \eqref{fibration}. This proves (1). 

For (2) note that the previous paragraph says that 
 \[(\pi^{k+1}_f)^{-1}(A_{f,p}) \simeq \GHilb^{k+1}_{\Lt_{12},\ldots, \Lt_{1k-1},0}(\tg_k)\simeq \GHilb^{k}_{\Lt_{12},\ldots, \Lt_{1k-1}}(\tg_{k-1}) \times \{(0,\ldots 0,\tx_k): \tx_k \in \CC\}.\]

For (3) we note that the fiber over the origin is the fiber $\GHilb^{k+1}_{\Lt_{12},\ldots, \Lt_{1k}}(\tg_k)$ where $\Lt_{12}=\ldots =\Lt_{1k}=0$. If $\a_i \neq \a_j$ for $2\le i<j \le k$, then from the equations [1j]' this is 
\[\GHilb^{k+1}_{0,\ldots, 0}(\tg_k)\simeq \{(0,\ldots,0,y_j,0,\ldots, 0): 1\le j \le k\}=\mathbf{A}_{\epsilon}^k\]

\end{proof}


Proposition \ref{prop:kporteous} shows irreducibility of $\GHilb^{k+1}(f)$ at points $p$ where $f_p$ has multiplicity $k+1$. Next we study local irreducibility at higher multiplicity points. These points naturally sit in the closure of the locus where the multiplicity is $k+1$.

\begin{proposition}\label{prop:k+lporteous}
Let $f:(\CC^n,0) \to (\CC^m,0)$ be a stable Thom-Boardman map germ with local algebra $A_{\Sigma^{k+\ell}}$ for some $\ell>0$. Then $\GHilb^{k+1}(f)$ is locally irreducible at points of $\GHilb^{k+1}_0(f)$ 
\end{proposition}
\begin{proof} Since $\xi_f=\Spec(A_{\Sigma^{k+\ell}})$ is a $k+\ell$-Porteous point, its kernel $\ker(f_0)$ is a $k+\ell$-dimensional subspace in $\CC^n$. Let $\xi_i \in \GHilb^{k+1}(f)$ be a sequence of $k$-fold points on $\CC^n$ with limit $\xi =\lim_{i\to \infty} \xi_i \in \GHilb^{k+1}_0(f)$ Then $\xi$ is a $k$-Porteous point with a $k$-dimensional kernel $\ker(\xi) \subset \ker(f_0)$, and any this kernel uniquely determines the point $\xi$.  Hence 
\[\GHilb^{k+1}_0(f)=\grass_k(\ker(f_0))\]
and by Proposition \ref{prop:kporteous} the $k+1$-fold locus $\GHilb^{k+1}(f)$ fibers over this Grassmannian with locally irreducible fibers. 
\end{proof}

We summarize the results of this section in the following proposition, which gives the local geometry of $\GHilb^{k}(f)$ near the punctual part $\GHilb^k_0(f)$ for stable Thom-Boardman maps $f:X \to N$. 
\begin{definition} Let $X$ be a smooth variety. Then $\pi_X^* \calo_X \subset \calo_X^{[k]}|_{\GHilb^k_0(X)}$ and $B=\calo_X^{[k]}/\pi_X^* \calo_X$ is the Haiman bundle over $\GHilb^k_0(X)$.
\end{definition}

\begin{proposition}\label{prop:hilbfloc}  Let $f:X \to Y$ be a stable Thom-Boardman map. Then 
\begin{enumerate}
\item $\GHilb^{k}(f)$ is locally irreducible at points of the punctual part $\GHilb^{k}_0(f)$.
\item $\GHilb^{k}(f)$ is locally isomorphic to the total space of the Haiman bundle $B$ established by a topological isomorphism $\nu$:
\begin{equation}\label{localmodelb}
\xymatrix{\GHilb^k_\nabla(f) \ar[d]  \ar[r]^\nu &  B  \ar[d]  \\
  \GHilb^k_0(f) \ar@{^{(}->}[r] &  \CHilb^k(X)}
 \end{equation}
 \end{enumerate}
\end{proposition}

 \begin{proof}
(1) Local irreducibility follows from Proposition \ref{prop:k+lporteous} and Proposition \ref{prop:boardman}. Indeed, for a Thom-Boardman map every singularity locus $\Sigma_A(f) \subset X$ contains Porteous-points in its closure, and local irreducibility is a closed condition. 

(2) In \cite{bsz2} we prove that $\GHilb^k_\nabla(f) \simeq B$ over the open dense curvilinear locus $\Curv^k(f)\subset \GHilb^k_\nabla(f)$. But $B$ extends to $\GHilb^k(f)$, and by Proposition \ref{prop:kporteous} $\GHilb^k_\nabla(f)$ fibers over $\GHilb^k_0(f)$ with affine fibers, which must coinside with $B$.
 \end{proof}

If $f: X \to Y$ is a stable Thom-Boardman map, then for any algebra $A \in \CHilb^k(\CC^n)$ there is a point $p \in X$ where the local algebra $A_{f,p}$ of $f$ is isomorphic to $A$. Hence the evaluation map 
\[\xymatrix{\GHilb^k_0(f) \ar@{->>}[r]^{ev} \ar[d]^{\pi_X} & \CHilb^k_0(\CC^n) \\ \Sigma_{A_k}(f)}\]
which sends $\xi_p \in \pi_m^{-1]}(p)$ to $\Spec(A_{f,p})$ is surjective. Hence \eqref{localmodelb} gives  
\begin{equation}\label{localmodelc} \xymatrix{\GHilb^k_\nabla(f) \ar[d]  \ar[r]^\nu &  B|_{\CHilb^k_0(\CC^n)}  \ar[d]   \\
\GHilb^k_0(f) \ar@{->>}[r]^{ev} & \CHilb^k_0(\CC^n)}
\end{equation}
This diagrams roughly says that for any local algebra $A \in \CHilb^k_0(\CC^n)$ the fiber $B_A$, can be identified with the fiber $B_{A_{f,p}}$ for a local algebra $A\simeq A_{f,p} \in \GHilb_0^k(f)$ for a stable Thom-Boardman map.

\section{First examples: Hilbert scheme of $2$ and $3$ points}\label{sec:tauintegral}

Let $X$ be a smooth projective variety of dimension $n$, and let $V$ be a rank $r$ algebraic vector bundle on $X$. Let $V^{[k]}$ be the corresponding tautological rank $rk$ bundle on $\Hilb^k(X)$ whose fibre at $\xi \in \Hilb^k(X)$ is $H^0(\xi,V|_\xi)$. 

 Let $\Phi(c_1,\ldots, c_{rk})$ be a Chern polynomial in $c_i=c_i(V^{[k]})$ of weighted degree equal to $\dim \GHilb^k(X)=km$. The Chern numbers 
\[\int_{\GHilb^k(X)} \Phi(V^{[k]})\]
are called tautological integrals of $V^{[k]}$ on the geometric component. Rennemo \cite{rennemo} shows that for any $\Phi$ there is a universal polynomial $P_{\Phi}$ (independent of $X$) such that the integral $\int_{\GHilb^k(X)} \Phi$ is obtained by substituting the Chern roots of $X$ and $V$ into $P_{\Phi}$. Over the curvilinear component $\CHilb^k(X)$ closed formula for $P_\Phi$ is developed in \cite{bercziG&T}. 

We first prove Theorem \ref{mainthm} when the number of points is $k=2$ and $k=3$. In these cases both the Hilbert scheme and the punctual Hilbert scheme is irreducible: 
\[\Hilb^k(\CC^n)=\GHilb^k(\CC^n) \text{ and } \GHilb^k_0(\CC^n)=\CHilb^k(\CC^n),\] 
and we do not have to deal with all technical difficulties, still, the proof shows most of the key ideas and steps. 

\subsection{Hilbert scheme of $2$ points}
 
For $k=2$ the Theorem \ref{mainthm} reads as follows. 
\begin{proposition}[\textbf{Integration over $\GHilb^2(X)$}]
	Let $V$ be a rank-$r$ bundle over $X$, and $\Phi$ be a symmetric 
	polynomial of degree $2n$ in $2r$ variables. Then
	\[\int_{\GHilb^2(X)}\Phi\left(V^{[2]}\right)=\int_X \res_{z=\infty} \Phi\left(V\oplus V(z)\right)\,
	s_X\left(\frac{1}{z}\right)\,\frac{dz}{z^{n+1}}+\int_{X \times X} \Phi\left(V_1 \oplus V_2\right)\]
where 
\begin{itemize}
	\item $V_i$ is the bundle $V$ pulled back from the $i$th copy of $X$ for $i=1,2$,
	\item $V(z)$ stands for the bundle $V$ tensored by the line $\mathbb{C}_z$, which is the representation of a torus $T$ with weight $z$. Then 
	\[ \Phi(V\oplus V(z)) =  \Phi(\theta_1,z+\theta_1,\theta_2,z+\theta_2,\dots,\theta_r,z+\theta_r)
	\]
	\item $s_X\left(\frac{1}{z}\right)$ is the Segre series of $X$.
\end{itemize}
\end{proposition}

\begin{proof}
Let $\hch:\GHilb^2(X)\to X\times X$ be the Hilbert-Chow morphism, and introduce the notation $V^{[1,1]} = \hch^*(V_1 \oplus V_2)$.	We begin by writing 
	\[\int_{\GHilb^2(X)}\Phi(V^{[2]})=
	\int_{\GHilb^2(X)}\left[ \Phi(V^{[2]})-\Phi(V^{[1,1]})\right] +
	\int_{\GHilb^2(X)}\Phi(V^{[1,1]}),
	\]
	and observing that we can rewrite the second term as follows:
	\[ \int_{\GHilb^2(X)}\Phi(V^{[1,1]}) = \int_{X \times X} \Phi\left(V_1 \oplus V_2\right). \]
	As to the first term, we note that the integrand may be represented as a form supported in a neighborhood of $\CHilb^2_0$ which is, in turn, isomorphic to the normal bundle $\nabla_0\to\CHilb^2_0$ (a line bundle in this case) of the punctual Hilbert scheme in the entire Hilbert scheme. According to the Thom isomorphism theorem, we have $\hcpt(\nabla_0)\simeq H^*(\CHilb^2_0)\cdot u$, where $u$ is the compactly supported Thom class. We thus have
	\begin{equation*}
		\int_{\GHilb^2(X)}\left[ \Phi(V^{[2]})-\Phi(V^{[1,1]})\right] =
		\int_{\nabla_0}\left[ \Phi(V^{[2]})-\Phi(V^{[1,1]}\right] =
		\int_{\CHilb^2_0} \frac{\Phi(V^{[2]})-\Phi(V^{[1,1]})}{u}
	\end{equation*} 
	
	Now we use the family version of Kalkman's formula: we represent $\CHilb^2_0$ as a $T=\C^*$ quotient of the normal bundle $\nabla_{ X}$ of $\Delta X\subset X\times X$. We lift the bundles $V^{[2]}$ and $V^{[1,1]}$ to equivariant bundles $V_z^{[2]}$ and $V_z^{[1,1]}$ on $\nabla_{ X}$, the Euler class $u$ to $u_z$, and then the formula reads
	\[ \int_{\CHilb^2_0} \frac{\Phi(V^{[2]})-\Phi(V^{[1,1]})}{u}
	=  \int_{\Delta X} \res_{z=\infty} \frac{\Phi(V_z^{[2]})-\Phi(V_z^{[1,1]})}{u_z\cdot c_m(\nabla^z_{\Delta X})},
	\]
	where all the objects under the integral sign are restricted to $\Delta X$, the zero-section of $\nabla_{ X}$; identifying $\Delta X$ with $X$, these restrictions are as follows:
\begin{itemize}
	\item $V_z^{[2]}=V\oplus V(z)$
	\item  $V_z^{[1,1]}=V^{[1,1]}$
	\item $u_z=z$
	\item $c_m(\nabla^z_{ X})=\prod_{i=1}^{m}(z+\lambda_i)$, where the $\lambda$s are the Chern roots of $\nabla_{ X}$.
\end{itemize}
	Thus we have
	\[  \int_{\CHilb^2_0} \frac{\Phi(V^{[2]})-\Phi(V^{[1,1]})}{u}
	= \int_{\Delta X} \res_{z=\infty} \frac{\Phi(V\oplus V(z))-\Phi(V^{[1,1]})}{z\cdot \prod_{i=1}^{m}(z+\lambda_i)} \]
	Finally, we observe that the second term of the numerator is independent of $z$, and thus does not contribute to the residue. We  thus arrive at the final formula
\[	\int_{\GHilb^2(X)}\left[ \Phi(V^{[2]})-\Phi(V^{[1,1]})\right] =	\int_X \res_{z=\infty} \Phi\left(V\oplus V(z)\right)\,
	s_X\left(\frac{1}{z}\right)\,\frac{dz}{z^{m+1}}, \]
and this completes the proof.
\end{proof}

\subsection{The Hilbert scheme of three points}\label{subsec:k=3}

In this section we follow our strategy to prove the formula for $3$ points. Xost of the key ingredients of the general argument already appears here. 

For $k=3$ the geometric component $\GHilb^3(X)=\Hilb^3(X)$ is, again, equal to the full Hilbert scheme, and the punctual part is irreducible, consisting only of the curvilinear component $\CHilb^3_0(\CC^n)$, there are no exotic components.  

\begin{proposition}\label{prop:k=3}
	Let $V$ be a rank-$r$ bundle over the $n$-dimensional manifold $X$ and $\Phi$ a symmetric 
	polynomial of degree $3n$ in $3r$ variables. Then
	\begin{eqnarray*} \int_{\GHilb^3(X)}\Phi\left(V^{[3]}\right)&=&\int_X \res_{z_1=\infty} \res_{z_2=\infty} \frac{\Phi\left(V\oplus V(z_1)\oplus V(z_2) \right)(z_1-z_2)dz_1dz_2}{(z_1z_2)^{n+1}(2z_1-z_2)}\,
		s_X\left(\frac{1}{z_1}\right) s_X\left(\frac{1}{z_2}\right)\,+ \\
		& +& \sum_{i=1}^3 \int_{X\times X} \res_{z=\infty} \frac{\Phi\left(V\oplus V(z) \oplus V_i \right)dz}{z^{n+1}}\,
		s_X\left(\frac{1}{z}\right)
		+\\
		&+& \int_{X \times X \times X} \Phi\left(V \oplus V \oplus V \right)
	\end{eqnarray*}
	
\end{proposition}
\begin{proof}
	We let $\GHilb^{[1,2,3]}(X)$ denote the ordered Hilbert scheme of $3$ points, which is a branched 6-fold cover of $\GHilb^3(X)$. Let $\GHilb^{[1,2],[3]}(X)=\overline{\im(\phi^{[1,2],[3]})}$ denote the closure of the image of the rational morphism 
	\[\phi^{[1,2],[3]}: \GHilb^{[1,2,3]}(X) \dasharrow \GHilb^{[1,2]}(X) \times X\]
	which is defined on the locus where point number $3$ does not collide with points number $1$ and $2$. We define $\GHilb^{[2,3],[1]}(X), \GHilb^{[1,3],[2]}(X)$ and $\GHilb^{[1],[2],[3]}(X)$ similarly, and we call these birational models \textit{ approximating Hilbert schemes} for short. Note that $\GHilb^{[1],[2],[3]}(X)=X\times X \times X$. We form a rational map 
\[\phi=(id, \phi^{[1,2],[3]}, \phi^{[1,3],[2]},  \phi^{[2,3],[1]}, \phi^{[1],[2],[3]})\]
which sends a generic point of the ordered Hilbert scheme to the product of the approximating sets:
\[\xymatrix{ \GHilb^{[1,2,3]}(X) \ar@{-->}[r]^-{\phi} & \GHilb^{[1,2,3]} \times \GHilb^{[1,2],[3]} \times \GHilb^{[1,3],[2]} \times \GHilb^{[2,3],[1]} \times \GHilb^{[1],[2],[3]}}\]
We call $N^3(X)=\overline{\im(\phi)}$ the \textit{fully nested Hilbert scheme}, which is endowed  
with projections $\pi^{[i,j],[k]}$ and $\pi^{[1],[2],[3]}$ to the corresponding approximating sets. Pull-back along these projections provide \textit{approximating bundles}
\[V^{[i,j],[k]}=\pi^{[i,j],[k]*}(V^{[i,j]} \oplus V) \text{ and } V^{[1],[2],[3]}=\pi^*(V \oplus V \oplus V)\]
where $V^{[ij]} = V^{[2]}$ stands for the tautological bundle over $\GHilb^{[i,j]}(X)$.

Let $\hch:\GHilb^{[1],[2],[3]}(X)\to X\times X \times X$ be the ordered Hilbert-Chow morphism, and for any partition $\l$ of $\{1,2,3\}$ let 
$N^\l_0(X)=(\pi^{[1],[2],[3]})^{-1} \Delta_\l$
denote the punctual part sitting over the corresponding diagonal of $X \times X \times X$. 
The sieve formula in Step 1 of the strategy has the following form:
\begin{align} \nonumber
	         \int_{\GHilb^3(X)} \Phi(V^{[3]})&=\int_{\GHilb^{[1],[2],[3]}(X)} \Phi(V^{[1,2,3]})=\int_{N^3(X)} \pi^{[1,2,3]*}\Phi(V^{[1,2,3]})=\\ \nonumber
		&=\int_{N^3(X)}\left[ \Phi(V^{[1,2,3]})-\Phi(V^{[1,2],[3]]})-\Phi(V^{[1,3],[2]]})-\Phi(V^{[2,3],[1]]})+2\Phi(V^{[1],[2],[3]]})\right] +\\ \nonumber
		&+\int_{N^3(X)}\left[ \Phi(V^{[1,2],[3]]})-\Phi(V^{[1],[2],[3]})\right]\\ \nonumber
		&+\int_{N^3(X)}\left[ \Phi(V^{[1,3],[2]]})-\Phi(V^{[1],[2],[3]})\right]\\ \nonumber
		&+\int_{N^3(X)}\left[ \Phi(V^{[2,3],[1]]})-\Phi(V^{[1],[2],[3]})\right]\\ \nonumber
		&+\int_{N^3(X)} \Phi(V^{[1],[2],[3]}) \nonumber
		\end{align}
The second term can be written as 
\[ \int_{N^3(X)}\Phi(V^{[1,2],[3]]})-\Phi(V^{[1],[2],[3]}) = \int_{\Hilb^{[1,2]}(X) \times X} (\Phi(V^{[1,2]})-\Phi(V^{[1],[2]})) \oplus V. \]
and we can apply the residue formula we presented for $k=2$ points in the previous section. This gives the second term in Proposition \ref{prop:k=3}. The last term can be written as 
\[\int_{N^3(X)} 2\Phi(V^{[1],[2],[3]})=\int_{X\times X\times X}V_1 \oplus V_2 \oplus V_3\]
It remains to prove the residue formula for the first term, where the integrand 
\[\Phi^+=\Phi(V^{[1,2,3]})-\Phi(V^{[1,2],[3]]})-\Phi(V^{[1,3],[2]]})-\Phi(V^{[2,3],[1]]})+2\Phi(V^{[1],[2],[3]]})\] 
is supported in a neighborhood $N^3_\nabla(X)$ of the punctual part $N^3_0(X)$. Hence, as explained in \S \ref{subsec:reduction}, we can reduce the calculation to equivariant  integration assuming $X=\CC^n$. In practice this means that we apply equivariant integration, and then substitute the Chern roots of $X$ into the torus weights (this is the Chern-Weil map).

The next crucial step is to define partial blow-ups $\widehat{\GHilb}^3(\CC^n) \to \GHilb^3(\CC^n)$ and $\hat{N}^3(\CC^n) \to N^3(\CC^n)$ with fibration 
\begin{equation}\label{fibration2}
\xymatrix{\hat{N}^3(\CC^n) \ar[r]^-\pi & \widehat{\GHilb}^3(\CC^n) \ar[r]^-\rho & \flag_2(T\CC^n) \ar[r] & \CC^n}
\end{equation}  
We define $\widehat{\GHilb}^3(\CC^n)$ as the closure of the graph of the rational map $\rho: \GHilb^3(\CC^n) \dasharrow \flag_{2}(T\CC^n)$ defined on the open locus formed by $3$ different points as
\[(p_1,p_2,p_3) \mapsto (\frac{1}{3}(p_1+p_2+p_3), \mathrm{Span}(p_1,p_2),\mathrm{Span}(p_1,p_2,p_3)).\]
In particular, the fiber over a flag over the origin $\bW=(0,W_1 \subset W_2 \subset \CC^n)$ is the \textit{balanced Hilbert scheme}
\[\rho^{-1}(\bW)=\BHilb^3(W_2)=\{\xi \in \GHilb^3(W_2): \text{baricenter of }\xi \text{ is the origin}\}\] 
Similarly, we define $\hat{N}^3(\CC^n)$ as the closure of the graph of the rational composition 
\[\rho \circ \pi: N^3(\CC^n) \to \GHilb^3(\CC^n) \dasharrow \flag_{2}(T\CC^n) \to \CC^n\]

We apply torus equivariant localisation for the action of the maximal torus $T \subset \GL(n)$ on the fiber $\hat{N}^3_{\mathrm{origin}}$ over the origin in $\CC^n$:
\[\hat{N}^3_{\mathrm{origin}} \to \flag_{2}(T_0\CC^n)=\flag_2(\CC^n).\]
This action is induced from the standard diagonal action on $\CC^n$ with weights $\l_1,\ldots, \l_m$ in the basis $\{e_1,\ldots, e_m\}$. The fibration \eqref{fibration2} is equivariant, and Proposition \ref{ABtoresidue} turns our integral over $\flag_k(\CC^n)$ into an iterated residue
\begin{equation}\label{integralfirst}
\int_{\hat{N}^3(\CC^n)} \Phi^+=\int_{\CC^n}\res_{z_1=\infty} \res_{z_2=\infty}\frac{(z_1-z_2)\Phi^+_{\bff}dz_2dz_1}{\prod_{i=1}^{n}(z_1-\lambda_i)\prod_{i=1}^{n}(z_2-\lambda_i) }
\end{equation}
where integration over $\CC^n$ is a shorthand notation for substituting the Chern roots of $X$ into $\l_1,\ldots, \l_m$, followed by integration over $X$ (see \S \ref{subsec:reduction}).
Here 
\[\Phi^+_\bff=\int_{\hat{N}^\bff}\Phi^+\]
is the integral of the form on the fiber $\hat{N}^\bff=(\rho \circ \pi)^{-1}(\bff)$ over a fixed reference flag $\bff=(\mathrm{Span}(e_1)\subset \mathrm{Span}(e_1,e_2) \subset \CC^n)\in \flag_2(\CC^n)$. Here $\CC_{[2]}=\Span(e_1,e_2)$, and $z_1,z_2$ stand for the weights of the residual $T_\bz^2=(\CC^*)^2$ action on $\hat{N}^\bff$, and $\hat{N}^\bff$ is the balanced fully nested Hilbert scheme formed by schemes with baricenter at the origin.

We will apply a second localisation with respect to this 2-dimensional torus acting on $\hat{N}^\bff$. The geometry of $\hat{N}^\bff$ is sketched in Figure 1. The dimension is $\dim(\hat{N}^\bff)=3$, whose generic point is a triple $(p_1,p_2,p_3) \in (\CC^2)^{\times 3}$ with $p_1+p_2+p_3=0$, such that the line $p_1p_2$ has fixed direction $e_1$.  The punctual part 
\[\hN^\bff_0=\pi^{-1}(\widehat{\GHilb}^3_0)\cap \hat{N}^\bff\]
consists of triples $(\xi,[v_{12}],[v_{13}],[v_{23}]) \in \GHilb^3_0(\CC^2) \times \PP^1 \times \PP^1$ where $\xi$ is a punctual scheme of length $3$, and $v_{ij}$ are directions corresponding to the length-2 subschemes formed by points $i,j$, but $[v_{12}]=[e_1]$ is fixed. Note that the punctual part $\GHilb^3_0(\CC^2)=\CHilb^3(\CC^2)$ has two types of points: curvilinear subschemes which form a 1-dimensional set, and the Porteous point $\xi_{por}$ corresponding to the ideal $\mathfrak{m}^2=(x^2,xy,y^2)$. 

The punctual part $\hN^\bff_0$ has two components: $\PP^1=CN_{main}^\bff$ is the curvilinear component which is mapped by $\pi$ dominantly to $\GHilb^3_0(\CC_{[2]})$, and a second component $CN_{Por}^\bff=\PP^1 \times \PP^1$ which sits over the Porteous point. A point in $\PP^1 \times \PP^1$ is given by a tuple $(\xi_{por},[v_{12}],[v_{13}],[v_{23}])$ with $[v_{12}]=[e_1]$. So in fact, this is determined by the two directions $[v_{13}],[v_{23}]$, indicated in Figure 1 by yellow resp. green arrows.

\begin{figure}[ht!]\label{figure:fixedpoints}
\centering
\includegraphics[width=90mm]{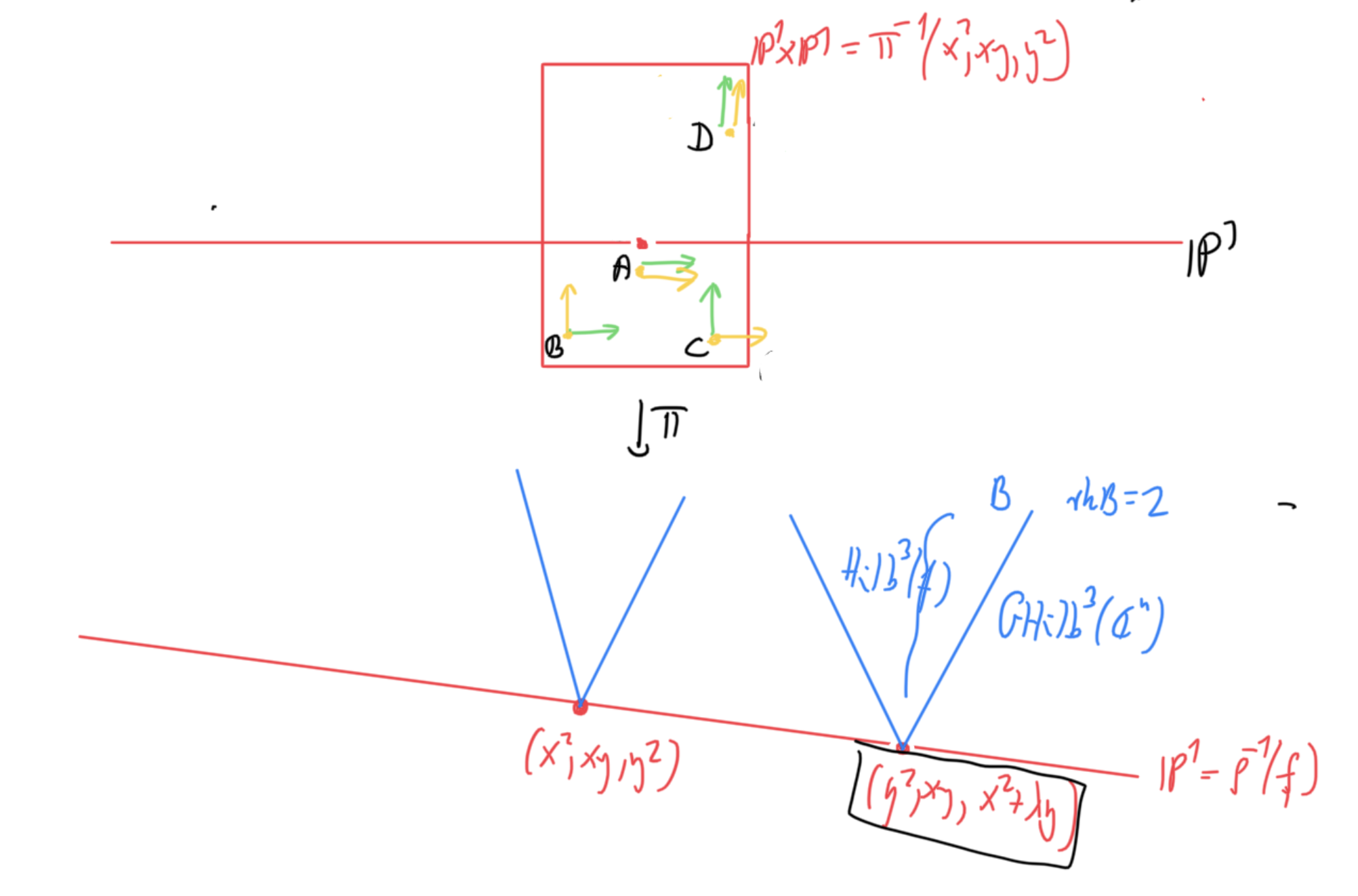}
\caption{$\pi: N^\bff_0 \subset N^\bff$}
\end{figure}

For equivariant localisation over $N^\bff$, in general, we will use a smooth ambient space 
\[\xymatrix{\hat{N}^\bff_0 \ar@{^{(}->}[r] \ar[d] & \widehat{\grass}_k(S^\bullet \CC^n) \ar[ld] \\ \flag_{k-1}(\CC^n) & }\]
For $k=3$, however, $\hat{N}^\bff$ is smooth, hence we can use $\hat{N}^\bff_0 \subset \hat{N}^\bff$ as a smooth ambient space. We apply Atiyah-Bott localisation on $\hat{N}^\bff$ directly. We introduce the notations
\begin{equation}\label{equivbundles}
V_{x,y}^{[1,2,3]}=V\oplus V(x) \oplus V(y),\ \ V_{x}^{[i,j],[k]}=V \oplus V(x) \oplus V,\ \ V^{[1],[2],[3]}=V \oplus V \oplus V
\end{equation}
The following table contains the restrictions of the terms of $\Phi^+_\bff$ at different fixed points $F \in (\hat{N}^\bff)^{T_\bz^2}$.
\begin{center}
\begin{tabular}{|c|c|c|c|c|c|}
\hline
$F$ & $V_F^{[1,2,3]}$ & $V_F^{[1,2],[3]}$ & $V_F^{[1,3],[2]}$ & $V_F^{[2,3],[1]}$ & $V_F^{[1],[2],[3]}$ \\
\hline
$(\xi_{por},e_1,e_1,e_2)$ & $V_{z_1,z_2}^{[1,2,3]}$ & $V_{z_1}^{[1,2],[3]}$ & $V_{z_1}^{[1,3],[2]}$ & $V_{z_2}^{[2,3],[1]}$ & $V^{[1],[2],[3]}$ \\
\hline
$(\xi_{por},e_1,e_2,e_1)$ & $V_{z_1,z_2}^{[1,2,3]}$ & $V_{z_1}^{[1,2],[3]}$ & $V_{z_2}^{[1,3],[2]}$ & $V_{z_1}^{[2,3],[1]}$ & $V^{[1],[2],[3]}$ \\
\hline
$(\xi_{por},e_1,e_2,e_2)$ & $V_{z_1,z_2}^{[1,2,3]}$ & $V_{z_1}^{[1,2],[3]}$ & $V_{z_2}^{[1,3],[2]}$ & $V_{z_2}^{[2,3],[1]}$ & $V^{[1],[2],[3]}$ \\
\hline
$(\xi_{curv},e_1,e_1,e_1)$ & $V_{z_1,z_1}^{[1,2,3]}$ & $V_{z_1}^{[1,2],[3]}$ & $V_{z_1}^{[1,3],[2]}$ & $V_{z_1}^{[2,3],[1]}$ & $V^{[1],[2],[3]}$ \\
\hline
\end{tabular}
\end {center}
Note that here $\xi_{por}=(x^2,xy,y^2)$ stands for the Porteous point and $\xi_{curv}=(y,x^3)$ stands for the curvilinear fixed point in $\Hilb^3(\CC^2)$. Substitute these into the Atiyah-Bott formula:
{\small \begin{equation}\label{sum}
\int\displaylimits_{\hat{N}^3(\CC^n)}\Phi^+=\res_{z_1,z_2=\infty}\sum\limits_{F \in (\hat{N}^3_\bff)^{T_\bz^2}} \frac{(z_1-z_2)\left(\Phi(V_F^{[1,2,3]})-\Phi(V_F^{[1,2],[3]})-\Phi(V_F^{[1,3],[2]})-\Phi(V_F^{[2,3],[1]})+2\Phi(V_F^{[1],[2],[3]})\right)}{\Euler^z(T_FN^3(\CC^2)) \prod_{i=1}^{n}(z_2-\lambda_i)\prod_{i=1}^{n}(z_1-\lambda_i) }\end{equation}}
The crucial observation is that only $V_F^{[1,2,3]}$ depends on both weights $z_1,z_2$, for all other terms, one of $z_1,z_2$ or both are missing. In particular, the residue of those terms where $z_2$ does not appear is zero, because for $m \ge 3$ the total degree in $z_2$ of the rational expression is larger than $-1$. The residue of those terms in which $z_1$ does not appear but $z_2$ does are not automatically zero: the problem is that a factor of the form $az_1+bz_2$ in the denominator has powers of $\frac{z_1}{z_2}$ in its Taylor expansion on the $|z_1|<|z_2|$ contour, and hence the big negative degree in $z_1$ can be compensated. 

But the residue vanishing theorem for fully nested Hilbert schemes (Theorem \ref{vanishing1}) tells us that in fact all these problematic terms cancel each other, and only $V_F^{[1,2,3]}$ survives. To see this in this small example, we list the normal and tangent weights to $\PP^1 \times \PP^1$ at the fixed points:

\begin{center}
\begin{tabular}{|c|c|c|c|c|c|}
\hline
$F$ &  Tangent to $\PP^1 \times \PP^1$ & Ideal  & Normal to $\PP^1 \times \PP^1$ \\
\hline
$(\xi_{por},e_1,e_1,e_2)$ & $-(z_2-z_1)^2$ & $(x-\epsilon, y^2)\cap (x+2\epsilon ,y)$ & $z_1$\\
\hline
$(\xi_{por},e_1,e_2,e_1)$ &  $-(z_2-z_1)^2$ & $(x+\epsilon, y^2) \cap (x-2\epsilon, y)$ &$z_1$\\
\hline
$(\xi_{por},e_1,e_2,e_2)$ &  $(z_2-z_1)^2$ & $(x^2,y+\epsilon)\cap (x,y-2\epsilon)$  & $z_2$\\
\hline
$(\xi_{curv},e_1,e_1,e_1)$ & $(z_2-z_1)^2$ & $(x^2+\epsilon y,xy,y^2)$ & $2z_1-z_2$ \\
\hline
\end{tabular}
\end {center}
Note that the ideals and normal directions correspond to the following points of the Hilbert scheme: DRAWING COXES HERE

Putting together the localisation data collected in the tables, the sum over the fixed points in $\PP^1 \times \PP^1$ in \eqref{sum} of those terms which contain $z_2$, and hence do not vanish,  reads as follows. There are $4$ such terms, corresponding to the vertical yellow and green arrays at the fixed points $B,C,D$ in Figure \ref{figure:fixedpoints}:
{\small \begin{multline}
\res_{z_1,z_2=\infty}\sum\limits_{F \in (\PP^1 \times \PP^1)^{T_\bz^2}} \frac{(z_1-z_2)\left(\Phi(V_F^{[1,2,3]})-\Phi(V_F^{[1,2],[3]})-\Phi(V_F^{[1,3],[2]})-\Phi(V_F^{[2,3],[1]})+2\Phi(V_F^{[1],[2],[3]})\right)}{\Euler^z(T_FN^3(\CC^2)) \prod_{i=1}^{n}(z_2-\lambda_i)\prod_{i=1}^{n}(z_1-\lambda_i) }=\\
\res_{z_1,z_2=\infty} \sum\limits_{F \in (\PP^1 \times \PP^1)^{T_\bz^2}} \frac{\Phi(V_{z_1,z_2}^{[1,2,3]})(z_1-z_2)}{\Euler^z(T_FN^3(\CC^2))\prod_{i=1}^{n}(z_2-\lambda_i)\prod_{i=1}^{n}(z_1-\lambda_i)}+\\
\res_{z_1,z_2=\infty} \frac{z_1-z_2}{\prod_{i=1}^{n}(z_2-\lambda_i)\prod_{i=1}^{n}(z_1-\lambda_i)} \left(\frac{\Phi(V_{z_2}^{[2,3],[1]})}{z_1(z_2-z_1)^2}+\frac{\Phi(V_{z_2}^{[1,3],[2]})}{z_1(z_2-z_1)^2}-\frac{\Phi(V_{z_2}^{[2,3],[1]})}{z_2(z_2-z_1)^2}-\frac{\Phi(V_{z_2}^{[1,3],[2]})}{z_2(z_2-z_1)^2}\right)=\\
\res_{z_1,z_2=\infty} \sum\limits_{F \in (\PP^1 \times \PP^1)^{T_\bz^2}} \frac{\Phi(V_{z_1,z_2}^{[1,2,3]})(z_1-z_2)}{\Euler^z(T_FN^3(\CC^2))\prod_{i=1}^{n}(z_2-\lambda_i)\prod_{i=1}^{n}(z_1-\lambda_i)}+2
\res_{z_1,z_2=\infty} \frac{\Phi(V^{\oplus 2} \oplus V(z_2))}{z_1z_2\prod_{i=1}^{n}(z_2-\lambda_i)\prod_{i=1}^{n}(z_1-\lambda_i)} 
\end{multline}}
Note that the factors $z_1-z_2$ cancel, hence the second residue vanishes: the total degree in $z_1$ of the rational expression is less than $-1$ if $m \ge 3$. The first term is the Atiyah-Bott contribution at the fixed point $A$ of the curvilinear component $\PP^1$. The key point is that the form $V_{z_1,z_2}^{[1,2,3]}=V\oplus V(z_1) \oplus V(z_2)=\pi^*V_{z_1,z_2}^{[1,2,3]}$ is a pull-back form from $\pi: N^3(\CC^n) \to \Hilb^3(\CC^n)=\CHilb^3(\CC^n)$. Hence we can integrate over the Hilbert scheme: 
\[{\small
\int\displaylimits_{\hat{N}^3(\CC^n)} \Phi^+= \res_{z_1,z_2=\infty}\sum\limits_{F \in \Hilb^3(\CC^2)^{T_\bz^2}}\frac{(z_1-z_2)\Phi(V\oplus V(z_1) \oplus V(z_2))dz_2dz_1}{\Euler^z(T_F(\Hilb^3(\CC^2)) \prod_{i=1}^{n}(z_2-\lambda_i)\prod_{i=1}^{n}(z_1-\lambda_i) }=\int\displaylimits_{\Hilb^3_\nabla(\CC^n)} \Phi(V_{z_1,z_2}^{[1,2,3]})
}\]
The next crucial observation is that the small neighborhood $\Hilb^3_\nabla(\CC^n)$ of $\Hilb^3_0(\CC^n)=\CHilb^3(\CC^n)$ can be modeled as a bundle over $\CHilb^3(\CC^n)$. More precisely, we will identify the normal bundle as a rank $2$ bundle $B$ over $\CHilb^3(\CC^n)$ such that the total space of $B$ is topologically isomorphic to $\Hilb^3_\nabla(\CC^n)$. We will see that $B$ is a tautological bundle over the punctual part, namely 
\[B|_{\CHilb^k(\CC^n)}=\calo_X^{[k]}/\calo_X\]
and hence the equivariant bundle over the fiber $\CHilb^3(\CC^2)$ over the reference flag $\bff$ has equivariant Euler class 
\[\Euler(B_{z_1.z_2})=z_1z_2\]
The Thom isomorphism and the iterated residue formula of \cite{bercziG&T,bsz} for integrals over the curvilinear component then gives 
\[\int_{\Hilb^3_\nabla(\CC^n)} \Phi=\int_{\CHilb^3(\CC^n)} \frac{\Phi}{\Euler(B)}=\res_{z_1=\infty} \res_{z_2=\infty}\frac{(z_1-z_2)\Phi(V\oplus V(z_1) \oplus V(z_2))dz_1dz_2}{(z_1z_2)(2z_1-z_2)\prod_{i=1}^{n}(z_2-\lambda_i)\prod_{i=1}^{n}(z_1-\lambda_i) }\]
After the substitution $z_i \to -z_i$ and the identity 
\[\frac{1}{\prod_{i=1}^{n}(z_1+\lambda_i)\prod_{i=1}^{n}(z_2+\lambda_i)}(\l_i \to \text{Chern roots of } X)=\frac{s_X(1/z_1)s_X(1/z_2)}{(z_1z_2)^m}\]
completes the proof of Proposition \ref{prop:k=3}.
\end{proof}

\subsection{The integral formula in general}\label{subsec:integralformula}

Let $X=\CC^n$ be the affine space with the standard $\GL(n)$ action which induces a $\GL(n)$ action on the Hilbert scheme $\GHilb^{k+1}(\CC^n)$. We formulate the equivariant version of Theorem \ref{mainthm}

\begin{theorem}[\textbf{Equivariant integrals over $\GHilb^{k}(\CC^n)$}]\label{equivariantintegral}
Let $\l_1,\ldots, \l_n$ denote the weights for the maximal torus $T \subset \GL(n)$ acting on $\CC^n$. Let $V$ be an equivariant bundle over $\CC^n$ with weights (i.e equivariant Chern root) $\theta_1, \ldots, \theta_r$ and $\Phi(V^{[k+1]})$ a tautological class given by the Chern polynomial $\Phi$. Then    
\[\int_{G\Hilb^{k}(\CC^n)}c_{d}(V^{[k]}) =\sum_{(\a_1,\ldots, \a_s) \in \Pi(k)} \int_{X^s} \res_{\bz^{\a_1}=\infty}\ldots \res_{\bz^{\a_s}=\infty} \mathcal{R}^\a(\theta_i, \bz) d\bz^{\a_s}\ldots d\bz^{\a_1}\]
where $\mathcal{R}^\a(\theta_i, \bz)$ is the rational form 
\[c_d(V(\bz^{\a_1}) \oplus \ldots \oplus V(\bz^{\a_s})) 
\prod_{l=1}^s \frac{(-1)^{|\a_l|-1} \prod_{1\le i<j \le |\a_l|-1}(z_i^{\a_l}-z_j^{\a_l})Q_{|\a_l|-1}d\bz^{\a_l}}
{\prod_{i+j\le l\le 
|\l_l|-1}(z_i^{\a_l}+z_j^{\a_l}-z_l^{\a_l})(z_1^{\a_l}\ldots z_{|\a_l|-1}^{\a_l})\prod_{i=1}^{|\a_l|-1} \prod_{j=1}^m
\left(\l_j-z_i^{\a_l}\right)}\]
with the same notations as in Theorem \ref{mainthm}.
\end{theorem}

\section{Local models for Hilbert schemes}\label{sec:model}


In this section let $X$ be a complex manifold of dimension $n$, and let $J_kX \to X$ be the bundle of $k$-jets of germs of parametrized curves in $X$; that is, $J_kX$ is the of equivalence classes of germs of holomorphic
maps $f:(\CC,0) \to (X,p)$, with the equivalence relation $f\sim g$
if and only if the derivatives $f^{(j)}(0)=g^{(j)}(0)$ are equal for
$0\le j \le k$ when computed in some local coordinate system of $X$ near $p\in X$. The projection map $J_kX \to X$ is simply $f \mapsto f(0)$. If we choose local holomorphic coordinates on an open neighbourhood $\Omega \subset X$
around $p$, the elements of the fibre $J_kX_p$ can be represented by Taylor expansions 
\[f(t)=p+tf'(0)+\frac{t^2}{2!}f''(0)+\ldots +\frac{t^k}{k!}f^{(k)}(0)+O(t^{k+1}) \]
 up to order $k$ at $t=0$ of $\CC^n$-valued maps $f=(f_1,f_2,\ldots, f_n)$
on open neighbourhoods of 0 in $\CC$. Locally in these coordinates the fibre $J_kX_p$ can be identified with the set of $k$-tuples of vectors $(f'(0),\ldots, f^{(k)}(0)/k!)=(\CC^n)^k$ 
which further can be identified with $J_k(1,m)$.

Note that $J_kX$ is not
a vector bundle over $X$ since the transition functions are polynomial but
not linear. In fact, let $\diff_X$ denote the principal $\diff_k(n)$-bundle over $X$ formed by all local polynomial coordinate systems on $X$. Then 
\[J_kX=\diff_X \times_{\diff_k(n)} J_k(1,n).\] 
is the associated bundle whose structure group is $\diff_k(n)$.

Let $J_k^{\reg}X$ denote the bundle of $k$-jets of germs of parametrized regular curves in $X$, that is, where the first derivative $f'\neq 0$ is nonzero. After fixing local coordinates near $p\in X$ the fibre $J_k^{\reg}X_p$ can be identified with $\jetreg 1n$ and 
\[J_k^{\reg}X=\diff_X \times_{\diff_k(n)} J_k^\reg(1,n).\]

Let $\cald_X^{\le k}$ denote the dual bundle formed by at most $k$th order differential operators over $X$. Then $\cald_X^{\le 0}=\calo_X$, and we let $\cald_X^k=\cald_X^{\le k}/\cald_X^{\le 0}$. We have a filtration of bundles 
\begin{equation}\label{dfiltration}
\calo_X=\cald_X^{\le 0} \subset \cald_X^{\le 1} \subset \ldots \subset \cald_X^{\le k}
\end{equation}
where the graded component $\cald_X^{\le i}/\cald_X^{\le i-1}\simeq \Sym^iTX$ but this filtration is not split in general, so $\cald_X^{\le k} \nsimeq \Sym^{\le k}TX$. After choosing local coordinates on $X$ near $p$, the fibre $\cald^{k}_{X,p}$ can be identified with the space $J_k(m,1)^*\simeq \symdot$ of $k$th order differential operators on $\CC^n$.
Therefore $\cald_X^k$ is the associated bundle 
\[\cald_X^k=\diff_X \times_{\diff_k(n)} J_k(n,1)^*=\diff_X \times_{\diff_k(n)} \symdot.\]
Given a regular $k$-jet $f:(\CC,0) \to (X,p)$ we may push forward the differential operators of order k on $\CC$ to $X$ along $f$ and obtain a $k$-dimensional subspace of $\cald^{\le k}_{X,p}$. This gives the bundle map 
\begin{equation}\label{diffoppushforward}
J_k^\reg X \to \grass_k(\cald_X^{k})
\end{equation} 
Note that $\diff_k(1)=\jetreg 11$ acts fibrewise on the jet bundle $J_k^{\reg}X$ and the map \eqref{diffoppushforward} is $\diff_k(1)$-invariant resulting an embedding 
\begin{equation}\label{invdiffoppush}
\tilde{\phi}: J_k^{\reg}X/\diff_k(1) \hookrightarrow \grass_k(\cald_X^{k})
\end{equation}
This is the fibered version of the map 
\begin{equation}\label{embedgrass1}
\phi:J_k(1,n)/\diff_k(1) \hookrightarrow \grass(k,J_k(n,1)^*)
\end{equation}
\begin{theorem}[\textbf{The test curve model of curvilinear Hilbert schemes}, \cite{bsz}, \cite{b}]\label{bszmodel}
\begin{enumerate}
\item For any $k,n$ we have 
\[\CHilb^{k+1}_0(\CC^n)=\overline{\mathrm{im}(\phi)} \subset \grass_k(J_k(n,1)^*)=\grass_k(\mathfrak{m}),\]
where $\mathfrak{m}$ is the maximal ideal at the origin. 
\item Let $\{e_1,\ldots, e_n\}$ be a basis of $\CC^n$. For $k\le n$ the $\GL(n)$-orbit of 
\[p_{k,n}=\phi(e_1,\ldots, e_k)=\mathrm{Span}_\CC (e_1,e_2+e_1^2,\ldots, \sum_{i_1+\ldots +i_r=k}e_{i_1}\ldots e_{i_r})\] 
forms a dense subset of the image $\jetreg 1n$ and
\[\CHilb^{k+1}_0(\CC^n)=\overline{\GL_n \cdot p_{k,n}}.\] 
\item There are two canonically defined bundles on $\CHilb^{k+1}_0(\CC^n)$: 
\begin{enumerate}
\item The tautological rank $k+1$ bundle $\calo_{\CC^n}^{[k+1]}$. 
\item The tautological rank $k$ bundle $\cale$, which is the restriction of the tautological bundle over $\grass_k(\wedge^k \symdot)$. 
\end{enumerate}
These fit into the short exact sequence 
\[\xymatrix{0 \ar[r] & \calo_\grass \ar[r] & \calo_{\CC^n}^{[k+1]} \ar[r] & \cale \ar[r] & 0}.\]
\item Similarly, we get an embedding of the curvilinear Hilbert scheme 
\[\CHilb^{k+1}(X)=\overline{\mathrm{im}(\tilde{\phi})} \subset \grass_k(\cald_X^{k}),\] 
and for a bundle $V$ over $X$ we obtain the short exact sequence of bundles over $\CHilb^{k+1}(X)$:
\[\xymatrix{0 \ar[r] & V \ar[r] & V^{[k+1]}=\calo_{X}^{[k+1]}\otimes V \ar[r] & \cale_X \ar[r] & 0}\]
where $\cale_X$ is the tautological bundle over $\grass_k(\cald_X^{k})$.
\end{enumerate}
\end{theorem}

\subsection{Fibration over the flags in $TX$.}

Let $k\le n$ and let $P_{n,k} \subset \GL_n$ denote the parabolic subgroup which preserves the flag 
\[\mathbf{f}=(\mathrm{Span}(e_1)   \subset \mathrm{Span}(e_1,e_2) \subset \ldots \subset \mathrm{Span}(e_1,\ldots, e_k)=\CC_{[k]} \subset \CC^n).\] 
\begin{definition}\label{def:xktilde}
Define the partial desingularization 
\[\widehat{\CHilb}^{k+1}_0(\CC^n)=\GL_n \times_{P_{n,k}} \overline{P_{n,k} \cdot p_{k,n}}\]
with the resolution map $\rho: \widehat{\CHilb}^{k+1}_0(\CC^n) \to \CHilb^{k+1}_0(\CC^n)$ given by $\rho(g,x)=g\cdot x$. The test curve model defines the embedding
\[\xymatrix{\widehat{\CHilb}^{k+1}_0(\CC^n) \ar@{^{(}->}[r]^-\rho \ar[d]^\rho & \widehat{\grass}_k(\sym^{\le k} \CC^n) \ar[ld] \\ \flag_k(\CC^n) & }\]
where $\widehat{\grass}_k(\sym^{\le k} \CC^n)=\GL_n \times_{P_{n,k}} \grass_k(\sym^{\le k}\CC_{[k]})$ and $\flag_k(\CC^n)=\GL(n)/P_{n,k}$ is the flag variety.
\end{definition} 
\begin{remark} The fiber $\rho^{-1}(\bff)$ consists of limit subschemes $\lim_{l \to \infty} \xi^l$ where $\xi^l=\{p_1^l \sqcup \ldots \sqcup p_{k+1}^l\}$ 
is non-reduced such that $\lim_{l \to \infty} \Span(p_1,\ldots, p_i)=\Span(e_1,\ldots, e_{i-1})$ for $2\le i \le k+1$. 
\end{remark}
Equivalently, let $\jetnondeg 1n \subset \jetreg 1n$ be the set of test curves with $\g',\ldots, \g^{(k)}$ linearly independent. These correspond to the nonsingular $n \times k$ matrices in $\Hom(\CC^k,\CC^n)$, and they fibre over the set of complete flags in $\CC^n$:
\begin{equation}\label{proj2}
\jetnondeg 1n/\diff_k(1) \to \Hom(\CC^k,\CC^n)/B_k=\flag_k(\CC^n)
\end{equation}
where $B_k \subset \GL(k)$ is the upper Borel. The image of the fibres under $\phi$ are isomorphic to $P_{n,k} \cdot p_k$, and therefore $\widehat{\CHilb}^{k+1}_0(\CC^n)$ is the fibrewise compactification of $\jetnondeg 1m$ over $\flag_k(\CC^n)$. 

We construct a fibred version of $\widehat{\CHilb}^{k+1}_0(\CC^n)$, a fibrewise partial desingularization 
\begin{equation}\label{desing}
\rho:\widehat{\CHilb}^{k+1}(X) \to \CHilb^{k+1}(X)
\end{equation}
over $X$, where $\widehat{\CHilb}^{k+1}(X)$ is a locally trivial bundle over $X$ with fibres isomorphic to $\widehat{\CHilb}^{k+1}_0(\CC^n)$. 
The immediate problem is that $J_k^{\reg}X/\diff_k(1)$ is embedded into $\flag_k(\cald_X^k)$ but there is no projection $\cald_X^k \to TX$ to define a fibration of $J_k^{\reg}X/\diff_k(1)$ over $\flag_k(TX)$. To resolve this problem we will work with a linearised bundle $\CHilb^{k+1}(X)^{\GL}$ associated to a principal $\GL_n$-bundle over $X$, rather than $\CHilb^{k+1}(X)$ which is associated to the principal $\diff_k(1)$-bundle $\diff_X$. This way we reduce the structure group of $\CHilb^{k+1}(X)$ to $\GL_n$, but the new space has the same topological intersection numbers. We explain the details in the next subsection. 

\subsection{Linearisation of the Hilbert scheme}

Recall from Step 1 the notation $\diff_k(n)=J_k^\reg(n,n)$ for the group of $k$th order diffeomorphism germs of $\CC^n$ at the origin. In Step 1 we introduced the bundle of differential operators
\[\cald_X^k=\diff_X \times_{\diff_k(n)} \symdot.\]
where $\diff_X$ stands for the principal $\diff_k(n)$-bundle over $X$ formed by all local polynomial coordinate systems on $X$. This is not a vector bundle---the structure group is $\diff_k(n)$---but we can linearise it. 
The set $\GL(n)$ of linear coordinate changes forms a subgroup of $\diff_k(n)$. Let $\GL_X$ denote the principal $\GL(n)$-bundle over $X$ formed by all local linear coordinate systems on $X$. Then the vector bundle $\sym^{\le k} TX=\oplus_{i=1}^k \sym^i TX$ is associated to the same $\symdot$ considered as a $\GL(n)$-module:
\[\sym^{\le k} TX=\GL_X \times_{\GL(n)} \symdot.\]
Note that $\cald_X^{k}$ and  $\Sym^{\le k}TX$ are not isomorphic bundles, and in particular the filtration defined in \eqref{dfiltration} does not split. Hence there is no projection map $\cald_X^k \to TX$ but there is a natural projection $\sym^{\le k} TX \to TX$.

There is an induced $\diff_k(n)$ action on $\grass_k(\symdot)$ and the image $\CHilb^{k+1}_0(\CC^n)=\overline{\im(\phi^\grass)}$ is $\diff_k(n)$-invariant subvariety of $\grass_k(\symdot)$. Then $\CHilb^{k+1}(X)$ is the associated bundle
\[\CHilb^{k+1}(X)=\diff_X \times_{\diff_k(n)} \CHilb^{k+1}_0(\CC^n) \subset \diff_X \times_{\diff_k(n)} \grass_k(\symdot)=\grass_{k}(\cald^k_X).\]
We define the corresponding linearised bundle 
\[\CHilb^{k+1}(X)^{\GL}=\GL_X \times_{\GL_n} \CHilb^{k+1}_0(\CC^n) \subset \GL_X \times_{\GL_n}\grass_k(\symdot)=\grass_k(\sym^{\le k} TX)\]
which is the linearised version of $\CHilb^{k+1}(X)$ remembering the linear action on the fibres. 
For torus localisation purposes we can replace $\CHilb^{k+1}(X)$ with its linearised version $\CHilb^{k+1}(X)^{\GL}$. Similarly to Definition \ref{def:xktilde} we have
\begin{definition}\label{def:widetilde} Define the partial desingularisation 
\[\widehat{\CHilb}^{k+1}(X)=\GL_X \times_{P_{n,k}}  \overline{P_{n,k} \cdot p_{k,n}}\]
and 
\[\widehat{\grass}_k(\sym^{\le k} TX)=\GL_X \times_{P_{n,k}} \grass_k(\sym^{\le k}\CC_{[k]})\]
with the partial resolution map 
\[\rho: \widehat{\CHilb}^{k+1}(X)=\GL_X \times_{\GL(n)} \left(\GL(n) \times_{P_{n,k}} \overline{P_{n,k}\cdot \mathfrak{p}_{m,k}}\right) \to \GL_X \times_{\GL(n)} \left(\overline{\GL(n)\cdot \mathfrak{p}_{m,k}}\right)=\CHilb^{k+1}(X)^{\GL}.\]
This results the diagram 
\[\xymatrix{\widehat{\CHilb}^{k+1}(X) \ar@{^{(}->}[r]^-\rho \ar[d]^\rho & \widehat{\grass}_k(\sym^{\le k} TX) \ar[ld] \\ \flag_k(TX) \ar[d]^\pi & \\ X & }\]
\end{definition}

\subsection{Neighborhood of $\CHilb^{k+1}(X)$ in $\GHilb^{k+1}(X)$} We start with $X=\CC^n$, and how to construct the model of a small neighborhood in $\GHilb^{k+1}(\CC^n)$ of the curvilinear locus $\CHilb^{k+1}(\CC^n)$ supported at some point on $\CC^n$. 
\begin{definition} Let $\xi=\xi_1 \sqcup \xi_2 \sqcup \ldots \sqcup \xi_s \in \GHilb^{k+1}(\CC^n)$ be a subscheme whose support consists of the $s$ points $p_1,\ldots, p_s \in \CC^n$ with $\mathrm{length}(\xi_i)=l_i$. The baricenter of $\xi$ is 
\[\bbb(\xi)=l_1p_1+\ldots +l_sp_s \in \CC^n\]
The balanced Hilbert scheme centered/balanced at $p\in \CC^n$ is 
\[\BHilb^{k+1}_p(\CC^n)=\bbb^{-1}(p) \subset \GHilb^{k+1}(\CC^n)\]
consists of all subschemes $\xi$ whose baricenter is $p$.
\end{definition}
Let $\{0\} \in U \subset \CC^n$ be a small open neighborhood of the origin and let 
\[\BHilb^{k+1}(U)=\{\xi \in \BHilb^{k+1}_{\{0\}}(\CC^n): \mathrm{supp}(\xi) \subset U\}\]
be a small neighborhood of $\CHilb^{k+1}_0(\CC^n)$ in $\BHilb^{k+1}_{\{0\}}(\CC^n)$. 
We construct a blow-up $\widehat{\BHilb}^{k+1}(U)$ of $\BHilb^{k+1}(U)$ which fibers above the flag $\flag_k(\CC^n)$. We follow a similar path to the construction of $\widehat{\CHilb}^{k+1}_0(\CC^n)=\GL_n \times_{P_{n,k}} \overline{P_{n,k} \cdot p_k}$. We use the following simple observation.
\begin{lemma}\label{keylemma}
Every point $\xi \in \BHilb^{k+1}(\CC^n)$ is contained in a $k$-dimensional linear subspace of $\CC^n$.
\end{lemma}
\begin{proof} Any subscheme $\xi \in \BHilb^{k+1}(\CC^n)$ is the limit $\xi=\lim_{l \to \infty} \xi^l$ of reduced subschemes 
\[\xi^l=p_1^l \sqcup \ldots \sqcup p^l_{k+1}\]
with $p_1+\ldots +p_{k+1}=0$. Hence $V^l=\mathrm{Span}(p_1,\ldots, p_{k+1})$ is a $k$-dimensional subspace in $\CC^n$, and $\xi \in  \lim_{l\to \infty} V^l$.

\end{proof}

Now let 
\[\mathbf{f}=(\mathrm{Span}(e_1)   \subset \mathrm{Span}(e_1,e_2) \subset \ldots \subset \mathrm{Span}(e_1,\ldots, e_k)=\CC_{[k]} \subset \CC^n).\] 
be our base flag as before and recall 
\[\CHilb^{k+1}_\ff(\CC^n)=\rho^{-1}(\ff)=\overline{P_{k,n}\cdot \mathfrak{p}_{k,n}}  \subset \grass_k(\sym^{\le k}\CC_{[k]})\]
is the part of the curvilinear Hilbert scheme which sits over $\bff$. We define similarly 
{\small \[\BHilb^{k+1}_\bff(U)=\overline{\{(p_1 \sqcup \ldots \sqcup p_{k+1}) \in \BHilb^{k+1}(U): \Span(p_1,\ldots, p_i)=\Span(e_1,\ldots, e_{i-1}) \text{ for } 2\le i \le k+1\}}\]}
be the space of those subschemes in the neighborhood $U$ which sit in the flag $\bff$.  For a complex vector space $V$ let $\grass_{k+1}(S^\bullet V)$ denote the Grassmannian of $k+1$-dimensional quotients of $S^\bullet V=\oplus_{i=0}^k \sym^kV^*$. Then 
\[\BHilb^{k+1}_\bff(U) \subset \grass_{k+1}(S^\bullet \CC_{[k]})\]
is $P_{k,m}$-invariant, and we arrive at
\begin{definition}\label{def:hatbhilb} We define the partial resolution 
\[\widehat{\BHilb}^{k+1}(U)=\GL(n) \times_{P_{n,k}} \BHilb^{k+1}_\bff(U)\]
which fibers over $\flag_k(\CC^n)=\GL(n) / {P_{n,k}}$, and 
\[\widehat{\grass}_{k+1}(S^\bullet \CC^n)=\GL_X \times_{P_{n,k}} \grass_{k+1}(S^\bullet \CC_{[k]})\]
which fit into the diagram 
\begin{equation}\label{diagramtwo}
\xymatrix{
\widehat{\CHilb}^{k+1}(\CC^n) \ar@{^{(}->}[r] \ar[d]^\rho &  \widehat{\BHilb}^{k+1}(U) \ar[ld] \ar@{^{(}->}[r]^-{\iota} & \widehat{\grass}_{k+1}(S^\bullet \CC^n) \ar[lld]^{\rho_\grass}\\
\flag_k(\CC^n)  & } 
\end{equation}
and by Lemma \ref{keylemma} $\rho: \widehat{\BHilb}^{k+1}(U) \to \BHilb^{k+1}(U)$ is surjective rational map (generically 1-1). 
\end{definition}

Finally, the same argument works over a smooth manifold $X$: we let $\bU \subset TX$ be a small $\GL_X$-invariant tubular neighborhood of the zero section, with the exponential map $\exp: \bU \to X$, which identifies $U_x$ with $\exp(\bU_x) \subset X$. We may assume that $\bU=\GL_X \times_{\GL(n)} U$ for some $U \subset \CC^n$. We define 
\[\widehat{\BHilb}^{k+1}(\bU)=\GL_X \times_{P_{n,k}} \BHilb^{k+1}_\bff(U)\]
We get the diagram
\begin{equation}\label{diagramfour}
\xymatrix{
\widehat{\CHilb}^{k+1}(X) \ar@{^{(}->}[r] \ar[d]^\rho &  \widehat{\BHilb}^{k+1}(\bU) \ar[ld]^{\rho_{\BHilb}} \ar@{^{(}->}[r]^-{\iota} & \grass_{k+1}(S^\bullet TX) \ar[lld]^{\rho_\grass}\\
\flag_k(TX) \ar[d]^\mu & &\\
X &&} 
\end{equation}
with a birational morphism $\widehat{\BHilb}^{k+1}(\bU) \to \BHilb^{k+1}(\bU)=\cup_{p\in X} \BHilb^{k+1}(U_p)$. 
\begin{remark} The fiber $\rho^{-1}(f)$ over a flag 
\[f=(F_1 \subset F_2 \subset \ldots \subset F_k \subset T_pX)\]
sitting over the point $p\in X$ consists of limit subschemes $\lim_{l \to \infty} \xi^l$ where $\xi^l=\{p_1^l \sqcup \ldots \sqcup p_{k+1}^l\}$ is non-reduced such that $p_i \in U_p$ and $\lim_{l \to \infty} \Span(p_1,\ldots, p_i)=F_{i-1}$ for $2\le i \le k+1$. 
\end{remark}

We obtain the following relation among the tautological bundles
 \[\calo_X^{[k+1]}=\iota^* \cale^{k+1}\]
 and for a bundle $V$ on $X$ the restriction of the tautological bundle to the curvilinear locus is 
 \begin{equation}\label{taurestrictedtocurvi}
 V^{[k+1]}|_{\widehat{\CHilb}^{k+1}(X)}=V \otimes \calo_X^{[k+1]}
 \end{equation}

\section{Fully nested Hilbert schemes}\label{sec:nested}

The central object in our proof of the main theorems is a generalisation of the standard nested Hilbert schemes of points. Recall the space  
\[\Hilb^{[1,2,\ldots, k]}(X)=\{(\xi_1 \subset \xi_2 \subset \ldots \subset \xi_k)| \xi \subset X, \dim H^0(\xi_i)=i\} \subset X \times \Hilb^2(X) \times \ldots \times \Hilb^k(X)\] 
which is formed by flags of subschemes is called the nested Hilbert scheme on $X$. Our fully nested Hilbert scheme defined in this section is associated to the ordered Hilbert scheme of $k$ labeled points, and it keeps track degenerations of any subset of the $k$ labeled points.
\subsection{Approximating bundles}
The ordered Hilbert scheme of points $\Hilb^{[k]}(X)$ is a branched cover of the ordinary Hilbert scheme, defined by the commutative diagram
\begin{equation*}
\xymatrix{\Hilb^{[k]}(X) \ar[r] \ar[d]^{\HC} &  \Hilb^k(X) \ar[d]^{\HC} \\
 X^k \ar[r]  & \Sym^k(X) } 
\end{equation*}
where the vertical arrows are the Hilbert-Chow morphisms taking a subscheme/ordered subscheme $Z$ to its support cycle. This is a ramified cover of $\Hilb^k(X)$ with a stratified ramification locus sitting over the diagonals of $\Sym^k(X)$. Define the bundle $F^{[k]}$ on $\Hilb^{[k]}(X)$ as the pullback of $F^{[k]}$ along $\Hilb^{[k]}(X) \to \Hilb^k(X)$. Note that we loosely use the same notation for $F^{[k]}$ on both $\Hilb^k(X)$ and $\Hilb^{[k]}(X)$.

Following Li \cite{junli} and Rennemo \cite{rennemo}, we define for every partition $\alpha \in \Pi(k)$ of $\{1,\ldots, k\}$ the scheme $\Hilb^{[\alpha]}(X)$ which is a certain approximation of  $\Hilb^{[k]}(X)$. We fix some notation and conventions first. For simplicity, let $\RHilb^{[k]}(X) \subset \GHilb^{[k]}(X)$ denote the open set of reduced subschemes of the form $\xi=x_1 \sqcup \ldots \sqcup x_k$, formed by $k$ different points on $X$.

\begin{definition}\label{def:approximatingsets} Let $\a=(\a_1,\ldots, \a_s) \in \Pi(k)$ be a partition of $\{1,\ldots, k\}$. 
\begin{enumerate}
\item We let $\sim_\alpha$ be the equivalence relation on $\{1,\ldots, k\}$ given by letting the
elements of $\alpha$ form the equivalence classes.
We introduce the partial order by setting $\alpha \le \beta$ if $\sim_\a$ is a refinement of $\sim_\beta$.
Let 
\[\Delta_\alpha=\{(x_1,\ldots, x_k) \in X^k | x_i=x_j \text{ if } i\sim_\a j\}.\]
denote the closed diagonal, then $\Delta_\b \subseteq \Delta_\a$ whenever $\a \le \b$.
\item Let 
\[\Hilb^{[\a]}(X)=\prod_{i=1}^s \Hilb^{[\a_i]}(X),\]
where for a subset $A \subset \{1,\ldots, k\}$, $\Hilb^{[A]}(X)$ denotes the ordered Hilbert scheme of $|A|$ points labeled by $A$. The punctual part sits over the corresponding diagonal:
\[\Hilb^{[\a]}_0(X)=\prod_{i=1}^s \Hilb^{[\a_i]}_0(X)=\HC^{-1}(\Delta_\a).\]
\item We define
\[\GHilb^{[\a]}(X) = \prod_{i=1}^s \GHilb^{[\a_i]}(X)\]
which is the the closure of $\RHilb^{[k]}(X)$ in $\Hilb^{[\a]}(X)$. Like above, the punctual part is $\GHilb^{[\a]}_0(X)=\HC^{-1}(\Delta_\a)$. We will often use the simpler notation $\GHilb^{k}(X)=\GHilb^{[\Lambda]}(X)$ where $\Lambda=\{1,\ldots, k\}$ is the trivial partition, that is, the main component of the ordered Hilbert scheme on $k$ points. 
\end{enumerate}
\end{definition}
The fully nested Hilbert scheme $\N^k(X)$ parametrises ordered collections of $k$ points in $X$, with the additional data that when $l$ points with labels in the same set in the partition $\a$ come together at $x$, one must specify a length $l$ subscheme supported at $x$. The master blow-up space encodes information on how different subsets of points collide. 

\begin{definition}  Let $X$ be a complex nonsingular manifold and $k\ge 1$. \begin{enumerate} 
\item The fully nested Hilbert scheme is $\N^k(X)=\overline{im(h)}$, the closure of the image of the natural map
\[h: \RHilb^{[k]}(X) \to \prod_{\a \in \Pi(k)} \GHilb^{[\a]}(X).\]
The Hilbert-Chow morphism extends to $\N^k(X)$ and gives a morphism $\HC: \N^k(X) \to X^k$, hence obtain the $\a$-punctual locus $\N^\a_0(X)=\HC^{-1}(\Delta_\a)$. The projection $\pi_\a:\N^k(X) \to \GHilb^{[\a]}(X)$ fits into the diagram
\begin{equation}\label{commdiagram}
\xymatrix{\N^k(X) \ar[r]^-{\pi_\a} \ar[d]^{\HC} & \GHilb^{[\a]}(X) \ar[dl]^{\HC_\a} \\
X^k 
}
\end{equation}
\item Let $\a=(\a_1,\ldots, \a_s) \in \Pi(k)$ be a partition. Pull-back along the projection map $\pi_\a: \N^k(X)  \to \GHilb^{[\a]}(X)$ defines an approximating bundle $\pi_\a^*(F^{[\a]})$ over $\N^k(X)$, which satisfies 
 \[\pi_\a^*F^{[\a]}=\pi_{\a_1}^*F^{[[\a_1]]} \oplus \ldots \oplus \pi_{\a_s}^*F^{[[\a_s]]}.\]
We will loosely use the shorthand notation $F^{[\a]}$ for the bundle $\pi_\a^*(F^{[\a]})$. 
\item The punctual part, which sits over the $\a$-diagonal, fits into the diagram 
\begin{equation}\label{commdiagram2}
\xymatrix{\N^\a_0(X) \ar[r]^{\pi_\a} \ar[d]^{\HC} & \GHilb^{[\a]}_0(X) \ar[dl]^{\HC_\a} \\
\Delta_\a 
}
\end{equation}
and  $F^{[\a]}_0=\pi_\a^*(F^{[\a]}|_{\GHilb^{[\a]}_0(X)})$ its restriction to the punctual part $\N^\a_0(X)$. 
\end{enumerate}
\end{definition}

The terminology is self-explanatory: $\N^k(X)$ can be considered as a nested Hilbert scheme, but nested with respect to the full partially ordered net of subsets of $\{1,\ldots, k\}$. Note that $N^k(X)$ is irreducible, being the closure of the image under $h$ of an irreducible variety. It is a blow-up of $\GHilb^k(X)$ via the natural projection map 
\[\pi_\L: \N^k(X) \to \GHilb^k(X)\] 
where $\L=\{1,\ldots, k\}$ is the trivial partition. However, the geometry of the restriction $\pi_\L: \N^k_0(X) \to \GHilb^k_0(X)$
to the punctual part is more subtle. In particular, the preimage of the curvilinear component $\CHilb^k(X)$ is not necessarily irreducible which is a delicate part of our integration argument, addressed in the next section.  

\begin{definition}\label{def:curvgeometricsubsetQ} The curvilinear part of $\N^k(X)$ is $\CN^k(X)=\pi_\L^{-1}(\CHilb^{[k]}(X))$ where $\L=\{1,\ldots, k\}$ is the trivial partition. The punctual curvilinear component supported at $p\in X$ is 
\[\CN^k_p(X)=\pi_\L^{-1}(\CHilb^{[k]}_p(X))=\CN^k(X) \cap \HC^{-1}(\{p,\ldots, p\}).\]
\end{definition}

\begin{example} We have seen in \S \ref{subsec:k=3} that $CN^3_0(\CC^2)$ is not irreducible, and  it has $2$ components: the curvilinear component $\CN^3_{main}(\CC^2)$ of dimension $1$, and an other component $CN^3_{Por}=\PP^1 \times \PP^1$ sitting over the Porteous point $I=\mathfrak{m}^2 \in \CHilb^3_0(\CC^2)$. 
\end{example}

Despite these extra components, the projection $\pi^\L$ is isomorphism over the curvilinear locus in $\CHilb^{k}(X)$. Recall this is defined as
\[\mathrm{Curv}^k(X)=\{\xi \in \Hilb^k_0(X): \calo_\xi \simeq \CC[t]/t^k\},\]
and it is a dense open subset of $\CHilb^k(X)$. 


\subsection{Fibration of the fully nested Hilbert scheme over the flag manifold}

Recall the curvilinear part of $\N^{k}(X)$, defined as $\CN^{k}(X)=\pi_\L^{-1}(\CHilb^{k}(X))$ where $\L=\{0,1,\ldots, k\}$ is the trivial partition. The punctual curvilinear component supported at $p\in X$ is 
\[\CN^{k}_p(X)=\pi_\L^{-1}(\CHilb^{k}_p(X))=\CN^{k}(X) \cap \HC^{-1}(\{p,\ldots, p\}).\]

We follow the same argument which resulted in diagram \eqref{diagramfour}. Let $\bU \subset TX$ be a small $\GL_X$-invariant tubular neighborhood of the zero section, with the exponential map $\exp: \bU \to X$, which identifies $U_x$ with $\exp(\bU_x) \subset X$. We may assume that $\bU=\GL_X \times_{\GL(n)} U$ for some $U \subset \CC^n$. We define 
\[\widehat{\BN}^{k}(\bU)=\pi_\Lambda^{-1}(\widehat{\BHilb}^{k}(\bU)).\]
We get the following extension of diagram \eqref{diagramfour}:
\begin{equation}\label{diagramfive}
\xymatrix{
\widehat{\CN}^{k}(X) \ar@{^{(}->}[r]  \ar[d]^{\pi_\Lambda} & \widehat{\BN}^{k}(\bU) \ar[d]^{\pi_\Lambda}  \ar@{^{(}->}[r]^-{\iota} & \prod\limits_{(\a_1,\ldots, \a_s) \in \Pi(k)} \prod\limits_{i=1}^s \widetilde{\grass}_{|\a_i|}(S^\bullet TX) \ar[d]^{\pi_\Lambda}\\ 
\widehat{\CHilb}^{k}(X) \ar@{^{(}->}[r] \ar[d]^\rho &  \widehat{\BHilb}^{k}(\bU) \ar[ld]  \ar@{^{(}->}[r]^-{\iota} & \widehat{\grass}_{k}(S^\bullet TX) \ar[lld]^{\rho_{\grass}} \\
\flag_{k-1}(TX) \ar[d]^\mu & \\
X &} 
\end{equation}
and $\widehat{\BN}^{k}(\bU) \to \widehat{\BHilb}^{k}(\bU)$ is a birational morphism. When $X=\CC^n$, all vertical maps in diagram \eqref{diagramfive} are $\GL(n)$-equivariant, and we can take $\bU=T\CC^n$. 

\begin{remark}\label{fiberdim}
Note that for any $p\in X$ any any flag $\bff \in \flag_{k-1}(T_pX)$ the fiber $(\rho_\grass)^{-1}(\bff)$ is smooth and its dimension only depends on $k$, not on $n$. 
\end{remark}

\section{Equivariant localisation and multidegrees}\label{sec:equiv}

This section is a brief of equivariant cohomology and localisation. For
more details, we refer the reader to Berline--Getzler--Vergne \cite{bgv} and B\'erczi--Szenes \cite{bsz}. 

Let $\kt\cong U(1)^m$ be the maximal compact subgroup of
$T\cong(\CC^*)^m$, and denote by $\mathfrak{t}$ the Lie algebra of $\kt$.  
Identifying $T$ with the group $\CC^n$, we obtain a canonical basis of the weights of $T$:
$\lambda_1,\ldots ,\lambda_m\in\mathfrak{t}^*$. 

For a manifold $X$ endowed with the action of $\kt$, one can define a
differential $d_\kt$ on the space $S^\bullet \mathfrak{t}^*\otimes
\Omega^\bullet(X)^\kt$ of polynomial functions on $\mathfrak{t}$ with values
in $\kt$-invariant differential forms by the formula:
\[   
[d_\kt\alpha](X) = d(\alpha(X))-\iota(X_X)[\alpha(X)],
\]
where $X\in\mathfrak{t}$, and $\iota(X_X)$ is contraction by the corresponding
vector field on $X$. A homogeneous polynomial of degree $d$ with
values in $r$-forms is placed in degree $2d+r$, and then $d_\kt$ is an
operator of degree 1.  The cohomology of this complex--the so-called equivariant de Rham complex, denoted by $H^\bullet_T(X)$, is called the $T$-equivariant cohomology of $X$. Elements of $H_T^\bullet (X)$ are therefore polynomial functions $\mathfrak{t} \to \Omega^\bullet(X)^K$ and there is an integration (or push-forward map) $\int: H_T^\bullet(X) \to H_T^\bullet(\mathrm{point})=S^\bullet \mathfrak{t}^*$ defined as  
\[(\int_X \alpha)(X)=\int_X \alpha^{[\mathrm{dim}(X)]}(X) \text{ for all } X\in \mathfrak{t}\]
where $\alpha^{[\mathrm{dim}(X)]}$ is the differential-form-top-degree part of $\alpha$. The following proposition is the Atiyah-Bott-Berline-Vergne localisation theorem in the form of \cite{bgv}, Theorem 7.11. 
\begin{theorem}[(Atiyah-Bott \cite{atiyahbott}, Berline-Vergne \cite{berlinevergne})]\label{abbv} Suppose that $X$ is a compact complex manifold and $T$ is a complex torus acting smoothly on $X$, and the fixed point set $X^T$ of the $T$-action on X is finite. Then for any cohomology class $\a \in H_T^\bullet(X)$
\[\int_X \alpha=\sum_{f\in X^T}\frac{\a^{[0]}(f)}{\mathrm{Euler}^T(T_fX)}.\]
Here $\mathrm{Euler}^T(T_fX)$ is the $T$-equivariant Euler class of the tangent space $T_fX$, and $\alpha^{[0]}$ is the differential-form-degree-0 part of $\alpha$. 
\end{theorem}

The right hand side in the localisation formula considered in the fraction field of the polynomial ring of $H_T^\bullet (\mathrm{point})=H^\bullet(BT)=S^\bullet \mathfrak{t}^*$ (see more on details in Atiyah--Bott \cite{atiyahbott} and \cite{bgv}). Part of the statement is that the denominators cancel when the sum is simplified.

\subsection{Equivariant Poincar\'e duals and multidegrees}
\label{subsec:epdmult} 

Restricting the equivariant de Rham complex to compactly supported (or quickly
decreasing at infinity) differential forms, one obtains the compactly
supported equivariant cohomology groups $ H^\bullet_{\kt,\mathrm{cpt}}(X)
$. Clearly $H^\bullet_{\kt,\mathrm{cpt}}(X) $ is a module over
$H^\bullet_\kt(X)$. For the case when $X=W$ is an $N$-dimensional
complex vector space, and the action is linear, one has
$H^\bullet_\kt(W)= S^\bullet\mathfrak{t}^*$ and $ H^\bullet_{\kt,\mathrm{cpt}}(W) $ is
a free module over $H^\bullet_\kt(W)$ generated by a single element of
degree $2N$:
\begin{equation}
  \label{thomg}
   H^\bullet_{\kt,\mathrm{cpt}}(W) = H^\bullet_{\kt}(W)\cdot\Thom_{\kt}(W)
\end{equation}

Fixing coordinates $y_1,\dots,y_N$ on $W$, in which the $T$-action is
diagonal with weights $\eta_1,\ldots,  \eta_N$, one can write an explicit
representative of  $\Thom_{\kt}(W)$ as follows:
\[   \Thom_{\kt}(W) = 
e^{-\sum_{i=1}^N|y_i|^2}\sum_{\sigma\subset\{1,\ldots , N\}}
\prod_{i\in\sigma}\eta_i/2\cdot\prod_{i\notin \sigma}dy_i\,d\bar y_i
\]

We will say that an algebraic variety has dimension $d$ if its
maximal-dimensional irreducible components are of dimension $d$.  A
$T$-invariant algebraic subvariety $\Sigma$ of dimension $d$ in $W$
represents $\kt$-equivariant $2d$-cycle in the sense that
\begin{itemize}
\item a compactly-supported equivariant form $\mu$ of degree $2d$ is
  absolutely integrable over the components of maximal dimension of
  $\Sigma$, and $\int_\Sigma\mu\in S^\bullet \mathfrak{t}$;
\item if $d_\kt\mu=0$, then $\int_\Sigma\mu$ depends only on the class
  of $\mu$ in $ H^\bullet_{\kt,\mathrm{cpt}}(W) $,
\item and $\int_\Sigma\mu=0$  if $\mu=d_\kt\nu$ for a
  compactly-supported equivariant form $\nu$.
\end{itemize}

\begin{definition} \label{defepd} Let $\Sigma$ be an $T$-invariant algebraic
  subvariety of dimension $d$ in the vector space $W$. Then the
  equivariant Poincar\'e dual of $\Sigma$ is the polynomial on $\mathfrak{t}$
  defined by the integral
\begin{equation}
 \label{vergneepd}
 \epd\Sigma = \frac1{(2\pi)^d}\int_\Sigma\Thom_{\kt}(W).
\end{equation}  
\end{definition}
\begin{remark}
  \begin{enumerate}
  \item An immediate consequence of the definition is that for an equivariantly
closed differential form $\mu$ with compact support, we have
\[  \int_\Sigma\mu = \int_W \epd\Sigma\cdot\mu.
\]
This formula serves as the motivation for the term {\em equivariant
  Poincar\'e dual.}
\item This definition naturally extends to the case of an analytic
  subvariety of $\CC^n$  defined in the neighborhood of the origin, or
  more generally, to any $T$-invariant cycle in $\CC^n$.
  \end{enumerate}
\end{remark}

Another terminology for the equivariant Poincar\'e dual is {\em multidegree}, which is close in spirit to the original
construction of Joseph \cite{joseph}. Let  $\Sigma \subset W$ be a $T$-invariant
subvariety. Then we have
\[       \epd{\Sigma,W}_T=\mdeg{I(\Sigma),\CC[y_1,\ldots , y_N]}.
\] 

Some basic properties of the equivariant Poincar\'e dual are listed in \cite{bsz}, these are: Positivity, Additivity, Deformation invariance, Symmetry and a formula for complete intersections. Using these properties one can easily describe an algorithm for
computing $\mdeg{I,S}$ as follows (see Xiller--Sturmfels \cite[\S8.5]{milsturm}, Vergne \cite{voj} and \cite{bsz} for details). %

\subsection{The Rossman formula} \label{subsec:rossman} 

The Rossmann equivariant localisation formula is a variant of the Atiyah-Bott/Berline-Vergne localisation for singular varieties sitting in a smooth ambient space. 
Let $Z$ be a complex manifold with a holomorphic $T$-action, and let
$X\subset Z$ be a $T$-invariant analytic subvariety with an isolated
fixed point $p\in X^T$. Then one can find local analytic coordinates
near $p$, in which the action is linear and diagonal. Using these
coordinates, one can identify a neighborhood of the origin in $\TT_pZ$
with a neighborhood of $p$ in $Z$. We denote by $\tc_pX$ the part of
$\TT_pZ$ which corresponds to $X$ under this identification;
informally, we will call $\tc_pX$ the $T$-invariant {\em tangent cone}
of $X$ at $p$. This tangent cone is not quite canonical: it depends on
the choice of coordinates; the equivariant dual of
$\Sigma=\tc_pX$ in $W=\TT_pZ$, however, does not. Rossmann named this
 the {\em equivariant multiplicity of $X$ in $Z$ at $p$}:
\begin{equation}\label{emult}
   \emu_p[X,Z] \overset{\mathrm{def}}= \epd{\tc_pX,\TT_pZ}.
\end{equation}

\begin{remark}
In the algebraic framework one might need to pass to the {\em
tangent scheme} of $X$ at $p$ (cf. Fulton \cite{fulton}). This is canonically
defined, but we will not use this notion.
\end{remark}
The analog of the Atiyah-Bott formula for singular subvarieties of smooth ambient manifolds is the following statement.
\begin{proposition}[Rossmann's localisation formula \cite{rossmann}]\label{rossman} Let $\mu \in H_T^*(Z)$ be an equivariant class represented by a holomorphic equivariant map $\mathfrak{t} \to\Omega^\bullet(Z)$. Then 
\begin{equation}
  \label{rossform}
  \int_X\mu=\sum_{p\in X^T}\frac{\emu_p[X,Z]}{\mathrm{Euler}^T(\TT_pZ)}\cdot\mu^{[0]}(p),
\end{equation}
where $\mu^{[0]}(p)$ is the differential-form-degree-zero component
of $\mu$ evaluated at $p$.  
\end{proposition}

\section{Proof of the main theorem}\label{sec:prooftauintegrals}

In this section we prove Theorem \ref{mainthm} and it equivariant counterpart, Theorem \ref{equivariantintegral}.

\noindent \textbf{Step 1: The stability trick} Assume $X$ is a complex manifold, and in fact, we can assume that $X=\CC^n$, due to the reduction to equivariant integrals below. Let $f: X\to X$ be a stable map, see \S\ref{sec:stablemaps} for definition. Stable maps are dense in the space of holomorphic maps, and we will pick one in the homotopy class of the identity if that exists, otherwise we approximate the identity map with stable maps. The map $f$ induces the $k$th Hilbert extension map
\[\mathrm{hf}^{[k]}: \GHilb^{k}(X) \to \GHilb^{k}(X \times X)\]
which sends the subscheme $\xi_I$ to $\xi_{(I,I_{\Gamma(f)})}$ where $I_{\Gamma(f)}$ is the ideal of the graph $\Gamma(f)$ of $f$. $\mathrm{hf}^{[k]}$ is a regular embedding, let 
$\GHilb^k(\Gamma(f))=\im(\mathrm{hf}^{[k]})$ denote its image, which is the Hilbert scheme of the graph. Let $\pi^{[k]}: \GHilb^k(\Gamma(f)) \to \GHilb^k(X), \xi_{(I,I_{\Gamma(f)})} \mapsto \xi_I$ be the inverse of $\mathrm{hf}^{[k]}$.
The normal bundle of $\GHilb^k(\Gamma(f))$ in $\GHilb^k(X \times X)$ is $(\pi^{[k]})^*(f^*TX)^{[k]}$, whereas the normal bundle of $\GHilb^k(X)$ in $\GHilb^k(X \times X)$ is $TX^{[k]}$, and $\pi^{[k]}$ extends to an isomorphism of a small neighborhood $\pi^{[k]}_\nabla: \GHilb^k_{\nabla}(\Gamma(f)) \to \GHilb^k_\nabla(X)$ with Thom classes $\Thom(\Gamma(f))$ and $\Thom(X)$ respectively, such that $(\pi^{[k]}_\nabla)^*(\Thom(X))=\Thom(\Gamma(f))$. Then for a bundle $V$ over $X$ and $\Phi(V^{[k]})$ a Chern polynomial of the tautological bundle we have 
\begin{equation}
\int_{\GHilb^k(X)}\Phi(V^{[k]})=\int_{\GHilb^k_\nabla(X)}\Phi(V^{[k]}) \cdot \Thom(X)=
\int_{\GHilb^k_\nabla(X)}\Phi(V^{[k]}) \cdot \pi^{[k]}_{\nabla *}\Thom(\Gamma(f))
\end{equation}
where 
\begin{equation}\label{thom}
\Thom(\Gamma(f))|_{\GHilb^k(\Gamma(f))}=\pi^{[k]*}\Euler((f^*TX)^{[k]})
\end{equation}
Recall the $f$-Hilbert scheme 
\[\GHilb^k(f)=\overline{\{\xi=\xi_1 \sqcup \ldots \sqcup \xi_k \in \GHilb^k(X): f(\xi_1)=\ldots =f(\xi_s) \in X\}}\]
as the set of subschemes supported on the fibers of $f$. 
\begin{proposition}[\cite{bsz2} Corollary 12.14] Let $f:X \to X$ be a stable Thom Boardman map. Then there is an embedding $\varphi_f: TX \hookrightarrow (f^*TX)^{[k]}$ such that the $k-1$-jet of $f$ induces a section $s_f$ of $f^*TX^{[k]}/\varphi_f(TX)$ presenting $\GHilb^k(f)$ as a local complete intersection.   
\[[\GHilb^k(f)]=\mathrm{Euler}(f^*TX^{[k]}/\varphi_f(TX))\]
\end{proposition}
This implies that the support of the Euler class sits in the $f$-Hilbert scheme:
\[\mathrm{supp}(\Euler(f^*TX)^{[k]}) \subset \GHilb^k(f).\]
and hence 
\begin{corollary}\label{cor:mostimportant} The class $\Phi(V^{[k]}) \cdot \Thom(X)$ can be represented by a form $\omega_{V,f}$ whose support satisfies
\[\mathrm{supp}(\omega_{V,f}|_{\GHilb^k(X)}) \subset \GHilb^k(f).\]
\end{corollary}
We define via Thom isomorphism the class 
\[\Phi_f(V^{[k]})=\frac{\omega_{V,f}|_{\GHilb^k(X)}}{\Euler(TX^{[k]})} \in \Omega^\bullet(\GHilb^k(X)),\]
such that 
\begin{equation}\label{stabilitytrick}
\int_{\GHilb^k(X)}\Phi(V^{[k]})=\int_{\GHilb^k(X)} \Phi_f(V^{[k]})
\end{equation}
In short, any stable Thom-Boarman map $f$ sufficiently close to the identity defines a deformation $\Phi_f(V^{[k]})$ of $\Phi(V^{[k]})$ whose support has better geometric properties. Indeed, in Proposition \ref{prop:hilbfloc} we proved that for stable Thom-Boardman map $f$ 
\[\GHilb^k(f) \cap \GHilb^k_0(X) \subset \CHilb^k(X),\]
that is, the punctual $k$-fold locus of a stable map sits in the curvilinear component, hence 
\begin{equation}\label{support1}
\mathrm{supp}(\Phi_f)\cap \GHilb^k_0(X) \subseteq \mathrm{supp}(\Euler((f^*TX)^{[k]}))\cap \GHilb^k_0(X) \subset \CHilb^k(X)
\end{equation}
Moreover, by Proposition \ref{prop:hilbfloc},  $\GHilb^k(f)$ is locally irreducible at the curvilinear points, and locally isomorphic to the Haiman bundle $B$.

\subsection{Sieve on fully nested Hilbert schemes}\label{subsec:sieve}
Let $V$ be a rank $r$ vector bundle over the complex manifold $X$ of dimension $m$ with Chern roots $\theta_1,\ldots, \theta_r$, and $\Phi(c_1,\ldots ,c_{kr})$ be a Chern polynomial in the Chern roots of the tautological bundle $V^{[k]}$ over $\Hilb^{k}(X)$. The branched cover $\kappa: \GHilb^{[k]}(X) \to \GHilb^{k}(X)$ gives 
\begin{equation}\label{n!}
n! \int_{\GHilb^{k}(X)}\Phi = \int_{\GHilb^{[k]}(X)} \kappa^*\Phi
\end{equation}
one can work over ordered Hilbert schemes, and we keep the notation $\Phi$ for the pulled-back form. 

The first step in our strategy is to pull-pack integration to the fully nested Hilbert scheme $N^{k}(X)$, which admits plenty of approximating bundles (coming from the different factors) to build linear combinations with punctual supports. A key observation of \cite{junli,rennemo} is that these tautological bundles can be combined into a sieve formula,  which decomposes $\Phi$ as a sum 
\[\Phi=\sum_{\a \in \Pi(k)}\Phi^\a\]
of forms indexed by partitions of $\{1, \ldots k\}$. For any partition $\{1,\ldots, k\}=\a_1 \sqcup \ldots \sqcup \a_s$ the form $\Phi^\a$ is supported on the approximating punctual subset  
\[\supp(\Phi^\a)=N^{\a}_0(X)=\HC^{-1}(\Delta_{\a})\]
and $\Phi^\a$ is a linear combination of the classes $\Phi(V^{[\b]})$ for those partitions $\b \in \Pi(k+1)$ which are refinements of $\a$, that is, $\b \le \a$ holds.  

\begin{definition} Let $\a \in \Pi(k)$ be a partition and $\Phi$ a homogeneous symmetric polynomial in the Chern roots of $V^{[k]}$. Define the class $\Phi^\a \in H^*(N^{k}(X))$ inductively by putting $\Phi^{[1],\ldots [k]}=\Phi(V^{[1],\ldots [k]})=\Phi(V\oplus \ldots \oplus V)$ and for $\a>([1],\ldots [k])$  
\[\Phi^\a=\Phi(V^{[\a]})-\sum_{\b < \a}\Phi^\b.\]
\end{definition}

\begin{remark}\label{philambda}
Let $\Lambda=[1,\ldots, k]$ be the trivial partition. The formula above gives us 
\[\Phi^{\Lambda}(V^{[k]})=\sum_{\b \in \Pi(k)}(-1)^{|\b|-1}(|\b|-1)!\Phi(V^\b).\]
\end{remark}

\begin{proposition}\label{thm:support} The restriction of $\Phi^\a$ to $\N^{k}(X) \setminus \N^\a(X)$ vanishes, that is, the support of $\Phi^\a$ is $\N^\a(X)=\HC^{-1}(\Delta_{\a})\subset \N^{k}(X)$.
In particular, the support of $\Phi^\Lambda$ is $\N^{k+1}_0(X)=\pi_\Lambda^{-1}(\GHilb^{k}_0(X))$, the punctual part of $\N^{k}(X)$ where all components are punctual subschemes supported at some point of $X$.  
\end{proposition}

\begin{proof} This follows from easy inclusion-exclusion and induction, for the proof see \cite{rennemo}.
\end{proof}

\begin{definition}\label{geocond} We say that the form $\Phi \in \Omega^\bullet(\GHilb^k(X))$ is properly supported if $\supp(\Phi)$  intersects the punctual Hilbert scheme only in the curvilinear component:
\begin{equation}\label{supportcond}
\mathrm{supp}(\Phi) \cap \GHilb^k_0(X) \subset \CHilb^k_0(X).
\end{equation}
\end{definition}

\begin{lemma}\label{lemma:support} Let  $\Phi$ be a properly supported Chern polynomial in the sense of Definition \ref{geocond}. Then 
\begin{enumerate}
\item  $\pi_\Lambda^*\Phi$ is also properly supported, that is, it is represented by a form whose support intersects the punctual part $\N^{k}_0(X)$ only in the curvilinear part $\CN^{k}(X)=\pi_\L^{-1}(\CHilb^{k}(X))$. 
\item More generally, for $\a=(\a_1,\ldots, \a_s) \in \Pi(k)$ the form $\Phi^\a$ is represented by a form whose support intersects the punctual part $\N^{k}_0(X)$ only in the curvilinear part
\[\CN^\a(X)=\pi_\a^{-1}(\CHilb^{[\a_1]}(X)\times \ldots \times \CHilb^{[\a_s]}(X))\]
\end{enumerate}
\end{lemma}

\begin{proof}
If $\omega$ is a properly supported form representing $\Phi$, then $\pi^*\omega$ represents $\pi_\Lambda^*\Phi$, and the statement follows from the fact that support of the pull-back form under a proper map is equal to the pre-image of the support.  
\end{proof}

The first step in proving Theorem \ref{mainthm} is to pull back the integral over the fully nested Hilbert scheme, and apply the sieve formula:
\begin{equation}\label{motivic1}
\int_{\GHilb^{[k]}(X)}\Phi=\int_{\N^{k}(X)}\pi^* \Phi=\sum_{\a \in \Pi(k)} \int_{\N^{k}(X)} \Phi^\a
\end{equation}

\subsection{Reduction to equivariant integration on $X=\CC^n$}\label{subsec:reduction}

The sieve formula \eqref{motivic1} reduces integration of forms supported on $\GHilb^{k+1}(X)$ to small neighborhoods of the punctual Hilbert scheme sitting over various diagonals under the Hilbert-Chow morphism.    
Take a small neighborhood $\widehat{\N}^{k}(\bU)$ of $\widehat{\CN}^{k}(X)$ as in diagram \eqref{diagramfive}, which fibers over $\flag_{k-1}(TX)$ and hence fibers over $X$. This bundle can be pulled back along a classifying map $\tau: X \to \mathrm{BGL}(m)$ from the universal bundle 
\[\mathbf{E}=\mathrm{BGL}(m) \times_{\GL(n)} \widehat{\BN}^{k}(\CC^n)\]
where with the notations of diagram \eqref{diagramfive} 
\[\widehat{\BN}^{k}(\CC^n)=(\mu \circ \rho \circ \pi_\Lambda)^{-1}(0)=\pi_\Lambda^{-1}(\widehat{\BHilb}^{k}(\CC^n))\]
is the fiber over the origin $0\in X=\CC^n$, which is the balanced fully nested Hilbert scheme supported at the origin of $\CC^n$. We get a commutative diagram 
\begin{equation*}
\xymatrix{\widehat{\CN}^{k}(X) \ar[rd]^\pi \ar@{^{(}->}[r] & \widehat{\BN}^{k}(\bU) \ar[r] \ar[d]^{\pi} & \mathbf{E}  \ar[d] \\
 & X \ar[r]^{\tau}  &  \mathrm{BGL}(m) } 
\end{equation*}
which induces a diagram of cohomology maps
\begin{equation*}
\xymatrix{H^*(\widehat{\BN}^{k}(\bU))  \ar[d]^{\pi_*} &  H^*(\mathbf{E}) \ar[l] \ar[d]^{\res} \\
 H^*(X)   &  H^*(B\GL(n)) \ar[l]^{\mathrm{Sub}}} 
\end{equation*}
Here 
\begin{itemize}
\item $\res$ is the equivariant push-forward (integration) map along the fiber $\BN^{k}(\CC^n)$, and in the next section we develop an iterated residue formula derived from equivariant localisation.
\item $\mathrm{Sub}$ is the Chern-Weil map, which is the substitution of the Chern roots of $X$ into the generators $\lambda_1,\ldots, \lambda_m$ of $H^*(\mathrm{BGL}(m))=H_{\GL(n)}^*(pt)=\QQ[\l_1,\ldots, \l_m]$,
\end{itemize}

Commutativity of the diagram tells us that for any form $\omega$ supported on some neighborhood $\widehat{\BN}^{k}(\bU)$ we have  
\[\int_{\N^{k}(X)}\omega =\int_{X} \int_{\widehat{\BN}^{k}(\CC^n)} \omega |_{\{\l_1,\ldots, \l_m\} \to \text{Chern roots of } TX}\]
We apply this formula for the form $\Phi^\a$ coming from the sieve. According to Proposition \ref{thm:support} and Lemma \ref{lemma:support} for $\a=(\a_1,\ldots, \a_s) \in \Pi(k)$ the form $\Phi^\a$ is compactly supported in a neighborhood of 
\[\widehat{\CN}^\a(X)=\pi_\a^{-1}(\CHilb^{[\a_1]}(X)\times \ldots \times \CHilb^{[\a_s]}(X)) \subset \widehat{\BN}^{k}(\bU)\]
and hence
\begin{equation}\label{reduceintegraltoaffinespace}
\int_{\GHilb^{[k]}(X)}\Phi=\int_X \sum_{\a \in \Pi(k)} \int_{\widehat{\BN}^{\a}(\CC^n)} \Phi^\a |_{\{\l_1,\ldots, \l_m\} \to \text{Chern roots of TX}}
\end{equation}
where $\widehat{\BN}^\a(\CC^n)=\widehat{\BN}^{\a_1}(\CC^n) \times \ldots \times \widehat{\BN}^{\a_s}(\CC^n)$. Hence it is enough to study the $\a=\Lambda=[1,\ldots, k]$ trivial case when 
\[\widehat{\BN}^{\Lambda}(\CC^n)=\widehat{\BN}^k(\CC^n) \text{ and } \widehat{\CN}^{\Lambda}(\CC^n)=\widehat{\CN}^k(\CC^n).\] 
In the next section, using equivariant localisation we explain how to reduce the equivariant integration of a form which is supported on a small neighborhood of $\widehat{\CN}^k(\CC^n)$ to an integral over $\widehat{\CN}^k(\CC^n)$ itself. 

\subsection{Localisation over the flag} 
Let $V$ be a rank $r$ T-equivariant bundle over $\CC^n$ with T-equivariant Chern roots $\theta_1,\ldots, \theta_r$.  
Let $\l_1,\ldots, \l_n\in \mathfrak{t}^*$ denote the torus weights for the diagonal $T \subset \GL(n)$ action on $\CC^n$ with respect to the basis $e_1,\ldots, e_n \in \CC^n$ and let 
\[\ff=(\Span(e_1) \subset \Span(e_1,e_2)\subset \ldots \subset \Span(e_1,\ldots,e_{k-1}) \subset \CC^n)\]
denote the standard flag in $\CC^n$ fixed by the parabolic $P_{n,k} \subset \GL(n)$.

The Atiyah-Bott-Berline-Vergne  localisation formula of Proposition \ref{abbv} over the flag $\flag_{k-1}(\CC^n)$ gives  
\begin{equation} \label{flagloc}
\int_{\widehat{\BN}^k(\CC^n)} \Phi^\Lambda= \sum_{\sigma\in S_n/S_{n-k+1}}
\frac{\Phi^\Lambda_{\sigma(\ff)}}{\prod_{1\leq j \leq
k-1}\prod_{i=j+1}^n(\lambda_{\sigma\cdot
    i}-\lambda_{\sigma\cdot j})},
\end{equation}
where 
\begin{itemize}
\item $\sigma$ runs over the ordered $k-1$-element subsets of $\{1,\ldots, n\}$ labeling the torus-fixed flags $\sigma(\ff)=(\Span(e_{\sigma(1)}) \subset \ldots \subset \Span(e_{\sigma(1)},\ldots, e_{\sigma(k-1)}) \subset \CC^n)$ in $\CC^n$.
\item $\prod_{1\leq j \leq k-1}\prod_{i=j+1}^n(\lambda_{\sigma(i)}-\lambda_{\sigma(j)})$ is the equivariant Euler class of the tangent space of $\flag_{k-1}(\CC^n)$ at $\s(\ff)$.
\item Let $\widehat{\BN}_{\sigma(\ff)}=\rho^{-1}(\sigma(\ff)) \subset \widehat{\BN}^{k}(\CC^n)$ denote the fiber over $\bff$. Then $\Phi^\Lambda_{\sigma(\ff)}=(\int_{\widehat{\BN}_{\sigma(\ff)}} \Phi^\Lambda)^{[0]}(\sigma(\ff))\in S^\bullet \mathfrak{t}^*$ is the differential-form-degree-zero part evaluated at $\sigma(\ff)$.
\end{itemize}
The Chern roots of the tautological bundle over $\flag_{k-1}(\CC^n)$ at the fixed point $\sigma(\ff)$ are represented by $\l_{\s(1)}, \ldots ,\l_{\s(k-1)}\in \mathfrak{t}^*$ and therefore 
\begin{equation}\label{alphasigmaf}
\Phi^{\Lambda}_{\s(\ff)}=\sigma \cdot \Phi^{\Lambda}_\ff=\Phi^{\Lambda}_\ff(\l_{\s(1)}, \ldots ,\l_{\s(k-1)})\in S^\bullet \mathfrak{t}^*,
\end{equation}
is the $\sigma$-shift of the polynomial $\Phi^{\Lambda}_{\ff}=(\int_{\BN_\ff} \Phi^{\Lambda})^{[0]}(\ff)\in S^\bullet \mathfrak{t}^*$ corresponding to the distinguished fixed flag $\ff$. By \eqref{taurestrictedtocurvi} the restriction of $V^{[k]}$ to the curvilinear part $\widehat{\CN}^{k}(\CC^n)$ is the tensor product
\[V^{[k]}|_{\widehat{\CN}^{k}(\CC^n)}=V \otimes  \calo_{\CC^n}^{[k]}\]
The test curve model formulated in Theorem \ref{bszmodel} (3) then tells that  
\[V^{[k]}=V \otimes \calo_{\CC^n}^{[k]}=V \otimes \cale\]
where $\cale$ is the tautological bundle over $\grass_{k}(S^\bullet \CC^n)$. Hence the Chern roots of $V^{[k]}$ on $\widehat{\CN}^{k}(\CC^n)$, and hence at the torus fixed points, are  the pairwise sums formed from Chern roots of $V$ and Chern roots of $\cale$. This means that  
\begin{equation}\label{cdff}
\Phi^\Lambda(V^{[k]})=\Phi^\Lambda(\theta_j,\l_i+\theta_j: 1\le i \le n, 1\le j \le r)
\end{equation}
is a polynomial in the tautological bundle.

\subsection{Transforming the localisation formula into iterated residue}\label{subsec:transform}
We transform the right hand side of \eqref{flagloc} into an iterated residue motivated by B\'erczi--Szenes \cite{bsz}. This step turns out to be crucial in handling the combinatorial complexity of the fixed point data in the Atiyah-Bott localisation formula and condense the symmetry of this fixed point data in an efficient way which enables us to prove the residue vanishing theorem.

To describe this formula, we will need the notion of an {\em iterated
  residue} (cf. e.g. \cite{szenes}) at infinity.  Let
$\omega_1,\dots,\omega_N$ be affine linear forms on $\CC^{d}$; denoting
the coordinates by $z_1,\ldots, z_{d}$, this means that we can write
$\omega_i=a_i^0+a_i^1z_1+\ldots + a_i^{d}z_{d}$. We will use the shorthand
$h(\bz)$ for a function $h(z_1\ldots z_d)$, and $\dbz$ for the
holomorphic $n$-form $dz_1\wedge\dots\wedge dz_d$. Now, let $h(\bz)$
be an entire function, and define the {\em iterated residue at infinity}
as follows:
\begin{equation}
  \label{defresinf}
 \iresd \frac{h(\bz)\,\dbz}{\prod_{i=1}^N\omega_i}
  \overset{\mathrm{def}}=\left(\frac1{2\pi i}\right)^d
\int_{|z_1|=R_1}\ldots
\int_{|z_d|=R_d}\frac{h(\bz)\,\dbz}{\prod_{i=1}^N\omega_i},
 \end{equation}
 where $1\ll R_1 \ll \ldots \ll R_d$. The torus $\{|z_m|=R_m;\;m=1, \ldots
 , d\}$ is oriented in such a way that $\res_{z_1=\infty}\ldots
 \res_{z_d=\infty}\dbz/(z_1\cdots z_d)=(-1)^d$.
We will also use the following simplified notation: $\sires \overset{\mathrm{def}}=\iresd.$

In practice, one way to compute the iterated residue \eqref{defresinf} is the following algorithm: for each $i$, use the expansion
 \begin{equation}
   \label{omegaexp}
 \frac1{\omega_i}=\sum_{j=0}^\infty(-1)^j\frac{(a^{0}_i+a^1_iz_1+\ldots
   +a_{i}^{q(i)-1}z_{q(i)-1})^j}{(a_i^{q(i)}z_{q(i)})^{j+1}},
   \end{equation}
   where $q(i)$ is the largest value of $j$ for which $a_i^j\neq0$,
   then multiply the product of these expressions with $(-1)^dh(z_1\ldots
   z_d)$, and then take the coefficient of $z_1^{-1} \ldots z_d^{-1}$
   in the resulting Laurent series.

\begin{proposition}[{\rm B\'erczi--Szenes \cite{bsz}, Proposition 5.4}]\label{ABtoresidue}  For any homogeneous polynomial $Q(\bz)$ on $\CC^d$ we have
\begin{equation}\label{flagres}
\sum_{\sigma\in S_m/S_{m-d}}
\frac{Q(\lambda_{\sigma(1)},\ldots ,\lambda_{\sigma(d)})}
{\prod_{1\leq j\leq d}\prod_{i=j+1}^m(\lambda_{\sigma\cdot
    i}-\lambda_{\sigma\cdot j})}=\sires
\frac{\prod_{1\leq i<j\leq d}(z_j-z_i)\,Q(\bz)\dbz}
{\prod_{i=1}^d\prod_{j=1}^m(\lambda_j-z_i)}.
\end{equation}
\end{proposition}

\begin{remark}
  Changing the order of the variables in iterated residues, usually,
  changes the result. In this case, however, because all the poles are
  normal crossing, formula \eqref{flagres} remains true no matter in
  what order we take the iterated residues.
\end{remark}

Proposition \ref{ABtoresidue} for $d=k-1$, together with \eqref{flagloc},\eqref{alphasigmaf} and \eqref{cdff} gives
\begin{corollary}\label{propflag} Let $k\le n$. Then  
\begin{equation*}
\int_{\widehat{\BN}^{k}(\CC^n)} \Phi^\Lambda=\sires
\frac{\prod_{1\leq i<j\leq k-1}(z_i-z_j)\, \Phi^\Lambda_\ff(z_1, \ldots ,z_{k-1})\dbz}
{\prod_{l=1}^{k-1}\prod_{i=1}^n(\lambda_i-z_l)}
\end{equation*}
where 
\begin{itemize}
\item $\Phi_\ff^\Lambda=\int_{\widehat{\BN}_{\ff}} \Phi^\Lambda$ is the integral over the 
fiber, which we will calculate using a second equivariant localisation.
\item $z_1,\ldots, z_{k-1}$ are the $T$-weights of the tautological bundle over $\flag_{k-1}(\symdot)$. Equivalently, the $\GL(n)$ action on $\Hom^\ff(\CC^{k-1},\CC^n)$ reduces to $\GL(\CC_{[k-1]}) \subset \GL(n)$, and $z_1,\ldots z_{k-1}$ are the weights of the $T^{k-1}_\bz \subset \GL(\CC_{[k-1]})$ action. 
\end{itemize}
\end{corollary} 
\subsection{Second equivariant localisation over the base flag}\label{subsec:secondlocalization} Next we proceed a second equivariant localisation on $\widehat{\BN}_\ff$ to compute $\Phi^\Lambda_\ff(z_1,\ldots ,z_{k-1})$.  
By Proposition \ref{prop:hilbfloc} the torus fixed points sit in the curvilinear locus $\widehat{\CN}_\ff=(\rho \circ \pi_\Lambda)^{-1}(\ff) \subset \widehat{\BN}_\ff$, and this curvilinear fiber sits over 
\[\widehat{\CHilb}_\ff=\rho^{-1}(\ff)\simeq \overline{P_{n,k-1}\cdot p_{n,k-1}} \subset \grass_{k-1}(\sym^{\le {k-1}}\CC_{[k-1]})\]
where  
\[p_{n,k-1}=\Span(e_1, \ldots, \sum_{\tau \in \mathcal{P}(k-1)}e_{\tau}) \in \grass_{k-1}(\sym^{\le k}\CC_{[k-1]}),\] 
and $P_{n,k-1} \subset \GL_n$ is the parabolic subgroup which preserves $\ff$. Equivalently, 
\[\widehat{\CHilb}_\ff=\overline{\phi(\Hom^\ff(\CC^{k-1},\CC^n)}\]
 where 
 \[\Hom^\ff(\CC^{k-1},\CC^n)=\{\psi \in \Hom(\CC^{k-1},\CC^n): \psi(e_i)\subset \CC_{[i]} \text{ for } i=1,\ldots, k-1\}\]
and $\CC_{[i]} \subset \CC^n$ is the subspace spanned by $e_1,\ldots, e_i$. 

The torus fixed points in $\widehat{\CHilb}^{k}_\ff \subset \grass_{k-1}(\sym^{\le k-1}\CC_{[k-1]})$ are subspaces of the form
\[F_\tau=\mathrm{Span}(e_{\tau_1}, e_{\tau_2} , \ldots , e_{\tau_{k-1}})\]
parametrised by $k-1$-tuples $\tau=(\tau_1,\ldots, \tau_{k-1})$ where $\tau_i \neq \tau_j$ for $1\le i <j \le k-1$ and 
\begin{equation}\label{tauconditions}
\tau_i = \{t^i_1,\ldots, t^i_{|\tau_i|}\} \text { such that } 1\le t^i_j \le i \text{ and } \sum_{j=1}^{|\tau_i|} t^i_j \le i \text{ for all } 1\le i \le k-1 
\end{equation}
We call these admissible $k-1$-tuples, and those $\tau$'s which correspond to fixed points in $\widehat{\CHilb}^{k}_\ff$ are called curvilinear admissible subsets, the set of these is $\mathcal{P}(k-1)$. 
 
 A point in $\widehat{\CN}_\ff \subset \prod_{\a \in \Pi(k)} \CHilb^\a(\CC_{[k-1]})$ has the form  $\xi=(\xi_A: A\subset \{1,\ldots, 
 k\})$. If this is torus fixed, then $\phi^\grass(\xi_{\{1\ldots, 
 k\}})=W_\tau \subset \sym^{\le k-1}\CC_{[k-1]}$ is a $k-1$-dimensional subspace for some $\tau=(\tau_1, \ldots, \tau_{k-1})$ and 
$\phi^\grass(\xi_A)=F_{A,\tau} \subset F_\tau$
is a torus-fixed subset of dimension $|A|-1$, hence 
\[F_{A,\tau}=\mathrm{Span}(e_{\tau_j}:j\in \Lambda_A)\]
for some $\Lambda_A \subset \{1,\ldots, k\}$ with $|\Lambda_A|=|A|-1$. The $T_\bz^{k-1}$ weights on the tautological bundle $\cale$ over $\grass_{k-1}(\symdot)$ are $z_1,\ldots, z_{k-1}$, hence for a subset $A \subset \{1,\ldots, k\}$ the equivariant Chern roots (i.e. torus weights) of $F_{A,\tau}$ are $\{z_{\tau_j}: j \in \Lambda_{A}\}$
where $z_{\tau_j}=\sum_{i \in \tau_j} z_i$. A simple but crucial consequence of \eqref{tauconditions} is the following
\begin{lemma}\label{crucial1} \begin{enumerate}
\item If $\tau \neq \{[1],[2],\ldots, [k-1]\}$, then at least one integer $2\le i \le k-1$ does not appear in $\tau$. Hence the $T^{k-1}_\bz$-weight of $F_{A,\tau}$ does not depend on $z_i$ for all $A$. 
\item If $\tau_{k-1} \neq [k-1]$ then $\tau$ does not contain $k-1$. In other words, the torus fixed points in $\widehat{\CHilb}^{k}_\ff$ which contain the weight $z_{k-1}$ are of the form $\tau=(\tau_1,\ldots, \tau_{k-2},[k-1])$.  
\end{enumerate}
\end{lemma}

Let $\mathcal{F}_\ff$ denote the set of torus fixed points on $\widehat{\CN}_\ff$. Diagram \eqref{diagramfive} provides a smooth ambient space 
\[\widehat{\BN}_\ff \subset \widehat{\grass}=\prod\limits_{(\a_1,\ldots, \a_s) \in \Pi(k)} \prod\limits_{i=1}^s \widehat{\grass}_{|\a_i|}(S^\bullet \CC_{[k-1]})\] 
and for $F \in \mathcal{F}_\ff$ let $\emu_F[\widehat{\BN}_\ff,\widehat{\grass}]$ denote the $T$-multidegree at $F$. The Rossman localisation formula with Corollary \ref{propflag} gives  
\begin{equation}\label{intnumberone} 
\int_{\widehat{\BN}^{k}(\CC^n)}\Phi(V^{[k]})=
\sum_{F\in \mathcal{F}_\ff} \sires \frac{\emu_F[\widehat{\BN}_\ff,\widehat{\grass}] \prod_{i<j}(z_i-z_j) \Phi^\Lambda_F(z_1,\ldots ,z_{k-1})}{
\Euler_F(\widehat{\grass})  \prod_{l=1}^{k-1}\prod_{i=1}^n(\lambda_i-z_l)} \,\dbz.
  \end{equation}
Now recall from Remark \ref{philambda} the form of $\Phi^\Lambda$:
\[\Phi^\Lambda(V^{[k]})=\Phi(V^{[k]}) + \sum \limits_{\a \in \Pi(k)\setminus \Lambda} (-1)^\a \Phi(V^{\a}).\]
Let $F=(F_{A,\tau}: A \subset \{1,\ldots, {k-1})$ be a torus fixed point on $\widehat{\CN}_\ff$, where $\tau=(\tau_1,\ldots, \tau_{k-1})$. Next we identify $\Phi_F^\Lambda(\bz)$ in the formula. 
Recall from the introduction that $V(z)$ stands for the bundle $V$ tensored by the line $\mathbb{C}_z$, which is the representation of a torus $T$ with weight $z$. Hence its Chern roots are $z+\theta_1,\ldots ,z+\theta_r$. For $z=z_{\tau_j}$ the Chern roots of $V(z_{\tau_j})$ are $(\sum_{l\in \tau_j}z_l + \theta_i: 1\le i \le r)$ and for $A \subset \{1,\ldots, k-1\}$ we write  
\[V(\bz^{A,\tau})=V \oplus \left(\oplus_{j \in \Lambda_A} V(z_{\tau_j})\right),\]
which has rank $r|A|$.  
Let $\a=(\a_1,\ldots, \a_s) \in \Pi(k)$ be a partition of the $k$ points. Recall from \eqref{taurestrictedtocurvi} that on the curvilinear locus $V^\a$ is the tensor product 
\[V^{\a}=V^{[\a_1]}\oplus \ldots \oplus V^{[\a_s]}=(V \otimes \calo_{\CC^n}^{[\a_1]}) \oplus \ldots \oplus (V \otimes \calo_{\CC^n}^{[\a_s]})\]
and by the test curve model formulated in Theorem \ref{bszmodel} (3) the 
tautological bundle is $\calo_{\CC^n}^{[k+1]}/\calo_\grass=\cale$, where 
$\cale$ is the tautological bundle over $\grass_k(\symdot)$. Hence the Chern roots of $V^{\a}=V^{[\a_1]}\oplus \ldots \oplus V^{[\a_s]}$ at the fixed point $F=(F_{A,\tau}: A\subset \{1,\ldots, k-1\})$ are 
\[\{\theta_\ell, \theta_\ell+z_{\tau_j}: 1\le \ell \le r, 1\le i \le s, j \in \Lambda_{\a_i}\}\]
Hence if we write
\[\Phi(V(\bz^{\a,\tau}))=\Phi(\oplus_{i=1}^s V(\bz^{\a_i,\tau}))\]
then 
\begin{equation}\label{phiviaz}
\Phi_F^\Lambda(\bz)=\Phi(V(\bz^{\Lambda,\tau}))+\sum \limits_{\a \in \Pi(k)\setminus \Lambda} (-1)^\a \Phi(V(\bz^{\a,\tau}))
\end{equation}
and substituting to \eqref{intnumberone} we get
\begin{equation}\label{intnumbertwo} 
\int_{\widehat{\BN}^{k}(\CC^n)}\Phi_f^\Lambda(V^{[k]})
=\sum_{F\in \mathcal{F}_\ff} \sires \frac{\emu_F[\widehat{\BN}_\ff,\widehat{\grass}] \prod\limits_{i<j} (z_i-z_j) \left(\Phi(V(\bz^{\Lambda,\tau}))+\sum \limits_{\a \in \Pi(k)\setminus \Lambda} (-1)^\a \Phi(V(\bz^{\a,\tau}))\right)}{
\Euler_F(\widehat{\grass})  \prod_{l=1}^{k-1}\prod_{i=1}^n(\lambda_i-z_l)} 
  \end{equation}

\subsection{Residue Vanishing theorem on fully nested Hilbert schemes}
The price for using the fully nested Hilbert scheme and the sieve formula to reduce integration to the curvilinear part is that the curvilinear locus $\widehat{\CN_\ff}$ over the flag $\ff$ is not irreducible: besides the main component there are further extra components.  The following fundamental theorem say that the contribution of the fixed points sitting on extra components is zero. 

\begin{theorem}{\textbf{Residue Vanishing Theorem on the fully nested Hilbert scheme}}\label{vanishing1}  Let $\a=(\a_1,\ldots, \a_s)\in \Pi(k)$ be a partition with $s>1$. Then the corresponding sum in \eqref{intnumberthree} vanishes:
\begin{equation}\label{eqn:residue}
\sum_{F\in \mathcal{F}_\ff} \sires \frac{\emu_F[\widehat{\BN}_\ff,\widehat{\grass}] \prod_{i<j}(z_i-z_j) \Phi(V(\bz^{\a,\tau}))\dbz}{
\Euler_F(\widehat{\grass})  \prod_{l=1}^{k-1}\prod_{i=1}^n(\lambda_i-z_l)}=0.
\end{equation}
\end{theorem}

Before we present the proof, we introduce some terminology, used in the argument. First, we recall an algebraic condition for residue vanishing.
 \begin{lemma}[\cite{bsz} Proposition 6.3]
  \label{lemma:vanishprop}
Let $p(\bz)$ and $q(\bz)$ be polynomials in the variables $z_1\ldots 
z_{k-1}$, and assume that $q(\bz)\prod_{i=1}^NL_i$ is a product of linear factors. Then
\[ \ires\frac{p(\bz)\dbz}{q(\bz)} = 0
\]
if for some $l\leq k-1$, the following two conditions hold: 
\begin{enumerate}
\item $\deg(p(\bz);l)+1<\deg(q(\bz);l)=\lead(q(\bz);l)$, and
\item If $a_1z_1+\ldots +a_{k-1}z_{k-1}$ is a linear factor of $q$ then $a_l \neq 0$ implies $a_m \neq 0$ for $m>l$. 
\end{enumerate}
\end{lemma}

Recall diagram \eqref{diagramfive} for $X=\CC^n$ and Definition \ref{def:curvgeometricsubsetQ}: we used the notation $\widehat{\CN}^k(\CC^n)= \pi_{\Lambda}^{-1}(\widehat{\CHilb}^k(\CC^n))=(\mu \circ \rho \circ \pi_\Lambda)^{-1}(0)$ for the fibre over the origin and $\widehat{\CN}_\ff=(\rho \circ \pi_\Lambda)^{-1}(\ff) \subset \widehat{\CN}^k(\CC_{[k-1]})$ the part over the distinguished flag $\ff$.
The projection $\pi_\Lambda$ is given as
\[\pi_\Lambda: \widehat{\CN}_\ff \to \widehat{\CHilb}^k(\CC_{[k-1]})\]
\[(\xi_S: S \subset \{1,\ldots, k\}) \mapsto  \xi_\Lambda=\xi_{\{1,\ldots, k-1\}}.\]
\begin{lemma}\label{lemma:bir} $\pi_{\Lambda}$ is a briational morphism.
\end{lemma}
\begin{proof}
A generic point in $\widehat{\CHilb}^k(\CC_{[k-1]})$ is curvilinear, and for any curvilinear scheme $\xi \in \Curv^k(\CC_{[k-1]})$ there is a unique point in $\N^k(\CC^n))$ over it, namely
\[\pi_\Lambda^{-1}(\xi)=\xi_{\{1,\ldots, |S|\}: S \subset \{1,\ldots, k\}}\]
Indeed, a curvilinear subscheme $\xi \in \Curv^k(\CC^n)$ of length $k$ has a unique subscheme $\xi_l \subset \xi$ of length $l\le k$ for all $1\le l \le k$. 
\end{proof}
The projection $\pi_\Lambda$ extends to a birational morphism from small neighborhood as in diagram \eqref{diagramfive}:
\[\pi_\Lambda: \widehat{\BN}_\ff \to \widehat{\BHilb}^k(\CC_{[k-1]})\]
\[(\xi_S: S \subset \{1,\ldots, k\}) \mapsto  \xi_\Lambda=\xi_{\{1,\ldots, k-1\}}.\]
where $\xi_S \in \Hilb^{S}(\CC_{[k-1]}$ is formed by the points labeled by $S$. Recall that $\pi_\Lambda$ is one of the natural projections from the fully nested Hilbert scheme defined for any $R \subset \{1,\ldots, k\}$ as
\[\pi_{R} : \widehat{\BN}^k_\ff \to \widehat{\GHilb}^R_\ff\]
which sends $(\xi_S: S \subset \{1, \ldots, k\})$ to $\xi_{R}$  This is the extension of \eqref{commdiagram} to a small neighborhood over $\ff$. 

Let $F=(F_{A,\tau}: A \subset \{1,\ldots, k-1\}) \in \widehat{\CN}_\ff$ be a torus fixed point with 
\[F_\Lambda=\pi_\Lambda(F)=\Span(e_{\tau_1},\ldots e_{\tau_{k-1}}) \in \widehat{\CHilb}^k_\ff\]
for an admissible sequence $\tau=(\tau_1,\ldots, \tau_{k-1})$. Let 
\[\ff_{-1}=(\Span(e_1) \subset \ldots \subset \Span(e_1,\ldots, e_{k-2}) \subset \CC^n) \in \flag_{k-2}(\CC^n)\]
denote the smaller flag we get by dropping the $k-1$ dimensional subspace from $\ff$, and 
\[F_{-1}=F \cap \ff_{k-2}=((F_{A,\tau}: A \subset \{1,\ldots, k-2\})\in \CN^{k-1}_{\ff_{k-2}}\]
the fixed point formed by the first $k-1$ points of $F$. The projection 
\[\pi_{k \to k-1}: \widehat{\BN}^{k}_\ff  \to   \widehat{\BN}^{k-1}_{\ff_{-1}}\]
given by $(\xi_S: S \subset \{1, \ldots, k\}) \mapsto (\xi_S: S \subset \{1, \ldots, k\})$ has a restriction 
\[\pi_{k \to k-1}: \widehat{\CN}^{k}_\ff  \to   \widehat{\CN}^{k-1}_{\ff_{-1}}\]
given by $F_{A,\tau}: A \subset \{1,\ldots, k-1\} \mapsto F_{A,\tau}: A \subset \{1,\ldots, k-2\}$, and it sends the torus fixed point $F$ to $F_{-1}$. In short, the torus fixed point determined by $\tau=(\tau_1,\ldots, \tau_{k-1})$ is sent to $\tau_{-1}=(\tau_1,\ldots, \tau_{k-2})$.



Before presenting the proof in full generality, we revisit the $k=3$ case which we worked out in \S \ref{subsec:k=3}. Recall that the punctual part $N^3_0(\CC_{[2]})$ has two components: $\PP^1=CN^3(\CC_{[2]})$ is the curvilinear component which is mapped by $\pi$ dominantly to $\GHilb^3_0(\CC_{[2]})$, and a second component $\PP^1 \times \PP^1$ which sits over the Porteous point $\xi_{por}=(x^2,xy,y^2)$. A point in $\PP^1 \times \PP^1$ is given by a tuple $(\xi_{por},[v_{12}],[v_{13}],[v_{23}])$ with $[v_{12}]=[e_1]$. So in fact, this is determined by the two directions $[v_{13}],[v_{23}]$, indicated in Figure 1 by yellow resp. green arrows. 
Note that here $\xi_{por}=(x^2,xy,y^2)$ stands for the Porteous point and $\xi_{curv}=(y,x^3)$ stands for the curvilinear fixed point in $\Hilb^3(\CC^2)$. Substitute these into the Atiyah-Bott formula:
{\small \[
\int\displaylimits_{\hat{N}^3(\CC^n)}\Phi^+=\res_{z_1,z_2=\infty}\sum\limits_{F \in (\hat{N}^3_\bff)^{T_\bz^2}} \frac{(z_1-z_2)\left(\Phi(V_F^{[1,2,3]})-\Phi(V_F^{[1,2],[3]})-\Phi(V_F^{[1,3],[2]})-\Phi(V_F^{[2,3],[1]})+2\Phi(V_F^{[1],[2],[3]})\right)}{\Euler^z(T_FN^3(\CC^2)) \prod_{i=1}^{n}(z_2-\lambda_i)\prod_{i=1}^{n}(z_1-\lambda_i) }\]}

We now turn to the proof in full generality. 

\begin{proof} We start with the following observation
\begin{lemma}\label{lemma:zk-1doesnotappear} Let $\tau=([1], \tau_2,\ldots, \tau_{k-1})$ and $F=(F_{A,\tau}: A \subset \{1,\ldots, k-1\}) \in \widehat{\CN}_\ff$ a torus fixed point. Then 
\begin{enumerate}
\item If $V(\bz^{\a,\tau})$ does not contain $z_{k-1}$, then residue is zero.
\item If $\tau_{k-1}\neq [k-1]$ then $V(\bz^{\a,\tau})$ does not contain $z_{k-1}$, hence the corresponding residue is zero.
\item If $\a=(\a_1,\ldots, \a_s)$ with $s>1$ and $F\in \widehat{\CN}_\ff^{main}$ then $V(\bz^{\a,\tau})$ does not contain $z_{k-1}$, hence the corresponding residue is zero. 
\end{enumerate}
\end{lemma}
\begin{proof} (1) If $V(\bz^{\a,\tau})$ does not contain $z_{k-1}$, then $z_{k-1}$ does not appear in $\Phi(V(\bz^{\a,\tau}))$, and hence by degree reasons the residue is zero. Indeed, the $z_{k-1}$-degree of the second term in the rational expression is $\deg_{z_{k-1}}\Phi(V(\bz^{\a,\tau}))-m-1$, and the first term does not depend on $z_{k-1}$. If $V(\bz^{\a,\tau})$ does not involve $z_{k-1}$, the total degree is smaller than $-1$, hence the residue is automatically zero. Part (2) is clear. For (3) assume $s>1$ and $F \in \widehat{\CN}^{main}$. Then $F$ is limit of curvilinear subschemes
\[F=(F_A: A\subset \{1,\ldots, k\}) =\lim_{i\to \infty} (F^i_A :A \in \{1,\ldots, k\})\]
then since $F^i_\Lambda \in \Curv^k_\ff$, by Lemma \ref{lemma:bir} $F^i_A=F^i_{1,\ldots, |A|-1} \subset \CC_{[k-2]}$ holds for all $A$, hence $F_A=F_{1,\ldots, |A|-1} \subset \CC_{[k-2]}$ for all $A$. Hence for $s=|\a|>1$ $z_{k-1}$ does not appear in $V(\bz^{\a,\tau})$.
\end{proof}

So it remains to sum over those fixed points $F=(F_{A,\tau}: A \subset \{1,\ldots, k-1\}) \in \widehat{\CN}_\ff$ in \eqref{eqn:residue}, where $\tau_{k-1}=[k-1]$, and hence $\tau=([1], \tau_2,\ldots, \tau_{k-2},[k-1])$. 

Since $\a=(\a_1,\ldots \a_s)$ is a partition of $\{1,\ldots, k\}$, we can assume that $k \in \a_s$. If we had $\a_s=\{k-1\}$ a singleton, then $\xi_{\a_i} \subset \CC_{[k-2]}$ for all $i$ and hence $V(\bz^{\a,\tau})$ does not depend on $z_{k-1}$ again. So we can assume that $|\a_s|\ge 2$.
First, we prove
We show that these fixed points all sit in the same boundary divisor. Define
\[\widehat{\CHilb}^{bound}_\ff=\{p_{k-1} \wedge v_{k-1}: [v_{k-1}] \in \PP[\CC_{[k-1]}], p_{k-1} \in \widehat{\CHilb}_{\ff_{-1}}\} \subset \widehat{\CHilb}_\ff\]
and 
\[\widehat{\CN}_\ff^{bound}=\pi_{\Lambda}^{-1}(\widehat{\CHilb}^{bound}_\ff).\]

\begin{lemma} \label{lemma:normal}
\begin{enumerate} 
\item  $\widehat{\CHilb}^{bound}_\ff$ forms a boundary divisor in $\widehat{\CHilb}_\ff$.
\item $\widehat{\CN}_\ff^{bound}$ is a component of $\widehat{\CN}_\ff$ which has codimension 1 in $\widehat{\BN}_\ff$, and normal bundle $N^\ff$.
\item The normal direction $N^\ff_F=\emu_F[\widehat{\CN}_\ff^{bound},\widehat{\BN}_\ff]$ at a fixed point $F \in \widehat{\CN}_\ff^{bound} \setminus \widehat{\CN}^{main}_\ff$ satisfying sits in $\CC_{[k-1]}$. In other works, the first order deformations of $F$ are not punctual, they are supported in at least two points.
\end{enumerate}
\end{lemma}
Let $F=(F_{A,\tau}: A \subset \{1,\ldots, k-1\}) \in \widehat{\CN}_\ff$ in \eqref{eqn:residue}, where $\tau_{k-1}=[k-1]$. Then $F \in \widehat{\CN}_\ff^{bound}$, but we can we can go further. Let 
\[\widehat{\CHilb}^{e_{k-1}}_\ff=\{p_{k-1} \wedge e_{k-1}: p_{k-1} \in \widehat{\CHilb}_{\ff_{-1}}\} \subset \widehat{\CHilb}_\ff^{bound}\]
and 
\[\widehat{\CN}_\ff^{e_{k-1}}=\pi_{\Lambda}^{-1}(\widehat{\CHilb}^{e_{k-1}}_\ff) \subset \widehat{\CN}_\ff^{bound}\]
and finally 
\[\widehat{\CN}_\ff^{e_{k-1},\perp}=\pi_{\Lambda}^{-1}(\widehat{\CHilb}^{e_{k-1}}_\ff) \cap \pi_{[1k]}^{-1}(e_{k-1}) \subset \widehat{\CN}_\ff^{bound}\] 
Note that $\perp$ in the upper index indicates that a point in $\widehat{\CN}_\ff^{e_{k-1},\perp}$ is a limit of non-reduced subschemes where the last, $k$th point approches the rest from direction $e_{k-1}$. We get a diagram
\begin{equation}\label{diagram:cnperp}
\xymatrix{ \widehat{\CN}_\ff^{e_{k-1},\perp}  \ar[d]^{\pi_{[1k]}} \ar@{^{(}->}[r]^{\iota} & \widehat{\CN}_\ff^{e_{k-1}} \ar[d]^{\pi_{[1k]}} \ar@{^{(}->}[r] & \widehat{\CN}_\ff^{bound} \ar@{^{(}->}[r]^{N^\ff} & \widehat{\BN}_\ff \\ [e_{k-1}] \ar@{^{(}->}[r] & \PP[\CC_{[k-1]}] & & } 
\end{equation}
and the normal bundle of $\iota$ is intuitively given by the direction in which the point labeled by $k$ approaches point labeled by $1$. In what follows, we make this statement more precise and rigorous. In particular, we will see that the the codimension of $\iota$ in the above diagram is $1$, and the normal bundle of $\iota$ can be identified by a $\PP^1$ sitting in $\PP[\CC_{[k-1]}]$.

\begin{lemma} The normal bundle of $\widehat{\CN}_\ff^{e_{k-1},\perp}$ in $\widehat{\CN}_\ff^{bound}$ at any point is 
\[\emu[\widehat{\CN}_\ff^{e_{k-1},\perp},\widehat{\CN}_\ff^{bound}]=\prod_{i=1}^{k-2}(e_{k-1}-e_i).\]
\end{lemma}

\begin{proof}
\end{proof}

\begin{lemma} Let $F=(F_{A,\tau}: A \subset \{1,\ldots, k-1\}) \in \widehat{\CN}_\ff$ with $\tau=([1], \tau_2,\ldots, \tau_{k-2},[k-1])$. Then the corresponding residue is zero unless
$\tau_{k-2}=[k-2]$.
\end{lemma}
\begin{proof} 
If $\tau_{[k-2]} \neq [k-2]$ then $k-2$ does not appear in $\tau$, hence $V(\bz^{\a,\tau})$ does not depend on $z_{k-2}$, hence by Lemma \ref{lemma:vanishprop} there must be a mixed term $az_{k-2}+bz_{k-1}$ with $ab \neq 0$ in the denominator of the residue expression. There are two options: if $F_{-1} \in \widehat{\CN}^{main}_{\ff_{-1}}$ then the first order deformation weights cancel out with factors in $Q_{k-2}$ according to Theorem \ref{vanishing2}. If $F_{-1} \notin \widehat{\CN}^{main}_{\ff_{-1}}$ then the first order deformation sits in $\CC_{[k-2]}$ by Lemma \ref{lemma:normal} and hence the corresponding tangent weights at $F$ which contain $z_{k-2}$ can only be $z_{k-2}$ and $z_{k-2}-z_{k-1}$. The latter cancels out with the same factor of the Vandermonde in the numerator. Hence there are no mixed terms in the denominator, and the residue is zero. 


\end{proof}

Assume $F=(F_{A,\tau}: A \subset \{1,\ldots, k-1\}) \in \widehat{\CN}_\ff$ with $\tau=([1], \tau_2,\ldots, \tau_{k-3}, [k-2], [k-1])$. 
Lemma \ref{lemma:normal} applied over the truncated flag $\ff_{-1}$ gives the  boundary component 
\[\widehat{\CHilb}^{bound}_{\ff_{-1}}=\{p_{k-3} \wedge v_{k-2}: p_{k-3} \in \widehat{\CHilb}_{\ff_{-2}}, v_{k-2} \in \CC_{[k-2]}\} \overset{\mathrm{codim}=1}{\subset} \widehat{\CHilb}_{\ff_{-1}}\]
and the corresponding codimension-1 component 
\[\widehat{\CN}_{\ff_{-1}}^{bound}=\pi_{\Lambda \setminus \{k\}}^{-1}(\widehat{\CHilb}^{bound}_{\ff_{-1}}) \overset{\mathrm{codim}=1}{\subset} \widehat{\BN}_{\ff_{-1}}\]

\begin{lemma}\label{lemma:crucial} Assume $F=(F_{A,\tau}: A \subset \{1,\ldots, k-1\}) \in \widehat{\CN}_\ff$ with $\tau=([1], \tau_2,\ldots, \tau_{k-3}, [k-2], [k-1])$. Following Lemma \ref{lemma:normal}, let  
\[N^{\ff_{-1}}_{F_{-1}}=\emu_{F_{-1}}[\widehat{\CN}_{\ff_{-1}}^{bound},\widehat{\BN}_{\ff_{-1}}] \in \CC_{[k-2]}\] 
denote the normal bundle of the codimension $1$ component $\widehat{\CN}_{\ff_{-1}}^{bound}$ at $F_{-1}$. Then 
\begin{enumerate}
\item $\pi_{[1k]}(\xi_A) \in \PP[e_{k-1},N_{F_{-1}}]$ for any $(\xi_A) \in \widehat{\CN}_{\ff_{-1}}^{bound,e_{k-1}}$. Hence \eqref{diagram:cnperp} has the following refinement:
\begin{equation}\label{diagram:cnperp2}
\xymatrix{ \widehat{\CN}_\ff^{e_{k-1},\perp}  \ar[d]^{\pi_{[1k]}} \ar@{^{(}->}[r]^{\iota} & \widehat{\CN}_\ff^{e_{k-1}} \ar[d]^{\pi_{[1k]}} \ar@{^{(}->}[r]^j & \widehat{\CN}_\ff^{bound} \\ [e_{k-1}] \ar@{^{(}->}[r] & \PP[e_{k-1},N_{F_{-1}}] & } 
\end{equation}
\item $\pi_{[1k]*}N^\ff = \calo_{\PP[e_{k-1},N_{F_{-1}}]}(1)$ in Diagram \eqref{diagram:cnperp2}.
\end{enumerate}
\end{lemma}

\begin{proof}

\end{proof}

We finally show that 

\begin{lemma} $\Phi(V(\bz^{\a,\tau}))$ is constant along the fibers of $\iota$.
\end{lemma} 
\begin{proof} Indeed, if $(\xi_A: A \subset \{1,\ldots, k\}) \in \omega^{-1}(\zeta)$, then $\xi_A = \zeta_A$ for $A \subset \{1,\ldots, k-1\}$ and hence $\xi_{\a_i}=\zeta_{\a_i}$ for $1\le i \le s-1$. Finally, $\xi_{\a_s}=\xi_{\a_s \setminus \{k\}} \wedge e_{k-1}=\zeta_{\a_s \setminus \{k\}} \wedge e_{k-1}$ is constant along the fiber. 
\end{proof}

We are ready for the inductive argument which proves Theorem \ref{vanishing1}.

Integration along the normal bundle of $\iota$ of the equivariant form $\Phi(V(\bz^{\a,\tau}))$ then gives 
\[\int_{\PP[e_{k-1},N_{F_{-1}}]} \frac{\Phi(V(\bz^{\a,\tau}))}{N^\ff} = \frac{\Phi}{z_{k-1}(z_{k-1}-N_{F_{-1}})}+\frac{\Phi}{N_{F_{-1}}(N_{F_{-1}}-z_{k-1})}=\frac{\Phi}{z_{k-1}N_{F_{-1}}}\]

Recall that the form $\Phi_f(V(\bz^{\a,\tau}))$ is compactly supported on an infinitesimal neighborhood of $\widehat{\CN}_\ff$ in $\widehat{\BN}_\ff$, and integration is additive on the components of these. By Lemma \ref{lemma:zk-1doesnotappear} and Lemma \ref{lemma:crucial}   
\begin{equation}\label{formuladecompose1}  
\int_{\widehat{\BN}_\ff} \Phi^\Lambda_f(V^{[k]}) = \int_{\widehat{\CN}^{bound}_\ff} \frac{\Phi_f^\Lambda(V^{[k]})}{N^\ff}=\int_{\widehat{\CN}^{e_{k-1},\perp}_\ff} \int_{\PP[e_{k-1},N_{F_{-1}}]} j_*\frac{\Phi_f^\Lambda(V^{[k]})}{N^\ff}
\end{equation}
Recall that $\Phi^\Lambda(V^{[k]})=\Phi(V^{[k]}) + \sum \limits_{\a \in \Pi(k)\setminus \Lambda} (-1)^\a \Phi(V^{\a})$, and for $\a=(\a_1,\ldots, \a_s)$ with $s>1$ the corresponding term has the residue expression 
\begin{multline}\label{formuladecompose2} 
\sum_{F \in \widehat{\CN}^{e_{k-1},\perp}_{\ff}} \sires \frac{\emu_{F_{-1}}[\widehat{\BN}_{\ff_{-1}},\widehat{\grass_{k-1}}] \prod_{i<j\le k-1}(z_i-z_j)\dbz}{
\Euler_{F_{-1}}(\widehat{\grass_{k-1}})  \prod_{l=1}^{k-1}\prod_{i=1}^n(\lambda_i-z_l)}\cdot \frac{\Phi(V(\bz^{\a,\tau}))}{N_{F_{-1}}z_{k-1}\prod_{i=1}^{k-2}(z_i-z_{k-1})}=\\
\sum_{F_{-1} \in \widehat{\CN}^{bound}_{\ff_{-1}}} \sires \frac{\emu_{F_{-1}}[\widehat{\BN}_{\ff_{-1}},\widehat{\grass_{k-1}}]}{\Euler_{F_{-1}}(\widehat{\grass_{k-1}})N_{F_{-1}}(\bz)} \cdot \frac{\prod_{i<j\le k-2}(z_i-z_j) \Phi(V(\bz^{\a,\tau}))\dbz}{
 z_{k-1} \prod_{l=1}^{k-1}\prod_{i=1}^n(\lambda_i-z_l)}
\end{multline}
The key feature of this formula that it separated the last residue variable $z_{k-1}$ from the rest: the first term is the fixed-point contribution at the fixed point $F_{-1}$ of the formal integral 
\[\int_{\widehat{\BN}_{\ff_{-1}}} \Phi^\Lambda_{\ff_{-1}}= \int_{\widehat{\CN}^{bound}_{\ff_{-1}}} \frac{\Phi_f^\Lambda}{N^{\ff_{-1}}},\]
and this is independent of $z_{k-1}$. Moreover we can now apply an inductive argument and separate the $z_{k-2}$ variable next: 
\[\int_{\widehat{\BN}_\ff} \Phi^\Lambda_f(V^{[k]})=\sum_{F_{-2} \in \widehat{\CN}^{bound}_{\ff_{-2}}} \sires \frac{\emu_{F_{-2}}[\widehat{\BN}_{\ff_{-2}},\widehat{\grass}]}{\Euler_{F_{-2}}(\widehat{\grass})N_{F_{-2}}(\bz)} \cdot \frac{\prod_{i<j\le k-3}(z_i-z_j) \Phi(V(\bz^{\a,\tau}))\dbz}{z_{k-2}z_{k-1}\prod_{l=1}^{k-1}\prod_{i=1}^n(\lambda_i-z_l)} \]

We proceed with the induction and conclude that the residue vanishes unless $\tau=({1},[2],\ldots, [k-1])$, that is, $F \in \widetilde{\CN}^{por}_\ff$ sits over the $k$-Porteous fixed point in the fully nested Hilbert scheme, and the formula reduces to 
\[\sires \frac{\Phi(V(\bz^{\a,\tau}))\dbz}{z_1 \ldots z_{k-1} \prod_{l=1}^{k-1}\prod_{i=1}^n(\lambda_i-z_l)} \]
If $\a=(\a_1,\ldots, \a_s\})$ with $s>1$, then by Lemma \ref{crucial1} (1) $\Phi(V(\bz^{\a,\tau}))$ does not depend on at least one $z_i$ for $1\le i \le k-1$ and hence by Lemma \ref{lemma:vanishprop} the residue vanishes.


\end{proof}

\subsection{Reducing integration from the fully nested to the curvilinear Hilbert scheme}

The residue vanishing theorem reduces the formula \eqref{intnumberone} to 
\begin{equation}\label{intnumberthree} 
\int_{\widehat{\BN}^{k}(\CC^n)}\Phi_f^\Lambda(V^{[k]})
=\sum_{F\in \mathcal{F}_\ff} \sires \frac{\emu_F[\widehat{\BN}_\ff,\widehat{\grass}] \prod\limits_{i<j} (z_i-z_j) \Phi(V(\bz^{\Lambda,\tau}))}{
\Euler_F(\widehat{\grass})  \prod_{l=1}^{k-1}\prod_{i=1}^n(\lambda_i-z_l)} 
  \end{equation}
But the bundle $V^{[k]}=(\pi_\Lambda)^*V^{[k]}$ is pulled-back from a small neighborhood $\CHilb^{k}_\nabla(\CC^n)$ of $\CHilb^{k}(\CC^n)$ in $\GHilb^k(\CC^n$.  Hence the right hand side of \eqref{intnumberthree} is the localisation formula for the integral over the small neighborhood $\CHilb^{k}_\nabla(\CC^n)$, which we formulate in the following corollary.
\begin{corollary}\label{cor:intnumberfour}  
\[
\int_{\widehat{\BN}^{k}(\CC^n)}\Phi_f^\Lambda(V^{[k]})=\int_{\CHilb^{k}_\nabla(\CC^n)} \Phi_f(V^{[k]})
\]
\end{corollary}

Recall that $\Phi_f(V^{[k]}) \in \Omega^\bullet(\GHilb^{k+1}(\CC^n))$ is properly supported form as in Definition \ref{geocond}. This means that $\supp(\Phi_f(V^{[k]})$ is locally irreducible at every point and 
\[\supp(\Phi_f) \cap \GHilb^{k}_0(\CC^n) \subset \CHilb^{k}(\CC^n)\]
Recall also that $\Curv^{k}(\CC^n) \subset \CHilb^{k}(\CC^n)$ is the nonsingular open locus which parametrises curvilinear subschemes. Let $B \to \Curv^{k+1}(\CC^n)$ denote the normal bundle of $\Curv^{k}(\CC^n)$ in $\GHilb^{k}(\CC^n)$.
\begin{proposition}{(Haiman \cite{haiman})} $B=\calo_{\CC^n}^{[k]}/\calo$ is a rank $k-1$ bundle.
\end{proposition}
Note that the tautological bundle $B$ extends over the whole $\GHilb^{k}(\CC^n)$, and in particular, over the closure $\CHilb^{k}(\CC^n)=\overline{\Curv^{k}(\CC^n)}$, and  
\[\supp_\nabla(\Phi_f):=\supp(\Phi_f) \cap \GHilb_\nabla(\CC^n) \subset B|_{\CHilb^{k}(\CC^n)}\]
Hence $\supp(\|phi_f)$ locally modeled over $\supp(\Phi_f)_0$ as a bundle, which means that 
\begin{equation}\label{localmodelgeneral}
\xymatrix{\supp_\nabla(\Phi_f) \ar[d]  \ar[r]^\rho &  B  \ar[d]  \\
  \supp_0(\a)  \ar@{^{(}->}[r] &  \CHilb^{k}(X)}
 \end{equation}
there is a topological isomorphism $\rho$ from $\supp_\nabla(\Phi_f)$ to the total space of $B$. This local model combined with Thom-isomorphism gives 
\begin{equation}
\int_{\CHilb^{k}_\nabla(\CC^n)} \Phi_f(V^{[k]})=\int_{\CHilb^{k}(\CC^n)} \frac{\Phi_f(V^{[k]})}{\Euler(B)}
\end{equation}

\subsection{Residue vanishing theorem on curvilinear Hilbert schemes}
We can follow the same argument we developed over the nested Hilbert scheme $\widehat{\BN}^{k}(\CC^n)$, but now for $\widehat{\CHilb}^{k}(\CC^n)$: we transform the localisation into iterated residue over $\flag_{k-1}(\CC^n)$, then we apply a second localisation over the fiber with the residual $k-1$-dimensional torus with weights $z_1, \ldots, z_{k-1}$. The fiber $\CHilb^{k}_\ff$ over the flag $\ff=(\Span(e_1) \subset \ldots \subset \Span(e_1,\ldots, e_{k-1})=\CC_{[k-1]})$ is singular, but it sits in the smooth ambient space 
\[\flag_\ff=\flag_{k-1}(\Sym^{\le k-1}\CC_{[k-1]})\] 
and hence the Rossmann localisation formula of Proposition \ref{rossman} gives
\begin{equation}\label{intnumbersix} 
\int_{\CHilb^{k}(\CC^n)}\frac{\Phi_f(V^{k})}{\Euler(B)}=
\sum_{F\in \mathcal{P}_\ff} \sires \frac{\emu_F[\CHilb^{k}_\ff,\flag_\ff]
\prod_{i<j\le k-1}(z_i-z_j) \Phi(V^\Lambda(\bz))}{
\Euler_F(B)\cdot  \Euler_F(\flag_\ff) \prod_{l=1}^{k-1}\prod_{i=1}^n(\lambda_i-z_l)} \,\dbz
  \end{equation}
Recall from \eqref{tauconditions} that the fixed points in $\mathcal{P}_\ff$ are parametrised by sequences $\tau=\tau_1-\ldots -\tau_{k-1}$ where 
\begin{equation}
\tau_i \subset \{1,\ldots, i\} \text { such that } \Sigma \tau_i \le i. 
\end{equation}
We call these admissible sequences, and those admissible sequences which correspond to fixed points in $\CHilb^{k}_\ff$ are called curvilinear admissible sequences. 

Note that at the fixed point $F=\tau=(\tau_1,\ldots, \tau_{k-1})$ corresponding to the admissible sequence $(\tau_1,\ldots, \tau_{k-1})$ the Euler class of $B$ is 
\[\Euler_F(B)=z_{\tau_1} \ldots z_{\tau_{k-1}}.\]
The tangent space of $\flag_\ff$ at $\tau=(\tau_1,\ldots, \tau_{k-1})$ is 
\[\Euler_\tau(\flag_\ff)=\prod_{l=1}^k\prod_{\Sigma \tau \leq l}^{\tau\neq\tau_1\ldots \tau_l}
(z_{\tau}-z_{\tau_l})\]
We prove in \cite{bercziG&T} the following vanishing theorem
\begin{theorem}{\textbf{Residue Vanishing Theorem on curvilinear Hilbert schemes (\cite{bercziG&T} Thm. 6.1)}}\label{vanishing2} Let $m \le k$. Then 
\begin{enumerate}
\item All terms but the one corresponding to $\tau_\dist = ([1], [2],\ldots , [k-1])$ vanish in \eqref{intnumbersix}, leaving us with 
\begin{equation}\label{eqn:residue2}
\sires \frac{\emu_{\tau_\dist}[\CHilb^{k}_\ff,\flag_\ff]
\prod_{i<j\le k-1}(z_i-z_j) \Phi(V^{[k]}(\bz))}{
\Pi_{\Sigma \tau \le l \le k-1}(z_\tau-z_l) \cdot  z_1 \ldots z_{k-1} \prod_{l=1}^{k-1}\prod_{i=1}^n(\lambda_i-z_l)} \,\dbz=0
\end{equation}
\item If $|\tau|\ge 3$ then $\emu_{\tau_\dist}[\CHilb^{k}_\ff,\flag_\ff]$ is divisible by $z_\tau-z_l$ for all $l \ge \Sigma \tau$. Let 
\[Q_{k-1}(\bz)=\frac{\emu_{\tau_\dist}[\CHilb^{k}_\ff,\flag_\ff]}{\prod_{\substack{|\tau| \ge 3\\\Sigma \tau \le l \le k-1}}(z_\tau-z_l)}\]
denote the quotient polynomial. Then we get the simplified formula:
\begin{equation} \label{intnumberthreeb}
\int_{\CHilb^{k}(\CC^n)}\frac{\Phi_f(V^{k})}{\Euler(B)}=\sires \frac{Q_{k-1}(\bz)\,\prod_{i<j\le k-1}(z_i-z_j) \Phi(V^{[k]}(\bz))}{
\prod_{i+j \le l \le k-1}
(z_i+z_j-z_l)  \prod_{l=1}^{k-1}\prod_{i=1}^n(\lambda_i-z_l)} \,\dbz
  \end{equation}
\end{enumerate} 
\end{theorem}

\begin{remark}\label{remarkq}
The geometric meaning of $Q_{k-1}(\bz)$ in \eqref{intnumberthreeb} is the following, see also \cite[Theorem 6.16]{bsz}. Let $T_{k-1}\subset B_{k-1}\subset \GL(k-1)$ be the subgroups of invertible
diagonal and upper-triangular matrices, respectively; denote the
diagonal weights of $T_{k-1}$ by $z_1, \ldots, z_{k-1}$.  Consider the $\GL(k-1)$-module of 3-tensors $\Hom(\CC^{k-1},\sym^2\CC^{k-1})$; identifying the
weight-$(z_m+z_r-z_l)$ symbols $q^{mr}_l$ and $q^{rm}_l$, we can
  write a basis for this space as follows:
\[ \Hom(\CC^{k-1},\sym^2\CC^{k-1})=\bigoplus \CC q^{mr}_l,\;  1\leq m,r,l \leq k-1.
\]
Consider the point $\epsilon=\sum_{m=1}^{k-1}\sum_{r=1}^{k-m-1}q_{mr}^{m+r}$
in  the $B_{k-1}$-invariant subspace
\begin{equation*}
  \label{nhmodule}
    W_{k-1} = \bigoplus_{1\leq m+r\leq l\leq k-1} \CC q^{mr}_l\subset
\Hom(\CC^{k-1},\sym^2\CC^{k-1}).
\end{equation*}
Set the notation $\OO_{k-1}$ for the orbit closure
$\overline{B_{k-1}\epsilon}\subset W_{k-1}$, then $Q_{k-1}(\bz)$ is the $T_{k-1}$-equivariant
Poincar\'e dual $Q_{k-1}(\bz) = \epd{\OO_{k-1},W_{k-1}}_{T_{k-1}}$,
which is a homogeneous polynomial of degree
$\dim(W_{k-1})-\dim(\OO_{k-1})$. For small $k-1$ these polynomials are the following (see \cite[Section 7]{bsz}):
\[Q_2=Q_3=1, Q_4=2z_1+z_2-z_4\]
\[Q_5=(2z_1+z_2-z_5)(2z_1^2 +3z_1z_2-2z_1z_5+2z_2z_3-z_2z_4-z_2z_5-z_3z_4+z_4z_5).\] 
\end{remark}

Hence we arrive to the final formula 
\begin{equation}\label{intnumberfour} 
\int_{\CHilb^{k}(\CC^n)}\frac{\Phi_f(V^{[k]})}{\Euler(B)}=
 \sires \frac{
Q_{k-1}(\bz) \prod_{i<j\le k-1}(z_i-z_j) \Phi(\{\theta_\ell,\theta_\ell+z_{\tau_i}\}:1\le \ell \le r, 1\le i\le k-1)}{
z_1\ldots z_{k-1} \prod_{i+j \le l \le k-1}(z_i+z_j-z_l)  \prod_{l=1}^{k-1}\prod_{i=1}^n(\lambda_i-z_l) \,\dbz}
  \end{equation}
This completes the proof of Theorem \ref{mainthm} and Theorem \ref{equivariantintegral}.

\section{Higher dimensional Segre and Chern numbers}

On surfaces, when $X=S$ is two-dimensional, the top Segre classes 
\[s_{2k}({V^{[k]}})=\int_{\Hilb^{k}(S)}s_{2k}({V^{[k]}})\]
of tautological bundles have long been studied. Here $s(F^{[k]})=1/c(F^{[k]})$ is the total Segre class of $F^{[k]}$ on $\Hilb^k(S)$. The generating series 
\[S_V(q)=\sum_{k\ge 0} s_{2k}({V^{[k]}}) q^k\]
was extensively studied for line bundles $V=L \in \mathrm{Pic}(S)$ and the well-known Lehn conjecture \cite{lehn} is a closed formula for $S_L(q)$. It was
first proved in \cite{mop2} in the special case of $K$-trivial surfaces, and then in general in
\cite{voisin,mop1}. This was the starting point of a serious of deep, beautiful structural results for Segre and Chern-integrals over surfaces: the line bundle was replaced by higher rank bundles, and a duality conjecture was formulated between Segre and Verlinde integrals \cite{johnson,mellitgottsche,mop1}. 

Our newly developed integration method allows us to generalise Segre numbers in higher dimensions. Let $X$ be a complex manifold of dimension $n$, $V$ as rank $r$ bundle over $X$. We define top Segre numbers as 
\[s_{nk}({V^{[k]}})=\int_{\GHilb^{k}(X)}s_{nk}({V^{[k]}})\]
With out notation the $\bz$-twisted classes are 
\[s(V^{[k]}(\bz))=s(\theta+\bz,\theta)=\prod_{j=1}^r\frac{1}{1+\theta_j}\cdot \prod_{i=1}^{k-1}\prod_{j=1}^r\frac{1}{1+\theta_j+z_i}=s_V \cdot (z_1\ldots z_{k-1})^{-r}\prod_{i=1}^{k-1}\cals(\frac{1}{z_i})\]
where $s_V$ is the total Segre class of $V$ and 
\[\cals(\frac{1}{z_i})=\prod_{j=1}^r \left(1-\frac{1+\theta_j}{z_i}+\frac{(1+\theta_j)^2}{z_i^2}- \ldots \right)\]
is a polynomial in $1/z_i$ with coefficients polynomials in the Chern-classes of $V$, that is $\cals(x) \in \CC[c_1(V),\ldots, c_r(V)][x]$.

Substituting into Theorem \ref{mainthm} we arrive at the following expression for the deepest term in $s_{nk}({V^{[k]}})$:
\[\sires \frac{\prod_{1\le i<j \le k-1}(z_i-z_j)Q_{k-1}(\bz)s_V d\bz}{\prod_{i+j\le l\le k-1}(z_i+z_j-z_l)(z_1\ldots z_{k-1})^{r+n+1}}\prod_{i=1}^{k-1} \cals \left(\frac{1}{z_i}\right)s_X\left(\frac{1}{z_i}\right),\]
which is a universal symmetric polynomial in the Chern roots of $X$ and $V$. This can be written as a linear form of degree-$mk$ monomials in the Chern polynomials of $X$ and $V$. Since the total Segre class is multiplicative in the sense that for the bundle decomposition $E=E' \oplus E''$  $s(E)=s(E')s(E'')$ holds, the full expression of Theorem \ref{mainthm} for $s_{nk}({V^{[k]}})$ can be reformulated using the generating function
\[S_V(q)=\sum_{k\ge 0} s_{nk}({V^{[k]}}).\]
By a general argument (see e.g. \cite{kazarian, rennemo}) 
\[S_V(q)=\exp \left(\sum_{k=0}^\infty \frac{q^k}{k!} \sires \frac{\prod_{1\le i<j \le k-1}(z_i-z_j)Q_{k-1}(\bz)s_V d\bz}{\prod_{i+j\le l\le k-1}(z_i+z_j-z_l)(z_1\ldots z_{k-1})^{r+n+1}}\prod_{i=1}^{k-1} \cals \left(\frac{1}{z_i}\right)s_X\left(\frac{1}{z_i}\right)\right)\]
The same exponential expansion holds for any multiplicative characteristic class, and in particular Chern series.

\bibliographystyle{abbrv}
\bibliography{thom.bib}
\end{document}